\documentclass[twoside]{article}

\usepackage[preprint]{aistats2027}
\usepackage[round]{natbib}

\usepackage{comment}
\usepackage{color}
\usepackage{xcolor}
\usepackage{amsthm,amsmath,amsfonts,amssymb,euscript,array,enumerate,mathrsfs}
\usepackage{mathtools}
\mathtoolsset{showonlyrefs}
\usepackage{stmaryrd}
\usepackage{algorithm}
\usepackage{algorithmic}
\usepackage{graphicx}
\graphicspath{{AISTATS/figures/}{figures/}}
\usepackage{booktabs}
\usepackage{longtable}
\usepackage{float}
\usepackage{enumitem}
\usepackage[colorlinks=true,
            linkcolor=black,
            citecolor=teal,
            urlcolor=teal]{hyperref}

\def \Ac{{\cal A}}
\def \Bc{{\cal B}}
\def \Cc{{\cal C}}
\def \Dc{{\cal D}}
\def \Ec{{\cal E}}
\def \Fc{{\cal F}}
\def \Gc{{\cal G}}

\def \Kc{{\cal K}}
\def \Lc{{\cal L}}
\def \Pc{{\cal P}}
\def \Rc{{\cal R}}

\def \Nc{{\cal N}}

\def \Sc{{\cal S}}
\def \Tc{{\cal T}}
\def \Uc{{\cal U}}

\def \Wc{{\cal W}}
\def \Xc{{\cal X}}
\def \Yc{{\cal Y}}
\def \Zc{{\cal Z}}

\def\argmin_#1{\underset{#1}{\mathrm{argmin\, }}}
\def\argmax_#1{\underset{#1}{\mathrm{argmax\, }}}

\def \R{\mathbb{R}}

\def \M{\mathbb{M}}
\def \Z{\mathbb{Z}}
\def \E{\mathbb{E}}

\def \P{\mathbb{P}}

\def \S{\mathbb{S}}

\def \d{\mathrm{d}}

\newtheorem{Theorem}{Theorem}[section] 
\newtheorem{Proposition}[Theorem]{Proposition}
\newtheorem{Assumption}[Theorem]{Assumption}
\newtheorem{Lemma}[Theorem]{Lemma}
\newtheorem{Corollary}[Theorem]{Corollary}
\newtheorem{Remark}[Theorem]{Remark}

\begin{document}

%

%

\runningauthor{Jamal, Mekkaoui, Hafsi, Pham}

\onecolumn

\aistatstitle{Randomized Transport Maps for Model-Free
Policy-Gradient Mean-Field Control}
\vspace{-0.5cm}
\begin{center}
\begin{tabular}{@{}c@{\hspace{1em}}c@{}}

\parbox[t]{0.43\textwidth}{
    \centering
    \textbf{Adonis Jamal}$^{*}$\\
    \normalfont ENS Paris-Saclay\\
    \normalfont \texttt{adonis.jamal@ens-paris-saclay.fr}
}
&
\parbox[t]{0.43\textwidth}{
    \centering
    \textbf{Samy Mekkaoui}$^{*}$\\
    \normalfont École Polytechnique\\
    \normalfont \texttt{samy.mekkaoui@polytechnique.edu}
}

\\[3.5em]

\parbox[t]{0.43\textwidth}{
    \centering
    \textbf{Yadh Hafsi}\\
    \normalfont École Polytechnique\\
    \normalfont \texttt{yadh.hafsi@polytechnique.edu}
}
&
\parbox[t]{0.43\textwidth}{
    \centering
    \textbf{Huyên Pham}\\
    \normalfont École Polytechnique\\
    \normalfont \texttt{huyen.pham@polytechnique.edu}
}

\end{tabular}
\end{center}

\vbox{\aistatsaddress{}}

\begingroup
\renewcommand{\thefootnote}{\fnsymbol{footnote}}
\footnotetext[1]{Equal contribution.}
\endgroup

\begin{abstract}
    We develop a model-free policy gradient method for discrete-time mean-field control (MFC). In MFC, the policy affects the objective both through the controlled dynamics and through the population distribution. Standard REINFORCE estimators capture the first effect but not the second. We introduce Transport REINFORCE, a transport map-based approach that perturbs a suitable transformation of the population distribution to estimate this missing mean-field contribution. The method applies to both finite and continuous state spaces. In finite state spaces, we perturb the population distribution directly on the probability simplex through a convex combination of the current population weights and random weights. In continuous state spaces, we  project the population distribution onto the manifold of Gaussian mixtures, and then randomize it via a transport map that ensures the perturbed law remains within this manifold. We prove consistency of the perturbed objective and gradient as the perturbation vanishes, and derive bias and mean-square error bounds for the resulting sample-based gradient estimator. Numerical experiments on several MFC benchmarks show that Transport REINFORCE improves over standard REINFORCE.
\end{abstract}

\section{Introduction}

We study model-free mean-field control (MFC) in discrete time over a finite horizon. A common randomized policy is optimized for a population of interacting agents. Each agent's policy, dynamics, and rewards may depend on the population distribution, which itself evolves under the policy. The learner observes sampled transitions and rewards, without analytic access to the model.

Existing learning approaches include dynamic programming on the space of population distributions \citep{carmona2023model,gu2021mean,gu2023dynamic}. Even for $N$ states, approximating a value function on the resulting $(N-1)$-dimensional simplex can be costly. Direct policy optimization offers another approach. Policy-gradient convergence has been established for discrete-time linear--quadratic MFC \citep{carmona2019linear}, while actor--critic methods and their analysis have been developed in continuous time \citep{frikha2025actor,pham2025actor,frikha2024full}. \citet{bayraktar2026policy} also derive a policy-gradient formula leading to an actor--critic scheme.

The difficulty is that applying standard REINFORCE \citep{williams1992simple} to a representative agent while treating the population flow as fixed misses part of the MFC gradient. Changing the policy also changes the population distribution, which enters the policy, dynamics, and rewards. The closest work, \citet{meunier2026model}, addresses this additional dependence in finite state spaces. Their MF-REINFORCE method randomizes the logit coordinates of the population weights and uses the resulting density score to estimate the mean-field contribution.

We introduce Transport REINFORCE, which randomizes the population law through transport maps. The perturbation combines population and random weights in finite state spaces, and perturbs Gaussian-mixture weights, means, and covariances in continuous spaces. Its explicit density score accounts for population dependence in the perturbed gradient. The transport also couples the original and perturbed laws to control their Wasserstein distance~\citep{villani2009optimal,delon2020wasserstein}.

Our contributions are threefold.
\begin{itemize}[nosep]
    \item We derive model-free policy-gradient estimators combining policy and population scores, using simulations at prescribed population arguments. Population sensitivities are estimated through a probabilistic representation in finite state spaces and centered policy differences in continuous state spaces.

    \item We prove objective perturbation bounds of order $O(\lambda)$ in finite state spaces and $O(\sqrt\lambda)$ for a fixed Gaussian-mixture representation. In the continuous case, we separate the perturbation error from the mixture approximation error. In both cases, we establish gradient consistency and bounds on the estimator's bias and mean-square error.

    \item Our experimental suite compares Transport REINFORCE with REINFORCE on seven discrete and continuous benchmarks, and with MF-REINFORCE on the finite-state benchmarks, under matched simulator budgets. We also develop an explicitly solvable linear-dynamics setting with Gaussian initial conditions.
\end{itemize}

\section{Background}
\label{sec:background}
Let $(\Omega,\Fc,\P)$ be a standard probability space and $T<\infty$. We set $\Tc\coloneqq\{0,\ldots,T\}$. The state space $(\Xc,\Bc(\Xc))$ and action space $(\Ac,\Bc(\Ac))$ are assumed to be Polish. For $t \in \llbracket 0,T-1 \rrbracket $, let
\begin{align}
    P_t : \Xc\times\Pc(\Xc)\times\Ac&\longrightarrow\Pc(\Xc),\\
    r_t : \Xc\times\Pc(\Xc)\times\Ac&\longrightarrow\R,
\end{align}
be respectively measurable transition kernel and running reward, and let $g:\Xc\times\Pc(\Xc)\to\R$ be a measurable terminal reward. A randomized mean-field feedback policy at time $t$ is a measurable transition kernel $\pi_t:\Xc\times\Pc(\Xc)\longrightarrow\Pc(\Ac)$. Rather than optimizing over all such kernels, we choose a parametric family $\{\pi^\theta:\theta\in\Theta\subset\R^{d_\theta}\}$, where $\Theta$ represents the compact admissible set of parameters. Whenever a policy score is used, we assume that $\pi_t^\theta$ admits a strictly positive density $p_t^\theta$ with respect to a fixed reference measure $\nu_A\in\Pc(\Ac)$: $\pi_t^\theta(\d a\mid x,m) = p_t^\theta(a\mid x,m)\nu_A(\d a)$. Given a parametrized randomized mean-field feedback policy $\pi^\theta$, an initial distribution $\mu_0\in\Pc(\Xc)$ and setting $\mu_t^\theta \coloneqq \P_{X_t^\theta}$, we consider the state and control processes, $(X_t^\theta)_{t\in\Tc}$ and $(\alpha_t^\theta)_{t=0}^{T-1}$ respectively, driven by
\begin{align}
    &X_0^{\theta} \sim \mu_0,\qquad
    \alpha_t^{\theta} \sim \pi_t^\theta(\cdot\mid X_t^{\theta},\mu_t^\theta),\qquad
    X_{t+1}^{\theta} \sim P_t(\cdot\mid X_t^{\theta},\mu_t^\theta,\alpha_t^{\theta}),
    \label{eq:mfc-dynamics-background}
\end{align}
The finite-horizon objective we seek to maximize is then defined as
\begin{align}
    \label{eq:background-control-problem}
    J(\theta) \coloneqq \E\bigg[ \sum_{t=0}^{T-1}r_t(X_t^\theta,\mu_t^\theta,\alpha_t^\theta) + g(X_T^\theta,\mu_T^\theta)  \bigg].
\end{align}
Appendix~\ref{appendix:problem-formulation} establishes well-posedness of~\eqref{eq:mfc-dynamics-background}-\eqref{eq:background-control-problem} under the stated assumptions.

Policy-gradient methods maximize objective~\eqref{eq:background-control-problem} by using a probabilistic representation of $\nabla_{\theta} J(\theta)$, an estimator $\widehat G(\theta)$ of it and updating $ \theta\leftarrow\theta+\gamma\widehat G(\theta)$ for $\gamma>0$. To see the difficulty created by the population law, let
\begin{align}
    \Rc_\theta \coloneqq \sum_{t=0}^{T-1}r_t(X_t^\theta,\mu_t^\theta,\alpha_t^\theta) + g(X_T^\theta,\mu_T^\theta),\qquad
    \Sc_{\mathrm{pol}}^\theta &\coloneqq \sum_{t=0}^{T-1}\nabla_\theta\log p_t^\theta( \alpha_t^\theta\mid X_t^\theta,\mu_t^{\theta}).
\end{align}
Throughout, the policy score differentiates $\theta\mapsto \log p_t^\theta(a\mid x,m)$ with $(a,x,m)$ held fixed, and is then evaluated at the sampled action, state and population argument.
The usual REINFORCE term is $\mathrm{RF}(\theta)=\E[\Rc_\theta\Sc_{\mathrm{pol}}^\theta]$. In mean-field control, however, the whole population flow $(\mu_t^\theta)_{t \in \Tc}$ also depends on $\theta$. Consequently, the full gradient takes the form
\begin{align}
    \nabla_\theta J(\theta) = \mathrm{RF}(\theta) + \mathrm{MF}(\theta),
\end{align}
where $\mathrm{MF}$ collects the effects of the population flow on rewards, policies, and state distributions which ordinary REINFORCE misses; see \citet{meunier2026model}.

\section{Transport map randomization}
\label{section:transport-map-randomization}
We randomize the population argument through a transport map. For $\lambda\in[0,1)$ and $t\in\llbracket1,T\rrbracket$, let $T_t^{\lambda}:\Omega\times\Xc\times[0,1]\to\Xc$ be jointly measurable, and let $\Gc:\Pc(\Xc)\to\Pc(\Xc)$ be a measurable transformation of the law. In continuous state space, we restrict its domain and codomain to $\Pc_2(\R^d)$. For any $\theta \in \Theta$, we set $M_0^{\lambda,\theta} \coloneqq \mu_0$, and for $t\in\llbracket1,T\rrbracket$, define the measurable map $M_t^{\lambda,\theta} : \Omega \to \Pc(\Xc)$ as 
\begin{align}
    M_t^{\lambda,\theta}(\omega)&\coloneqq
    T_t^{\lambda}(\omega,\cdot,\cdot)\sharp
\left(\Gc(\mu_t^\theta)\otimes\Uc([0,1])\right).
    \label{eq:general-transport-randomization}
\end{align}
The transformation $\Gc$ supplies a finite-dimensional representation: it is the identity in finite state spaces and an encoder--decoder map $\Dc\circ\Ec$ in continuous spaces.
The auxiliary uniform variable allows the transport to change mixture weights. A deterministic map of the state alone cannot, in general, split the mass of an atom. The uniform variable supplies the randomness needed for this operation.

For each fixed $\omega$, a sample from $M_t^{\lambda,\theta}(\omega)$ is obtained by evaluating $T_t^\lambda(\omega,X,U)$ with independent draws $X\sim\Gc(\mu_t^\theta)$ and $U\sim\Uc([0,1])$, without requiring an explicit expression for the resulting law. This sampling construction extends beyond the convex-combination and Gaussian-mixture examples below to any jointly measurable transport map that can be evaluated on sampled inputs.

The perturbation of \citet{meunier2026model} also has an encoder--decoder interpretation: it takes the form $\operatorname{softmax}(\operatorname{logit}(\mu_t^\theta)+\lambda U_t)$, with standard Gaussian $U_t$. Their noise is added to the encoded coordinates. In~\eqref{eq:general-transport-randomization}, randomness enters through the transport map on the state space. This representation gives a coupling between the original and perturbed laws, which can be used to study distances between these laws.

Using the probability-valued random variable  $M_t^{\lambda,\theta}$, we then consider the state and control processes  $(X_t^{\lambda,\theta})_{t \in \Tc}$ and $(\alpha_t^{\lambda,\theta})_{t =0}^{T-1}$ driven by
\begin{align}
    &X_0^{\lambda,\theta} \sim \mu_0,\qquad
    \alpha_t^{\lambda,\theta} \sim \pi_t^\theta(\cdot\mid X_t^{\lambda,\theta},M_t^{\lambda,\theta}),\qquad
    X_{t+1}^{\lambda,\theta} \sim P_t(\cdot\mid X_t^{\lambda,\theta},M_t^{\lambda,\theta},\alpha_t^{\lambda,\theta}),
    \label{eq:perturbed-controlled-process}
\end{align}
with return and policy score
\begin{align}
    \label{eq:perturbed-objective}
    \begin{split}
    \Rc_\theta^\lambda \coloneqq\sum_{t=0}^{T-1}r_t(X_t^{\lambda,\theta},M_t^{\lambda,\theta},\alpha_t^{\lambda,\theta})
    +g(X_T^{\lambda,\theta},M_T^{\lambda,\theta}),\qquad
    \Sc_{\mathrm{pol}}^{\lambda,\theta}&\coloneqq\sum_{t=0}^{T-1}
    \nabla_\theta\log p_t^\theta(\alpha_t^{\lambda,\theta}\mid X_t^{\lambda,\theta},M_t^{\lambda,\theta}),
    \end{split}
\end{align}
and objective $J^\lambda(\theta)\coloneqq\E[\Rc_\theta^\lambda]$.

Since $T_t^{0}$ is the identity, $J^0=J$ for $\Gc = \mathrm{Id}$; Gaussian-mixture projection gives $J_K^\lambda$ in~\eqref{eq:JKlambda}. Appendix~\ref{appendix:problem-formulation} gives a canonical construction of the control problem.


\subsection{Finite-state space}
\label{subsec:finite}
Let $\Xc=\{x_1,\ldots,x_N\}$, with $N\geq2$. We identify $\Pc(\Xc)$ with $\Delta_N = \{p\in[0,1]^N : \sum_i p_i = 1\}$ and write $\mu(i)\coloneqq\mu(\{x_i\})=\mu(x_i)$ for the density with respect to the counting measure. Since $\mu=\sum_{i=1}^N\mu(i)\delta_{x_i}$, randomizing the population law amounts to randomizing its weights. In the canonical construction of Appendix~\ref{appendix:canonical-construction}, take $E=\R^{N-1}$ and $R_t=U_t$, where $(U_t)_{t \in \llbracket1,T\rrbracket}$ are independent random vectors with density $\rho$ with respect to Lebesgue measure. Define
\begin{align}\label{eq : def_varp_phi}
    \varphi :\R^{N-1}\to\overset{\circ}{\Delta}_N,\quad
    \varphi(u)\coloneqq\frac{(e^{u_1},\ldots,e^{u_{N-1}},1)}{1+\sum_{j=1}^{N-1}e^{u_j}},
\end{align}
and set $Q_t\coloneqq\varphi(U_t)$. For $q\in\Delta_N$, let $F_q:[0,1]\to\Xc$ be the inverse distribution map, with $F_q(0)=x_1$ and $F_q(v)=x_{\min\{i:v\leq\sum_{j=1}^i q_j\}}$ for $v>0$. For $\lambda\in(0,1)$, define
\begin{align}
    T_t^{\lambda}(\omega,x,v)\coloneqq
    \begin{cases}
        x,&v\leq1-\lambda,\\
        F_{Q_t(\omega)}\bigl((v-1+\lambda)/\lambda\bigr),&v>1-\lambda,
    \end{cases}
\end{align}
and set $T_t^{0}(\omega,x,v)=x$. The transport retains the state with probability $1-\lambda$ and otherwise draws from $Q_t$. Consequently,
\begin{align}
    M_t^{\lambda,\theta}(\omega) &\coloneqq T_t^{\lambda}(\omega,\cdot,\cdot)\sharp (\mu_t^\theta\otimes\Uc([0,1]))\\
    &=(1-\lambda)\mu_t^\theta+\lambda Q_t(\omega),\qquad
    t\in\llbracket1,T\rrbracket.
    \label{eq:discrete-main-transport}
\end{align}
The assumptions and proofs below are given in Appendix~\ref{appendix:finite-proofs}.

\begin{Proposition}[Perturbation estimate]
    \label{prop:discrete-perturbation-estimate}
    For every $\lambda\in[0,1)$, $\theta\in\Theta$, and $t \in \Tc$,
    \begin{align}
        d_{\mathrm{TV}}(M_t^{\lambda,\theta}(\omega),\mu_t^\theta)\leq\lambda,
        \qquad\P(\d \omega)\text{-a.s.}
    \end{align}
\end{Proposition}

\begin{Theorem}[Perturbation consistency]
    \label{theorem:discrete-perturbation-consistency}
    Under Assumption~\ref{assumption:discrete-uniform-regularity}, there exists $C_T>0$, independent of $\theta$ and $\lambda$, such that, for $\lambda\in[0,1)$,
    \begin{align}
        \sup_{\theta\in\Theta}|J^\lambda(\theta)-J(\theta)|\leq C_T\lambda.
    \end{align}
    Moreover, under Assumption~\ref{assumption:discrete-gradient-regularity}, there exists $C_T^\nabla > 0$ such that
    \begin{align}
        \sup_{\theta\in\Theta}\|\nabla_\theta J^\lambda(\theta)-\nabla_\theta J(\theta)\|
        \leq C_T^\nabla\lambda.
    \end{align}
\end{Theorem}

\begin{Corollary}[Near-optimality of perturbed maximizers]
    \label{corollary:discrete-maximizers}
    Under Assumptions~\ref{assumption:discrete-uniform-regularity} and~\ref{assumption:discrete-gradient-regularity}, suppose that $\Theta$ is compact. Then, for every $\lambda\in[0,1)$, $J^\lambda$ attains its maximum on $\Theta$ at some $\theta_\lambda^\star$, and
    \begin{align}
        J(\theta_\lambda^\star)\geq\max_{\theta\in\Theta}J(\theta)-2C_T\lambda.
    \end{align}
    Moreover, every accumulation point of $(\theta_\lambda^\star)$ as $\lambda\to0$ maximizes $J$.
\end{Corollary}

\paragraph{Perturbation density and score.}
Use simplex coordinates $(q_1,\ldots,q_{N-1})$, with $q_N=1-\sum_{i=1}^{N-1}q_i$, and set $\varphi^{-1}(q)\coloneqq(\log(q_i/q_N))_{i=1}^{N-1}$. Densities on $\Delta_N$ are taken with respect to Lebesgue measure in these coordinates. By the change-of-variable formula, $Q_t$ has density
\begin{align}
    f_Q(q) = \frac{\rho(\varphi^{-1}(q))}{\prod_{i=1}^Nq_i},
    \qquad q\in\overset{\circ}{\Delta}_N.
\end{align}
If $\rho$ is strictly positive and $\Cc^1$, define, for $u\in\R^{N-1}$, $a_i(u)\coloneqq\partial_{u_i}\log\rho(u)$. Set $H\coloneqq(H_1,\ldots,H_{N-1})$, and define, for $k \in \llbracket 1,N-1 \rrbracket$, $H_k : \overset{\circ}{\Delta}_N \to \R$ by
\begin{align}
    H_k(q)&\coloneqq\partial_{q_k}\log f_Q(q)\\
    &= \frac{a_k(\varphi^{-1}(q))-1}{q_k} + \frac{1+\sum_{i=1}^{N-1}a_i(\varphi^{-1}(q))}{q_N}.
\end{align}
Write $z_t^\theta\coloneqq(\mu_t^\theta(1),\ldots,\mu_t^\theta(N-1))$. Let $\Zc_N\coloneqq\{z\in[0,1]^{N-1}:\sum_{i=1}^{N-1}z_i\leq1\}$ and write $q_{<N}=(q_1,\ldots,q_{N-1})$ for $q\in\Delta_N$. For $z\in\Zc_N$ and $q\in\overset{\circ}{\Delta}_N$, define
\begin{align}
    \Psi_\lambda(z,q)\coloneqq(1-\lambda)z+\lambda q_{<N}\in\Zc_N^\circ.\label{eq:discrete-perturbation-map}
\end{align}
Set $Z_t^{\lambda,\theta}\coloneqq\Psi_\lambda(z_t^\theta,Q_t)$. By~\eqref{eq:discrete-main-transport}, $Z_t^{\lambda,\theta}$ is the coordinate vector associated to $M_t^{\lambda,\theta}$.
Fix $z\in\Zc_N$ and set $Y_\lambda(z)\coloneqq(1-\lambda)z+\lambda\Zc_N^\circ$. For $y\in Y_\lambda(z)$, define $q_i^\lambda(y\mid z)=(y_i-(1-\lambda)z_i)/\lambda$ for $i<N$ and $q_N^\lambda(y\mid z)=1-\sum_{i<N}q_i^\lambda(y\mid z)$, so that $q^\lambda(y\mid z)\in\overset{\circ}{\Delta}_N$.
By Lemma~\ref{lemma:discrete-law-M}, $\Psi_\lambda(z,Q_t)$ admits a density $h^\lambda(\cdot\mid z)$, zero outside $Y_\lambda(z)$ and given, for $y\in Y_\lambda(z)$, by
\begin{align}
    h^\lambda(y\mid z)=\lambda^{-(N-1)}f_Q\big(q^\lambda(y\mid z)\big).
    \label{eq:discrete-main-density}
\end{align}
In particular, $Z_t^{\lambda,\theta}$ has density $h^\lambda(\cdot\mid z_t^\theta)$. Its base-coordinate score, with $y$ held fixed, is
\begin{align}
    s^\lambda(y,z):= \nabla_z\log h^\lambda(y| z)
    =\frac{\lambda-1}{\lambda}H\big(q^\lambda(y| z)\big).
    \label{eq:discrete-main-score}
\end{align}
The score with respect to the policy parameter is
\begin{align}
    \nabla_\theta\log h^\lambda(y\mid z_t^\theta)=(D_t^\theta)^\top s^\lambda(y,z_t^\theta),
\end{align}
where $D_t^\theta\coloneqq\nabla_\theta z_t^\theta\in\R^{(N-1)\times d_\theta}$.

The support $Y_\lambda(z_t^\theta)$ of $h^\lambda(\cdot\mid z_t^\theta)$, and hence the integration domain, moves with $\theta$. For the score-function representation to have no boundary terms, the perturbation density $f_Q$ must vanish on the simplex boundary $\partial\Delta_N$. This holds when $\rho$ satisfies Assumption~\ref{assumption:randomizer_rho_discrete}: a Gaussian randomizer does, whereas a uniform distribution on the simplex does not. The logit–softmax parametrization avoids this requirement, since its integration domain does not depend on $\theta$.

\paragraph{Policy-gradient representation.}
Define the population score along a perturbed trajectory by
\begin{align}
    \Sc_{\mathrm{mf}}^{\lambda,\theta}
    &\coloneqq\sum_{t=1}^T\nabla_\theta\log h^\lambda(Z_t^{\lambda,\theta}\mid z_t^\theta)\\
    &=\sum_{t=1}^T(D_t^\theta)^\top s^\lambda(Z_t^{\lambda,\theta},z_t^\theta)
    =\frac{\lambda-1}{\lambda}\sum_{t=1}^T(D_t^\theta)^\top H(Q_t),
    \label{eq:discrete-main-scores}
\end{align}
where the last equality uses $q^\lambda(Z_t^{\lambda,\theta}\mid z_t^\theta)=Q_t$ in~\eqref{eq:discrete-main-score}.
\begin{Theorem}[Model-free policy-gradient representation]
    \label{theorem:discrete-pg-representation}
    Under Assumptions ~\ref{assumption:discrete-uniform-regularity},
    ~\ref{assumption:discrete-gradient-regularity} and~\ref{assumption:randomizer_rho_discrete}, for every $\lambda\in(0,1)$, and using the quantities defined in~\eqref{eq:perturbed-objective} and~\eqref{eq:discrete-main-scores},
    \begin{align}
        \nabla_\theta J^\lambda(\theta) = \E\big[\Rc_\theta^\lambda (\Sc_{\mathrm{pol}}^{\lambda,\theta}+\Sc_{\mathrm{mf}}^{\lambda,\theta})\big].
    \end{align}
\end{Theorem}
    
The estimator follows Algorithm~\ref{alg:transport-reinforce}: $\mu_t^\theta$ is replaced by the empirical flow $\widehat\mu_t^M$ of $M$ interacting particles, and $D_t^\theta$ by the estimate $\widehat D_t^{\eta,M,n}$ in~\eqref{eq:discrete-sensitivity-estimator}, computed from $n$ auxiliary trajectories at scale $\eta$. With the returns $\widehat\Rc^{(b)}$ and policy scores $\widehat\Sc_{\mathrm{pol}}^{(b)}$ of $B$ independent trajectories perturbed by $Q_t^{(b)}$, the estimator of $\nabla_{\theta} J^{\lambda}(\theta)$ is
\begin{align}
    \widehat G_{B,M,n}^{\lambda,\eta}(\theta) &\coloneqq\frac1B\sum_{b=1}^B\widehat\Rc^{(b)} \Big(\widehat\Sc_{\mathrm{pol}}^{(b)}+ \widehat\Sc_\mathrm{mf}^{(b)}\Big).
    \label{eq:discrete-main-estimator}
\end{align}
A deterministic scalar baseline may reduce variance without changing the expectation, since both scores have mean zero.

\begin{Theorem}[Bias and mean-square error]
    \label{theorem:discrete-bias-mse}
    Under Assumptions~\ref{assumption:discrete-uniform-regularity}, \ref{assumption:discrete-gradient-regularity}, \ref{assumption:randomizer_rho_discrete}, and~\ref{assumption:discrete-score-regularity}, if $0<\lambda,\eta<1$, integers $B,M,n\geq1$, and $n\eta^2\geq1$, there exists $C>0$, independent of $\theta,\lambda,\eta,B,M,n$, such that
    \begin{align}
        \|\E[\widehat G_{B,M,n}^{\lambda,\eta}(\theta)]-\nabla_\theta J(\theta)\|
        &\leq C\bigg(\lambda+\eta+\frac1{\eta\sqrt n}+\frac1{\sqrt M}\bigg),\\
        \E \Big[\|\widehat G_{B,M,n}^{\lambda,\eta}(\theta)-\nabla_\theta J(\theta)\|^2 \Big]
        &\leq C\bigg(\lambda^2+\frac{1+\lambda^{-2}}B+\eta^2+\frac1{n\eta^2}+\frac1M\bigg).
    \end{align}
\end{Theorem}
Minimizing the $\lambda$-dependent terms gives $\lambda_\star\asymp B^{-1/4}$, and balancing the auxiliary terms gives $\eta_\star\asymp n^{-1/4}$. These are choices based on the MSE bound, with unknown constants.

\subsection{Continuous-state space}
\label{subsec:continuous}
\paragraph{Projection onto Gaussian mixtures.}
Let $\Xc=\R^d$ and let the action space $\Ac$ be a non-empty closed subset of $\R^q$.
Population laws belong to $\Pc_2(\R^d)$, and the policies take values in $\Pc_2(\Ac)$.
To obtain a tractable perturbation density, we represent the population law by a Gaussian mixture and randomize its parameters. Fix $K\geq1$ and a unit vector $e\in\R^d$. Define $\M_K^e \coloneqq \{(m_1,\ldots,m_K)\in(\R^d)^K: \langle e,m_1\rangle<\cdots<\langle e,m_K\rangle\}$, $\Zc_K \coloneqq\Delta_K^\circ\times\M_K^e\times(\S_{++}^d)^K$ and $q_K\coloneqq(K-1)+Kd+K\frac{d(d+1)}2$.

We use the first $K-1$ weights as coordinates, with $p_K = 1-\sum_{j=1}^{K-1}p_j$. Together with the means and the independent covariance entries, they identify $\Zc_K$ with an open subset of $\R^{q_K}$. The ordering fixes component labels and restricts fitted mixtures to distinct projected means. Randomization keeps these labels but may reorder the means, so the decoder is defined on the labeled space $\widetilde\Zc_K\coloneqq\Delta_K^\circ\times(\R^d)^K\times(\S_{++}^d)^K$, with no ordering constraint. The encoder assigns mixture parameters to each population law:
\begin{align}
    \Ec_K : \Pc_2(\R^d) \to \Zc_K,\quad
    \mu \mapsto \bigl(p_K(\mu),m_K(\mu),\Sigma_K(\mu)\bigr),
\end{align}
where $p_K(\mu)=(p_{j,K}(\mu))_{j=1}^K$, $m_K(\mu)=(m_{j,K}(\mu))_{j=1}^K$, and $\Sigma_K(\mu) =\big(\Sigma_{j,K}(\mu)\big)_{j=1}^{K}$ are the fitted weights, means, and covariance matrices, respectively. The decoder reconstructs the corresponding mixture:
\begin{align}
    \Dc_K : \widetilde\Zc_K \to \Pc_2(\R^d),\quad
    (p,m,\Sigma) \mapsto \sum_{j=1}^K p_j\Nc(m_j,\Sigma_j).
\end{align}
Thus the population approximation is
\begin{align}
    \Gc_K : \Pc_2(\R^d)\to\Pc_2(\R^d),\quad
    \Gc_K(\mu)\coloneqq\Dc_K(\Ec_K(\mu)).
\end{align}
Throughout, $\Sigma_j$ denotes a positive definite covariance matrix and $\Gamma_j=\Sigma_j^{1/2}$ its unique symmetric positive-definite square root. We assume a measurable encoder with bounded, nondegenerate fitted parameters along the population flow. Gradient results additionally require smooth dependence of these parameters on $\theta$; Appendix~\ref{app:encoder-decoder} states the precise conditions required to ensure the needed regularity. One choice is a constrained likelihood fit, approximated on empirical laws by constrained Expectation-Maximization iterations \citep{dempster1977maximum}.

The initial population argument is kept exact: $M_{K,0}^{\lambda,\theta} = \mu_0$. For each $t\in\llbracket1,T\rrbracket$,
write $z_{K,t}^\theta \coloneqq\Ec_K(\mu_t^\theta)$, with components $(p_{K,t}^\theta,m_{K,t}^\theta,\Sigma_{K,t}^\theta)$, and $\nu_{K,t}^\theta\coloneqq\Dc_K(z_{K,t}^\theta)$. The vectors $p_{K,t}^\theta$, $m_{K,t}^\theta$, and $\Sigma_{K,t}^\theta$ collect the component weights, means, and covariance matrices.

\paragraph{Randomized transport.}\label{paragraph:randomized_transport}
Let $R_t=(Q_t,A_t,B_t)$, $t\in \llbracket 1, T \rrbracket$, be i.i.d. randomizers with law $\rho$, independent of $\theta$ and of the simulation randomness. Their coordinates lie in $\Delta_K^\circ\times(\R^d)^K\times(\S^d)^K$. We take $Q_t=\varphi(U_t)$, using the map in~\eqref{eq : def_varp_phi} with $N$ replaced by $K$. Assume that $B_{j,t}\in\Bc_\lambda\coloneqq\{B\in\S^d:(1-\lambda)I_d+\lambda B\succ0\}$ almost surely. When $K=1$, $Q_t=1$ and the free weight coordinates are absent. For $z=\big(p,m,(\Sigma_j)_{j=1}^{K} \big)$ and $r=\big(q,a,(b_j)_{j=1}^{K} \big)$, define
\begin{align}
 \Psi_\lambda(z,r)
 =\Big((1-\lambda)p+\lambda q,\ (1-\lambda)m+\lambda a,
        \big(\Sigma_j^{1/2}((1-\lambda)I_d+\lambda b_j)^2\Sigma_j^{1/2}\big)_{j=1}^K\Big),
 \label{eq:continuous-perturbation-map}
\end{align}
and set $Z_{K,t}^{\lambda,\theta}\coloneqq\Psi_\lambda(z_{K,t}^\theta,R_t)$, with components $(p_{K,t}^{\lambda,\theta},m_{K,t}^{\lambda,\theta},\Sigma_{K,t}^{\lambda,\theta})$. It takes values in $\widetilde\Zc_K$, since positive definiteness follows from the support condition on $B_t$. Set $M_{K,t}^{\lambda,\theta}(\omega) :=\Dc_K(Z_{K,t}^{\lambda,\theta}(\omega))$, $\Gamma_{j,K,t}^{\lambda,\theta}(\omega):=(\Sigma_{j,K,t}^{\lambda,\theta}(\omega))^{1/2}$, and $\Gamma_{j,K,t}^{\theta}:=(\Sigma_{j,K,t}^{\theta})^{1/2}$. 

\begin{Proposition}[Transport representation]
\label{prop:continuous-state-transport}
For every $\lambda\in[0,1)$, the construction in Appendix~\ref{app:continuous-construction} gives a jointly measurable map $T_{K,t}^{\lambda,\theta}: \Omega \times \R^d \times [0,1] \to \R^d$ such that
\begin{align}
 M_{K,t}^{\lambda,\theta}(\omega)&=T_{K,t}^{\lambda,\theta}(\omega,\cdot,\cdot)\sharp(\nu_{K,t}^\theta\otimes\Uc([0,1]))\\
 &= \sum_{j=1}^K p_{j,K,t}^{\lambda,\theta}(\omega)\,\Nc\Big(m_{j,K,t}^{\lambda,\theta}(\omega),\Sigma_{j,K,t}^{\lambda,\theta}(\omega)\Big).
\end{align}
\end{Proposition}
The perturbed process $(X_t^{K,\lambda,\theta},\alpha_t^{K,\lambda,\theta})$, its return $\Rc_K^{\lambda,\theta}$ and its policy score $\Sc_{\mathrm{pol}}$ are defined by~\eqref{eq:perturbed-controlled-process}--\eqref{eq:perturbed-objective}, with $M_{K,t}^{\lambda,\theta}$ in place of $M_t^{\lambda,\theta}$, and
\begin{align}
    \label{eq:JKlambda}
    J_K^{\lambda}(\theta) \coloneqq \E\big[\Rc_K^{\lambda,\theta} \big].
\end{align}
For $\lambda=0$, the transport is the identity, so $M_{K,t}^{0,\theta} = \nu_{K,t}^\theta$ for $t\in\llbracket1,T\rrbracket$. We write $J_K \coloneqq J_K^0$ for the projected objective. The assumptions and proofs of the following results are in Appendix~\ref{subsec:continuous-proofs}.

\begin{Proposition}[Perturbation estimate]
    \label{proposition:continuous-perturbation-estimate}
    Under Assumption~\ref{assumption:continuous-moment-bounds}, there exists $C_K>0$, uniform in $t$, $\theta$, and $\lambda\in[0,1)$, such that
    \begin{align}
        \big(\E[\Wc_2^2(M_{K,t}^{\lambda,\theta},\mu_t^\theta)]\big)^{1/2}
        \leq\Wc_2(\nu_{K,t}^\theta,\mu_t^\theta)+C_K \sqrt{\lambda}.
    \end{align}
\end{Proposition}

\begin{Theorem}[Objective perturbation consistency]
    \label{theorem:continuous-objective-perturbation-consistency}
    Under Assumptions~\ref{assumption:continuous-moment-bounds} and~\ref{assumption:continuous-objective-regularity}, there exists $C_{K,T}>0$, independent of $\theta$ and $\lambda$, such that
    \begin{align}
        \sup_{\theta\in\Theta}|J_K^\lambda(\theta)-J_K(\theta)|\leq C_{K,T}\sqrt\lambda,
        \qquad\lambda\in[0,1).
    \end{align}
\end{Theorem}

\paragraph{Projection error.}
The preceding limits hold for fixed $K$. To compare with the original objective, define $\varepsilon_K\coloneqq\sup_{\theta\in\Theta}\max_{1\leq t\leq T} \Wc_2(\nu_{K,t}^\theta,\mu_t^\theta)$. Under Assumption~\ref{assumption:continuous-objective-regularity}, the same stability argument gives
\begin{align}
    \sup_{\theta\in\Theta}|J_K(\theta)-J(\theta)| &\leq C_T^{\mathrm{proj}}\varepsilon_K,\\
    \sup_{\theta\in\Theta}|J_K^\lambda(\theta)-J(\theta)| &\leq C_T^{\mathrm{proj}}\varepsilon_K+C_{K,T}\sqrt\lambda,
    \label{eq:continuous-projection-error}
\end{align}
where $C_T^{\mathrm{proj}}>0$ depends only on $T$ and the Lipschitz constant; see Appendix~\ref{proof:continuous-objective-consistency}. A rate for $\varepsilon_K$ as $K$ grows requires additional approximation assumptions.

\begin{Proposition}[Optimality gap]
    \label{prop:continuous-optimality-gap}
    Under Assumptions~\ref{assumption:continuous-moment-bounds} and~\ref{assumption:continuous-objective-regularity}, let $\theta^{\star}_{K,\lambda}\in\Theta$ be a maximizer of $J_K^\lambda$ over $\Theta$. Then
    \begin{align}
        0\leq\sup_{\theta\in\Theta}J(\theta)-J(\theta^{\star}_{K,\lambda})
        \leq2\big(C_T^{\mathrm{proj}}\varepsilon_K+C_{K,T}\sqrt\lambda\big).
    \end{align}
\end{Proposition}
Thus a maximizer of $J_K^\lambda$ is near-optimal for the original problem. The loss has two sources: the mixture projection error $\varepsilon_K$, controlled by $K$, and a perturbation cost of order $\sqrt\lambda$, controlled by the randomization scale. The proof is in Appendix~\ref{proof:continuous-objective-consistency}.

\begin{Theorem}[Gradient perturbation consistency]
    \label{theorem:continuous-gradient-convergence}
    Under the preceding Assumptions and Assumption~\ref{assumption:continuous-gradient-regularity}, $J_K^\lambda$ and $J_K$ are differentiable, there exists $C_{K,T}^\nabla < \infty$, independent of $\theta$ and $\lambda$, and there exists $\lambda_0 \in (0,1)$, such that, for $\lambda\in[0,\lambda_0]$,
    \begin{align}
        \sup_{\theta\in\Theta} \|\nabla_\theta J_K^\lambda(\theta) - \nabla_\theta J_K(\theta)\| \leq C_{K,T}^\nabla \lambda.
    \end{align}
\end{Theorem}

\paragraph{Perturbation density and policy gradient.}
Densities on $\widetilde\Zc_K$ are taken with respect to Lebesgue measure in the free weight coordinates, mean coordinates, and independent covariance entries. Fix $\lambda \in(0,1)$ and a base point $z = (p,m,\Sigma) \in\widetilde \Zc_K$, and set
\begin{align}
    Y_\lambda(z)\coloneqq\big((1-\lambda)p+\lambda\Delta_K^\circ\big)
    \times(\R^d)^K\times(\S_{++}^d)^K.
\end{align}
Inverting $\Psi_\lambda(z,\cdot)$ at $y=(\widetilde p,\widetilde m,\widetilde\Sigma)\in Y_\lambda(z)$ gives the inverse coordinates $(q^\lambda,a^\lambda,b^\lambda)(y\mid z)$: for $j\in\llbracket1,K\rrbracket$,
\begin{equation}
\begin{aligned}
    &(q^\lambda,a^\lambda)=\lambda^{-1}\big[(\widetilde p,\widetilde m)-(1-\lambda)(p,m)\big],\qquad
    b_j^\lambda=\lambda^{-1}\big[(\Sigma_j^{-1/2}\widetilde\Sigma_j\Sigma_j^{-1/2})^{1/2}-(1-\lambda)I_d\big].
\end{aligned}
\label{eq:continuous-inverse-coordinates}
\end{equation}

\begin{Lemma}[Perturbation density]
\label{lemma:continuous-main-density}
Under conditions (1)--(2) of Assumption~\ref{assumption: randomizer_rho}, the random variable $\Psi_\lambda(z,R_t)$ admits a density $h^\lambda(\cdot\mid z):\widetilde\Zc_K\to[0,\infty)$. For $y\in Y_\lambda(z)$, write $(q,a,b)=(q^\lambda,a^\lambda,b^\lambda)(y\mid z)$. Then
\begin{align}
    h^\lambda(y\mid z)
    =\frac{f_Q(q)\,f_A(a)\,f_B(b)}{\mathcal J^\lambda(z,b)},
    \label{eq:continuous-main-density}
\end{align}
where $\mathcal J^\lambda:\widetilde\Zc_K\times(\Bc_\lambda)^K\to(0,\infty)$ is the Jacobian factor given in~\eqref{eq:continuous-explicit-jacobian}. The density is zero outside $Y_\lambda(z)$.
\end{Lemma}
The formula for $\mathcal J^\lambda$ and the proof are given in Lemma~\ref{lemma: Law_z_k} in Appendix~\ref{proof:continuous-perturbation-density}. For $K=1$, the factor $f_Q$ and the free weight coordinates are omitted. Where the density is positive and differentiable in $z$, define its base-coordinate score by
\begin{align}
    s^\lambda(y,z)\coloneqq\nabla_z\log h^\lambda(y\mid z)\in\R^{q_K},
    \label{eq:continuous-main-score}
\end{align}
which is the score with respect to the base coordinates $z$, with $y$ held fixed. The score with respect to the policy parameter, $D_t^\theta=\nabla_\theta z_{K,t}^\theta\in\R^{q_K\times d_\theta}$, is
\begin{align}
      \nabla_{\theta} \log h^{\lambda}(y \mid z_{K,t}^{\theta})= (D_t^\theta)^\top s^\lambda(y,z_{K,t}^\theta).
      \label{eq:continuous-parameter-score}
\end{align}

This general construction specializes to the fully solvable linear dynamics framework developed in Section~\ref{eq : LQ-setting}, where restricting to a single Gaussian component ($K=1$), the parameter score in~\eqref{eq:continuous-parameter-score} reduces exactly to the explicit mean-field score of Corollary~\ref{corollary: score_mean_field_density}.

Along a perturbed trajectory, define the population score
\begin{align}
 \Sc_{\mathrm{mf}}
 \coloneqq \sum_{t=1}^{T} \nabla_{\theta} \log h^{\lambda}(Z_{K,t}^{\lambda,\theta} |z_{K,t}^{\theta}) = \sum_{t=1}^T(D_t^\theta)^\top s^\lambda(Z_{K,t}^{\lambda,\theta},z_{K,t}^\theta).
 \label{eq:continuous-main-scores}
\end{align}

\begin{Theorem}[Model-free policy-gradient representation]
 \label{theorem:continuous-pg-representation}
 Let $\lambda\in(0,1)$. Under Assumptions~\ref{assumption: randomizer_rho} and~\ref{assumption:continuous-likelihood-ratio-regularity}, the following holds:
 \begin{align}
  \nabla_\theta J_K^\lambda(\theta)=\E[\Rc_K^{\lambda,\theta}(\Sc_{\mathrm{pol}}+\Sc_{\mathrm{mf}})].
 \end{align}
\end{Theorem}
The fitted coordinates are $\widehat z_{K,t}^M=\Ec_K(\widehat\mu_t^M)$, $\widehat D_t^{\eta,M,n}$ is the centered difference~\eqref{eq:continuous-algorithm-sensitivity} \citep{jia2026zeroth} over $n=2d_\theta n_0$ particles, and $\widehat Z_t^{(b)}=\Psi_\lambda(\widehat z_{K,t}^M,R_t^{(b)})$. With the returns $\widehat\Rc^{(b)}$ and policy scores $\widehat\Sc_{\mathrm{pol}}^{(b)}$ of $B$ independent main trajectories, the estimator is
\begin{align}
 \widehat G_{B,M,n}^{\lambda,\eta}(\theta)
 &\coloneqq\frac1B\sum_{b=1}^B\widehat\Rc^{(b)}
 \Big(\widehat\Sc_{\mathrm{pol}}^{(b)} + \widehat \Sc_\mathrm{mf}^{(b)}\Big).
 \label{eq:continuous-main-estimator}
\end{align}

Let $a_L$ bound the root mean-square error of the coordinates fitted from $L$ particles. Then $a_M$ controls the population fitting error, while $a_{n_0}/\eta$ controls the fitting error in the centered differences. Unlike empirical frequencies in finite state space, mixture coordinates require a non-linear fit, so the rate $a_L = L^{-1/2}$ needs an additional assumption on its stability.
\begin{Theorem}[Bias and mean-square error]
    \label{theorem:continuous-bias-mse}
    Under Assumptions~\ref{assumption:continuous-moment-bounds}, \ref{assumption:continuous-objective-regularity}, \ref{assumption:continuous-gradient-regularity}, \ref{assumption: randomizer_rho}, \ref{assumption:continuous-likelihood-ratio-regularity}, and~\ref{assumption:continuous-statistical-regularity}, let $n_0 = n/(2d_\theta)$. For admissible scales $0<\lambda\leq\lambda_0$, $0<\eta\leq\eta_0$, with $a_{n_0} \leq \eta$, there exists $C<\infty$, independent of $\theta,\lambda,\eta,B,M,n_0$ such that 
    \begin{align}
        \Big\|\E\big[\widehat G^{\lambda,\eta}_{B,M,n}(\theta)\big]-\nabla_\theta J_K(\theta)\Big\|
        &\leq C\bigg(\lambda + a_M + \eta^2 + \frac{a_{n_0}}{\eta}  \bigg),\\
        \E \Big[\big\|\widehat G^{\lambda,\eta}_{B,M,n}(\theta)-\nabla_\theta J_K(\theta)\big\|^2 \Big]
        &\leq C\bigg( \lambda^2 + a_M^2 + \eta^4 + \frac{a_{n_0}^2}{\eta^2} + \frac{1+\lambda^{-2}}{B} \bigg).
    \end{align}
\end{Theorem}
If Assumption~\ref{assumption:continuous-statistical-regularity} (2) holds with $a_N = N^{-1/2}$, i.e. $\sup_\theta \E\big[\|\widehat z_K^{N,\theta} - z^\theta\|^2 \big] \leq C/N$, then for $n_0\eta^2 \geq 1$, we can balance the scale-dependent terms in the MSE bound to get $\lambda_\star\asymp B^{-1/4}$ and $\eta_\star\asymp n_0^{-1/6}$.

\begin{algorithm}[h]
\caption{Transport REINFORCE}
\label{alg:transport-reinforce}
\begin{algorithmic}[1]
\REQUIRE initial parameter $\theta$; budgets $M,n,B$; scales $\lambda,\eta$; step size $\gamma$; randomizer law $\rho$; mixture size $K$ (continuous case)
\FOR{each policy update}
    \STATE \textbf{Offline stage: population and sensitivity}
    \STATE Estimate $(\widehat\mu_t^M)_{t=0}^T$ using $M$ interacting particles under $\theta$
    \IF{finite state space}
        \STATE Compute $(\widehat D_t^{\eta, M, n})_{t=1}^T$ recursively by~\eqref{eq:discrete-sensitivity-estimator}, using $n$ rollouts per horizon at scale $\eta$
    \ELSE
        \STATE Fit $K$-component mixtures to $(\widehat\mu_t^M)_t$
        \STATE For each $\ell=1,\ldots,d_\theta$, fit $\widehat z_t^{\theta\pm\eta e_\ell}$ from independent $n/(2d_\theta)$-particle systems under $\theta\pm\eta e_\ell$; set, for $t=1,\ldots,T$, 
        \STATE $\displaystyle \widehat D_t^{\eta,M,n} e_\ell\gets\big(\widehat z_t^{\theta+\eta e_\ell}-\widehat z_t^{\theta-\eta e_\ell}\big)/\big(2\eta\big)$
    \ENDIF
    \STATE \textbf{Online stage: gradient estimation}
    \STATE Hold the offline estimates fixed and simulate $B$ perturbed trajectories at scale $\lambda$, as in Section~\ref{subsec:finite} or~\ref{subsec:continuous}
    \STATE Compute $\widehat G_{B,M,n}^{\lambda,\eta}$ using~\eqref{eq:discrete-main-estimator} or~\eqref{eq:continuous-main-estimator}
    \STATE $\theta\gets\theta+\gamma\widehat G_{B,M,n}^{\lambda,\eta}$
\ENDFOR
\end{algorithmic}
\end{algorithm}

\section{Policy gradient algorithm}
\label{sec:policy-gradient-algorithm}

Each policy update has two stages. The offline stage estimates the population flow and its sensitivities at the current parameter $\theta$. The online stage holds these estimates fixed and uses $B$ perturbed trajectories to estimate the gradient. Both stages are repeated after each update; Appendix~\ref{app:algorithm-complexity} gives the sampling details and computational cost.
Let $(e_\ell)_{\ell=1}^{d_\theta}$ denote the standard basis of $\R^{d_\theta}$.

\section{Numerical experiments}
\label{sec:experiments}

\paragraph{Setup.}
We evaluate four finite-state benchmarks (two-state control, cybersecurity, distribution planning, advertising) and three continuous-state benchmarks (linear--quadratic (LQ) control, portfolio, bimodal allocation). The results for cybersecurity and advertising are presented in Appendix~\ref{app:additional-results}. Baselines are REINFORCE, MF-REINFORCE \citep{meunier2026model} on finite states, and, on LQ control and portfolio, Transport-Proba, which replaces the centered differences by a likelihood-ratio estimate of the moment sensitivities, as per Appendix~\ref{eq : LQ-setting}. Methods share the policy class, initialization, number of updates and simulator budget. Results use five paired seeds.
We report $|J(\widehat\theta)-J(\theta^\star)|$, where $\widehat \theta$ is the final parameter returned by training and $\theta^\star$ maximizes the original, unperturbed objective $J$ over the benchmark's admissible policy class. Thus the reference $\theta^\star$ does not depend on $\lambda$, although the learned parameter does. Distribution planning uses a numerical reference.
Appendix~\ref{app:experiments} gives models, evaluation and settings. The code is available at: \url{https://github.com/adonis107/RL-MFC-TransportREINFORCE}.

\begin{figure*}[h]
    \centering
    \includegraphics[width=\textwidth]{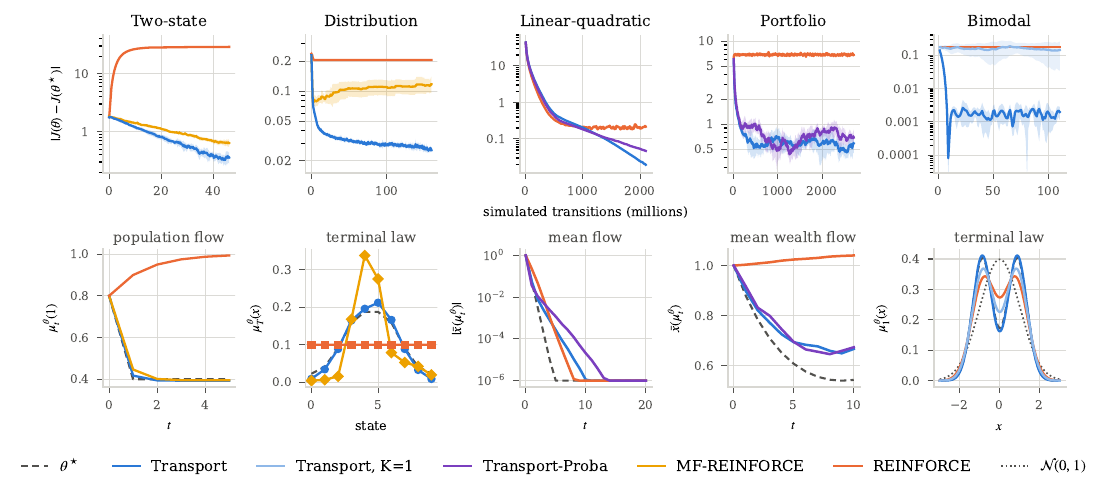}
    \caption{Top: validation gap against simulated transitions (mean and standard deviation over five paired seeds). Bottom: the final policy's population flow or terminal law for the first seed. Transport and Transport-Proba are shown at their best scale. Dashed black curves show the reference policy.}
    \label{fig:main-benchmarks}
\end{figure*}

\begin{table*}[h]
\centering
\footnotesize
\setlength{\tabcolsep}{2.8pt}
\caption{Final optimality gap $|J(\widehat\theta)-J(\theta^\star)|$ (mean $\pm$ std. over five paired seeds; smaller is better). Transport is shown at $\lambda_\star=B^{-1/4}$ and at its best tested scale in parentheses. Transport-Proba uses its best tested scale. Scale choices and evaluation of the reference policies are detailed in Appendix~\ref{app:evaluation}.}
\label{tab:main-summary}
\resizebox{\textwidth}{!}{%
\begin{tabular}{lrrrrrr}
\toprule
& & & & \multicolumn{2}{c}{Transport} & Transport-Proba \\
\cmidrule(lr){5-6}\cmidrule(lr){7-7}
Benchmark & $J(\theta^\star)$ & REINFORCE & MF-REINFORCE & at $\lambda_\star$ & best scale & best scale \\
\midrule
Two-state & $-2.64$ & $29.0\pm0.0$ & $0.637\pm0.093$ & $1.36\pm0.13$ & $\mathbf{0.359\pm0.091}$ {\scriptsize$(\lambda_\star/4)$} & --- \\
Distribution & $-0.05699$ & $0.206\pm0.002$ & $0.118\pm0.024$ & $0.0296\pm0.0012$ & $\mathbf{0.0255\pm0.0011}$ {\scriptsize$(\lambda_\star/2)$} & --- \\
Linear--quadratic & $-7.224$ & $0.216\pm0.002$ & --- & $0.0390\pm0.0027$ & $\mathbf{0.0196\pm0.0012}$ {\scriptsize$(\lambda_\star/4)$} & $0.0469\pm0.0028$ {\scriptsize$(\lambda_\star/2)$} \\
Portfolio & $-13.15$ & $6.87\pm0.07$ & --- & $1.22\pm0.07$ & $\mathbf{0.586\pm0.152}$ {\scriptsize$(\lambda_\star/4)$} & $0.692\pm0.156$ {\scriptsize$(\lambda_\star/2)$} \\
Bimodal ($K=2$) & $0$ & $0.175\pm0.009$ & --- & $0.0903\pm0.0053$ & $\mathbf{0.00199\pm0.00094}$ {\scriptsize$(\lambda_\star/8)$} & --- \\
\bottomrule
\end{tabular}
}
\end{table*}

\paragraph{Optimality gap.}
For each seed $s$, we compute the paired validation gap $\widehat\Delta_s=|\widehat J_s(\widehat\theta_s)-\widehat J_s(\theta^\star)|$ and report its mean and standard deviation. Transport has the smallest mean gap in every row of Table~\ref{tab:main-summary} after scale selection. The largest absolute gains occur in two-state control and the portfolio, where REINFORCE ignores how the policy changes the population; in two-state control it moves almost all the mass to state~1 and incurs the population penalties (Figure~\ref{fig:main-benchmarks}). The comparison is benchmark-dependent: cybersecurity gives similar outcomes for Transport and REINFORCE, while MF-REINFORCE has the higher mean objective on advertising (Appendix~\ref{app:additional-benchmarks}). At matched budgets, a Transport run takes $2$ to $16$ times the wall-clock time of REINFORCE, since each update also estimates sensitivities and population scores (Table~\ref{tab:budget-runtime}).

\paragraph{Mixture size.}
A single Gaussian retains the moments needed by linear--quadratic control and the portfolio. Bimodal allocation instead depends on the shape of the terminal law: its mean and variance are constant in $\theta$, so the $K=1$ projected objective is constant. With $K=2$, the law is represented exactly. At the same scale, the gap decreases from $0.1425\pm0.1239$ to $0.0020\pm0.0009$ (Table~\ref{tab:bimodal-components}).

\begin{figure}[h]
    \centering
    \includegraphics[width=0.5\columnwidth]{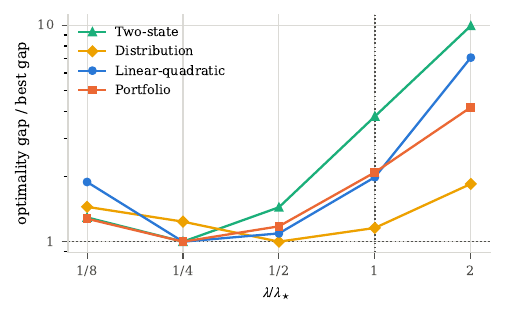}
    \caption{Final gap against $\lambda/\lambda_\star$, divided by the smallest gap on the grid (mean over five paired seeds).}
    \label{fig:lambda-tradeoff}
\end{figure}

\paragraph{Perturbation scale.}
Smaller $\lambda$ reduces perturbation bias but increases the variance of the population score. Consequently, the best training scale depends on both approximation and optimization error. Figure~\ref{fig:lambda-tradeoff} shows the resulting trade-off on four benchmarks: the smallest gap is reached at $\lambda_\star/4$ or $\lambda_\star/2$. On two-state control, $\lambda_\star$ gives a larger gap than MF-REINFORCE, whereas $\lambda_\star/4$ gives a smaller one. On bimodal allocation, the mean optimization error is below $10^{-4}$ across the grid, and the smallest tested $\lambda$ performs best. Thus $B^{-1/4}$ provides a scale for tuning, with a benchmark-dependent constant.

\begin{table}[h]
\centering
\footnotesize
\setlength{\tabcolsep}{3pt}
\caption{Predicted powers $p_0$ and ranges of fitted log-log slopes $p$ across the six benchmarks in Table~\ref{tab:bounds-exponents}. The exact-flow $\eta$ sweep uses the two continuous benchmarks. Fitting details are in Appendix~\ref{experiments:bias-mse}.}
\label{tab:bounds-exponents-main}
\begin{tabular}{llrr}
\toprule
Error & Swept & $p_0$ & Measured $p$ \\
\midrule
Sensitivity MSE & $\eta$ (exact) & $4$ & $[3.99,4.02]$ \\
Sensitivity MSE & $n$ & $-1$ & $[-1.23,-0.51]$ \\
Gradient bias & $\lambda$ & $1$ & $[1.00,1.28]$ \\
Gradient variance & $B$ & $-1$ & $[-1.03,-0.90]$ \\
Population MSE & $M$ & $-1$ & $[-1.07,-0.98]$ \\
\bottomrule
\end{tabular}
\end{table}

\paragraph{Bias and mean-square error.}
Table~\ref{tab:bounds-exponents-main} compares predicted and measured power. Sampling largely follows $n^{-1}$, $M^{-1}$ and $B^{-1}$, perturbation bias is nearly linear in $\lambda$ at small scales, and exact-flow sensitivity follows $\eta^4$.

\section{Conclusion}

We introduced Transport REINFORCE, a model-free policy-gradient method for discrete-time mean-field control. Randomizing the population argument through a transport map gives the perturbed law a tractable density, whose score recovers the mean-field contribution missed by REINFORCE. We established objective and gradient consistency for the finite-state and projected continuous-state problems, with estimator bias, mean-square error, and objective projection error bounds. At matched simulator budgets, the method attains the smallest mean optimality gap on the five main benchmarks after scale selection, and the measured error rates agree with the predicted ones.

Several limitations remain. The bounds only fix the order $B^{-1/4}$ of the perturbation scale; the best constant is benchmark-dependent, between $\lambda_\star/8$ and $\lambda_\star/2$, and is selected after training.
In continuous state spaces, Gaussian mixtures are dense in
$(\Pc_2(\R^d),\Wc_2)$ \citep{delon2020wasserstein}, but our guarantees concern $J_K$ for a chosen $K$. Deriving rates for $\varepsilon_K$ remains future work.
Centered differences require $2d_\theta$ auxiliary particle systems. Finally, Transport is comparable to REINFORCE on cybersecurity and below MF-REINFORCE on advertising. Future work includes adaptive scale selection and sensitivity estimators whose cost does not grow with $d_\theta$.

Potential extensions include non-exchangeable mean-field systems \citep{de2026mean,mekkaoui2026non} and graphon interactions \citep{bayraktar2023graphon,coppini2025nonlinear,cao2025probabilistic}. This replaces $\Pc(\Xc)$ by $\Pc_{\nu}(I\times \Xc)$, the space of probability laws with prescribed label marginal $\nu$ on a compact label space $I$ \citep{mekkaoui2026learning}; transport perturbations would then act on the state coordinate while preserving labels. Another direction is multi-agent reinforcement learning \citep{yang2018mean}, particularly cooperative systems approximated by mean-field control \citep{gu2021mean}, where population randomization could account for the effect of policy updates on the collective dynamics.

\subsection*{Acknowledgements}

S. Mekkaoui is supported by the SoGé Chair ``Risques Financiers'', and the
``Deep Finance and Statistics'' Qube-RT Chair.

Y. Hafsi acknowledges support from the
Chaire Risques Financiers, Société Générale,
at École Polytechnique, and from the
Institut Europlace de Finance (IEF).

H. Pham is supported by the SoGé Chair
``Risques Financiers'', by FiME
(Laboratory of Finance and Energy Markets),
and the EDF--CACIB Chair
``Finance and Sustainable Development''.

\bibliography{bibliography}

\clearpage
\appendix
\thispagestyle{empty}
\onecolumn

\section{Problem formulation}
\label{appendix:problem-formulation}

We present in this section the formal setup used throughout the paper: the canonical construction of the controlled process, the well-posedness of the control problem, and the decomposition of the policy gradient.

\paragraph{Notations.}
For a Polish space $E$, let $\Bc(E)$ be its Borel $\sigma$-algebra and $\Pc(E)$ its Borel probability measures, equipped with the weak topology. For a measurable map $F :E \to E'$, $F\sharp\mu$ denotes the pushforward of $\mu \in \Pc(E)$ by $F$; in particular, $\Lc(X) = X\sharp\P$ whenever $X : \Omega \to E $ is measurable and $\P \in \Pc(\Omega)$. We write $\Pc_2(\R^d)$ for probability measures with finite second moment. 
Let $\Cc_c^1(D)$ denote the space of continuously differentiable functions with compact support over $D \subset \R^m$.
We use $|\cdot|$ for absolute values, $\|\cdot\|$ for Euclidean norms and operator norms of matrices, $\|\cdot\|_\mathrm{F}$ for the Frobenius norm, and $\|v\|_1 = \sum_i|v_i|$. Define
\begin{align}
    d_\mathrm{TV}(\mu,\nu) \coloneqq \sup_{A \in \Bc(E)} |\mu(A) - \nu(A)|,\qquad
    \Wc_2^2(\mu,\nu) \coloneqq \inf_{\gamma \in \Pi(\mu,\nu)} \int_{\R^d\times\R^d} \|x-y\|^2\gamma(\d x,\d y),
\end{align}
where the second definition applies to $\mu,\nu\in\Pc_2(\R^d)$ and $\Pi(\mu,\nu)$ is their set of couplings. On a finite space, $d_\mathrm{TV}(\mu,\nu) = \frac{1}{2}\|\mu-\nu\|_1$ \citep[Proposition~4.2]{levin2026markov}. Let $\Delta_N = \{p\in[0,1]^N : \sum_i p_i=1\}$ with relative interior $\overset{\circ}{\Delta}_N \coloneqq \{p\in\Delta_N: p_i>0\ \forall i\}$ and relative boundary $\partial\Delta_N\coloneqq\Delta_N\setminus\overset{\circ}{\Delta}_N=\{p\in\Delta_N:p_i=0\text{ for at least one }i\}$. Finally, $\S^d$ denotes the set of symmetric  matrices over $\R^{d \times d}$ while $\S_{++}^d$ denotes the symmetric positive-definite ones. We denote $\Nc(m,\Sigma)$ for the Gaussian law with mean $m$ and covariance $\Sigma$, and $\Uc([0,1])$ for the uniform law on $[0,1]$.
Given two matrices $A,B \in \S^{d}$, we say that $A \succ B$ if $A-B \in \S^{d}_{++}$.
Given an open set \(O \subset \R^p\) and a differentiable function \(f:O\to\R\), we denote by \(\nabla f:O\to\R^p\) its gradient. More generally, if \(f:O_1\times O_2\to\R\), we denote by \(\nabla_x f\) and \(\nabla_y f\) its gradient maps with respect to the first and second variables, respectively. 

\subsection{Canonical construction of the control problem}
\label{appendix:canonical-construction}

Fix a finite integer $T\geq1$ and set $\Tc\coloneqq\{0,\ldots,T\}$. Throughout, $\Xc$ and $\Ac$ are non-empty Polish spaces, $\Pc(\Xc)$ is endowed with the topology of weak convergence and its Borel $\sigma$-algebra, and the kernels $P_t:\Xc\times\Pc(\Xc)\times\Ac\to\Pc(\Xc)$ and $\pi_t:\Xc\times\Pc(\Xc)\to\Pc(\Ac)$, $t\in\llbracket0,T-1\rrbracket$, are measurable. By \citet[Lemma~2.22 (kernels and randomization)]{kallenberg1997foundations}, there exist measurable maps
\begin{align}
    F_t:\Xc\times\Pc(\Xc)\times\Ac\times[0,1]\to\Xc,
    \qquad
    \widetilde F_t:\Xc\times\Pc(\Xc)\times[0,1]\to\Ac,
\end{align}
such that $F_t(x,m,a,\cdot)\sharp\,\Uc([0,1])=P_t(x,m,a)$ and $\widetilde F_t(x,m,\cdot)\sharp\,\Uc([0,1])=\pi_t(x,m)$ for all $(t,x,m,a)\in\llbracket0,T-1\rrbracket\times\Xc\times\Pc(\Xc)\times\Ac$. Let also $E$ be a Polish space and $\rho\in\Pc(E)$, used below to randomize the measure argument.

\paragraph{Canonical space.} Let $\mu_0 \in\Pc(\Xc)$. Set
\begin{align}
    \Omega\coloneqq\Xc\times[0,1]^T\times[0,1]^T\times E^{T},
\end{align}
with canonical element $\omega=(x_0,\epsilon_{1:T},\widetilde\epsilon_{1:T},e_{1:T})$, endowed with its product Borel $\sigma$-algebra $\Fc$ and the product measure
\begin{align}
    \P(\d\omega)\coloneqq\mu_0(\d x_0)\otimes\Uc([0,1])^{\otimes T}(\d\epsilon_{1:T})\otimes\Uc([0,1])^{\otimes T}(\d\widetilde\epsilon_{1:T})\otimes\rho^{\otimes(T)}(\d e_{1:T}).
\end{align}
The coordinate maps $\xi(\omega)=x_0$, $\epsilon_t(\omega)=\epsilon_t$, $\widetilde\epsilon_t(\omega)=\widetilde\epsilon_t$, $R_t(\omega)=e_t$ are then mutually independent, with $\xi\sim\mu_0$, $\epsilon_t,\widetilde\epsilon_t\sim\Uc([0,1])$ and $R_t\sim\rho$.

\paragraph{State and control processes.} Define $X_t:\Omega\to\Xc$ for $t\in\Tc$ and $\alpha_t:\Omega\to\Ac$ for $t\in\llbracket0,T-1\rrbracket$ recursively by
\begin{equation}
    \begin{cases}
        X_0=\xi,\\
        \alpha_t=\widetilde F_t\big(X_t,\P_{X_t},\widetilde\epsilon_{t+1}\big),\\
        X_{t+1}=F_t\big(X_t,\P_{X_t},\alpha_t,\epsilon_{t+1}\big),
    \end{cases}
    \qquad t\in\llbracket0,T-1\rrbracket.
    \label{eq:canonical-construction-rollout}
\end{equation}
This is well defined since $\P_{X_t}$ is determined by the construction up to time $t$. Let $\Fc_t\coloneqq\sigma(\xi,\epsilon_{1:t},\widetilde\epsilon_{1:t})$. Since $\widetilde\epsilon_{t+1}$ is independent of $\Fc_t$ and $\epsilon_{t+1}$ is independent of $\Fc_t\vee\sigma(\widetilde\epsilon_{t+1})$, we get, $\P$-a.s.,
\begin{align}
    \Lc\big(\alpha_t\mid\Fc_t\big)=\pi_t(X_t,\P_{X_t}), \qquad  \Lc\big(X_{t+1}\mid\Fc_t\vee\sigma(\widetilde\epsilon_{t+1})\big)=P_t(X_t,\P_{X_t},\alpha_t).
\end{align}

\paragraph{Randomized measure argument.} 
Fix $\theta\in\Theta$ and let $(X_t^\theta)_{t\in\Tc}$ denote the unperturbed process associated with $\pi^\theta$, with population flow $\mu_t^\theta \coloneqq \P_{X_t^\theta}$. We set $M_0^{\lambda,\theta}\coloneqq \mu_0^\theta = \mu_0$. For $\lambda\in[0,1)$ and $t\in\llbracket1,T\rrbracket$, define the $\Pc(\Xc)$-valued random variable
\begin{align}
    \Omega\ni\omega\mapsto M_t^{\lambda,\theta}(\omega)\coloneqq T_t^{\lambda}(\omega,\cdot,\cdot)\sharp\big(\Gc(\mu_t^\theta)\otimes\Uc([0,1])\big),
\end{align}
where $T_t^{\lambda}(\omega,x,u)\coloneqq\Phi^\lambda\big(R_t(\omega),x,u\big)$ for some measurable maps $\Phi^\lambda:E\times\Xc\times[0,1]\to\Xc$ and $\Gc:\Pc(\Xc)\to\Pc(\Xc)$. For every bounded continuous function $f:\Xc\to\R$,
\begin{align}
    \int_\Xc f(y)\,M_t^{\lambda,\theta}(\omega)(\d y)
    = \int_\Xc\int_{[0,1]} f\big(\Phi^\lambda(R_t(\omega),x,u)\big)\Uc([0,1])(\d u)\Gc(\mu_t^\theta)(\d x)
\end{align}
is $\sigma(R_t)$-measurable by Fubini's theorem. Since the Borel $\sigma$-algebra of $\Pc(\Xc)$ is generated by the maps $m\mapsto\int_\Xc f\,\d m$, $M_t^{\lambda,\theta}$ is a $\sigma(R_t)$-measurable random probability measure.

\paragraph{Perturbed state and control processes.}
For fixed $\theta\in\Theta$, apply the preceding randomization construction to the policy $\pi^\theta$ and denote by $\widetilde F_t^\theta:\Xc\times\Pc(\Xc)\times[0,1]\to\Ac$ a measurable map satisfying $\widetilde F_t^\theta(x,m,\cdot)\sharp\Uc([0,1]) = \pi_t^\theta(\cdot\mid x,m)$. Define $X_t^{\lambda,\theta} : \Omega\to\Xc$ for $t\in\Tc$ and $\alpha_t^{\lambda,\theta}:\Omega\to\Ac$ for $t\in\llbracket0,T-1\rrbracket$ recursively by
\begin{align}
    \begin{cases}
        X_0^{\lambda,\theta}=\xi,\\
        \alpha_t^{\lambda,\theta} = \widetilde F_t^\theta\big(X_t^{\lambda,\theta}, M_t^{\lambda,\theta},\widetilde\epsilon_{t+1}\big),\\
        X_{t+1}^{\lambda,\theta} = F_t\big(X_t^{\lambda,\theta}, M_t^{\lambda,\theta}, \alpha_t^{\lambda,\theta},\epsilon_{t+1}\big)
    \end{cases}
    \qquad t\in\llbracket0,T-1\rrbracket.
\end{align}
As for the unperturbed process, conditionally on $(X_t^{\lambda,\theta},M_t^{\lambda,\theta})$, the action $\alpha_t^{\lambda,\theta}$ has law $\pi_t^\theta(\cdot\mid X_t^{\lambda,\theta},M_t^{\lambda,\theta})$ and, given also $\alpha_t^{\lambda,\theta}$, $X_{t+1}^{\lambda,\theta}$ has law $P_t(\cdot\mid X_t^{\lambda,\theta},M_t^{\lambda,\theta},\alpha_t^{\lambda,\theta})$. For $t\geq 1$, $X_t^{\lambda,\theta}$ depends on the perturbation randomizers only through $R_1,\ldots,R_{t-1}$, whereas $M_t^{\lambda,\theta}$ depends only on $R_t$; the two are therefore independent. In general, $M_t^{\lambda,\theta} \neq \Lc(X_t^{\lambda,\theta})$, since $M_t^{\lambda,\theta}$ is a randomized population argument supplied to the policy and transition kernel rather than the law induced by the perturbed state process.

\paragraph{Simulation access.}
Given $(t,x,m,a)\in\llbracket0,T-1\rrbracket\times\Xc\times\Pc(\Xc)\times\Ac$, the simulator returns the reward $r_t(x,m,a)$ and an independent draw from $P_t(\cdot\mid x,m,a)$. At the terminal time, given $(x,m)\in\Xc\times\Pc(\Xc)$, it returns $g(x,m)$. The transition kernels and reward functions are otherwise unknown. The policy $\pi_t^\theta$ and its score $\nabla_\theta\log p_t^\theta$ are known and can be evaluated at prescribed population arguments. Population laws are estimated from particles whenever they are not available exactly.

\subsection{Well-posedness of the control problem}
\label{appendix:well-posedness}

The state space $\Xc$ is either a finite subset of $\R^d$ or $\R^d$ itself. The action space $\Ac$ is a non-empty closed subset of $\R^q$. Policies are parametrized by $\theta\in\Theta$: we denote by $(X_t^\theta, \alpha_t^\theta)$ the processes defined by~\eqref{eq:canonical-construction-rollout} with $\pi_t = \pi_t^\theta$, set $\mu_t^\theta \coloneqq \P_{X_t^\theta}$, and define the objective $J : \Theta \to \R$ as
\begin{align}
    J(\theta) \coloneqq \E\Big[\sum_{t=0}^{T-1} r_t\big(X_t^\theta, \mu_t^\theta,\alpha_t^\theta\big) + g\big(X_T^\theta,\mu_T^\theta\big) \Big].
    \label{eq:objective}
\end{align}
We first give the finite-state argument, then treat the moment and continuity requirements in continuous state space.

\subsubsection{Finite state space}

Let $\Xc=\{x_1,\ldots,x_N\}$ and identify its probability laws with $\Delta_N$. For fixed $\theta$, the policy-averaged kernel $K_t^\theta:\Xc\times\Delta_N\to\Delta_N$ and reward $R_t^\theta:\Xc\times\Delta_N\to\R$ are
\begin{align}
    K_t^\theta(x_j\mid x_i,m)&\coloneqq\int_\Ac P_t(\{x_j\}\mid x_i,m,a)\pi_t^\theta(\d a\mid x_i,m),\nonumber\\
    R_t^\theta(x_i,m)&\coloneqq\int_\Ac r_t(x_i,m,a)\pi_t^\theta(\d a\mid x_i,m).
    \label{eq:discrete-well-posedness-averages}
\end{align}

\begin{Assumption}[Finite-state well-posedness]
\label{assumption:discrete-well-posedness}
The initial law $\mu_0\in\Delta_N$ is independent of $\theta$, and the policies and transition kernels are measurable. The rewards satisfy $|r_t(x_i,m,a)|\leq\overline r$ and $|g(x_i,m)|\leq\overline g$ uniformly in their arguments. For every $t,i,j \in \llbracket 0, T-1 \rrbracket \times \llbracket 1, N \rrbracket \times \llbracket 1, N \rrbracket$, the maps $(\theta,m)\mapsto K_t^\theta(x_j\mid x_i,m)$ and $(\theta,m)\mapsto R_t^\theta(x_i,m)$ are continuous on $\Theta\times\Delta_N$, and $m\mapsto g(x_i,m)$ is continuous on $\Delta_N$.
\end{Assumption}

\begin{Proposition}
\label{prop:discrete-well-posedness}
Under Assumption~\ref{assumption:discrete-well-posedness}, the original process and its finite-state perturbation are uniquely defined in law. Their objectives $J$ and $J^\lambda$, $\lambda\in[0,1)$, are continuous on $\Theta$ and bounded in absolute value by $T\overline r+\overline g$. In particular, they attain their maxima when $\Theta$ is non-empty and compact.
\end{Proposition}

\begin{proof}
Starting from $\mu_0$, compute successively
\begin{align}
    \mu_{t+1}^\theta(j)=\sum_{i=1}^N\mu_t^\theta(i)K_t^\theta(x_j\mid x_i,\mu_t^\theta),
    \qquad j\in\llbracket1,N\rrbracket.
    \label{eq:discrete-well-posedness-flow}
\end{align}
Each right-hand side is nonnegative, and summing over $j$ gives one. This recursion therefore determines a unique flow in $\Delta_N$. Substituting that flow into the policy and transition kernels determines the joint law of states and actions by the canonical construction. To check continuity, let $\theta_n\to\theta$. The claim is immediate at $t=0$, since $\mu_0^{\theta_n}=\mu_0^\theta=\mu_0$. Suppose that $\mu_t^{\theta_n}(i)\to\mu_t^\theta(i)$ for every $i$. By the continuity of $K_t^\theta$ and the fact that the sum is finite, for every $j$,
\begin{align}
    \mu_{t+1}^{\theta_n}(j)
    &=\sum_{i=1}^N\mu_t^{\theta_n}(i)K_t^{\theta_n}(x_j\mid x_i,\mu_t^{\theta_n})
    \longrightarrow
    \sum_{i=1}^N\mu_t^\theta(i)K_t^\theta(x_j\mid x_i,\mu_t^\theta)
    =\mu_{t+1}^\theta(j).
\end{align}
Thus induction gives continuity of $\theta\mapsto\mu_t^\theta$ for every $t$. Consequently,
\begin{align}
    J(\theta)=\sum_{t=0}^{T-1}\sum_{i=1}^N\mu_t^\theta(i)R_t^\theta(x_i,\mu_t^\theta)
    +\sum_{i=1}^N\mu_T^\theta(i)g(x_i,\mu_T^\theta)
\end{align}
is continuous and has the stated bound.

For the perturbed process, fix a realization $q=(q_1,\ldots,q_T)\in(\Delta_N)^T$ of the randomizers and write
\begin{align}
    M_0^{\lambda,\theta}(q)&\coloneqq\mu_0, &
    M_t^{\lambda,\theta}(q)&\coloneqq(1-\lambda)\mu_t^\theta+\lambda q_t,\quad t\geq1.
\end{align}
These supplied population arguments are continuous in $\theta$ by the result just proved. Conditional on $q$, define the state-law recursion
\begin{align}
\begin{cases}
    \widetilde\mu_0^{\lambda,\theta,q}&\coloneqq\mu_0,\\
    \widetilde\mu_{t+1}^{\lambda,\theta,q}(j)
    &\coloneqq\sum_{i=1}^N\widetilde\mu_t^{\lambda,\theta,q}(i)
    K_t^\theta\big(x_j\mid x_i,M_t^{\lambda,\theta}(q)\big).
\end{cases}
\end{align}
The same induction as above, now using continuity of both $\theta\mapsto M_t^{\lambda,\theta}(q)$ and $(\theta,m)\mapsto K_t^\theta(\cdot\mid x_i,m)$, shows that $\widetilde\mu_t^{\lambda,\theta_n,q}\to\widetilde\mu_t^{\lambda,\theta,q}$ for every $t$. Hence the conditional objective
\begin{align}
    \mathcal J^\lambda(\theta,q)
    \coloneqq\sum_{t=0}^{T-1}\sum_{i=1}^N\widetilde\mu_t^{\lambda,\theta,q}(i)
    R_t^\theta\big(x_i,M_t^{\lambda,\theta}(q)\big) +\sum_{i=1}^N\widetilde\mu_T^{\lambda,\theta,q}(i) g\big(x_i,M_T^{\lambda,\theta}(q)\big)
\end{align}
is continuous in $\theta$. Moreover, $|\mathcal J^\lambda(\theta,q)|\leq T\overline r+\overline g$ for every $(\theta,q)$. Since $J^\lambda(\theta)=\mathbb E_q[\mathcal J^\lambda(\theta,q)]$ and the law of $q$ is independent of $\theta$, dominated convergence gives continuity of $J^\lambda$. The final claim follows from compactness.
\end{proof}

\subsubsection{Continuous state space}

We now take $\Xc = \R^d$. Equip $\Pc_2(\Xc)$ and $\Pc_2(\Ac)$ with the Wasserstein distance $\Wc_2$. For any $m\in\Pc_2(\Xc)$, we also denote $M_2(m) \coloneqq \int_\Xc \|x\|^2 m(\d x)$. The following assumptions ensure finite second moments and a finite, continuous objective.

\begin{Assumption}
    \label{assumption:well-posedness}
    The parameter set $\Theta \subset \R^{d_\theta}$ is non-empty, and the initial law $\mu_0 \in \Pc_2(\Xc)$ is independent of $\theta$. For each $t \in \llbracket0,T-1\rrbracket$, the following conditions hold.
    \begin{enumerate}
        \item \emph{Continuity of the kernels.} For every sequence $(\theta_n,x_n,m_n,a_n) \to (\theta,x,m,a)$ in $\Theta\times \Xc\times\Pc_2(\Xc) \times \Ac$,
        \begin{align}
            \Wc_2\big(\pi_t^{\theta_n}(\cdot\mid x_n,m_n), \pi_t^\theta(\cdot\mid x,m)\big) \longrightarrow 0,\qquad
            \Wc_2\big(P_t(\cdot\mid x_n,m_n,a_n), P_t(\cdot\mid x,m,a)\big) \longrightarrow0.
        \end{align}
        \item \emph{Continuity of the rewards.} The maps $r_t : \Xc\times \Pc_2(\Xc) \times \Ac \to \R$ and $g:\Xc\times\Pc_2(\Xc) \to \R$ are jointly continuous. 
        \item \emph{Quadratic growth.} There exists $C\geq 0$ such that, for all $(\theta,x,m,a) \in \Theta\times\Xc\times\Pc_2(\Xc)\times\Ac$,
        \begin{align}
            \int_\Ac \|b\|^2 \pi_t^\theta(\d b\mid x,m) &\leq C\big(1+\|x\|^2 + M_2(m)\big),\\
            \int_\Xc \|y\|^2 P_t(\d y\mid x,m,a) &\leq C\big(1+\|x\|^2 + \|a\|^2 + M_2(m)\big),\\
            |r_t(x,m,a)| &\leq C\big(1+\|x\|^2 + \|a\|^2 + M_2(m)\big),\\
            |g(x,m)| &\leq C\big(1+\|x\|^2+M_2(m)\big).
        \end{align}
    \end{enumerate}
\end{Assumption}

\begin{Proposition}
    Under Assumption~\ref{assumption:well-posedness}, $\mu_t^\theta \in \Pc_2(\Xc)$ for all $(\theta,t) \in \Theta\times\Tc$, and
    \begin{align}
        \sup_{\theta\in\Theta}\big(\max_{t\in\Tc} \E\big[\|X_t^\theta\|^2\big] + \max_{t\in\llbracket0,T-1\rrbracket} \E\big[\|\alpha_t^\theta\|^2 \big]\big) < \infty.
    \end{align}
    Moreover, for every $t\in\Tc$, the map $\theta\mapsto\mu_t^\theta$ is continuous from $\Theta$ to $(\Pc_2(\Xc), \Wc_2)$, and the objective $J : \Theta\to\R$ is finite, bounded and continuous on $\Theta$.
\end{Proposition}

\begin{proof}
    Fix $\theta\in\Theta$. The canonical construction yields, for $t\in\llbracket0,T-1\rrbracket$ and $B\in\Bc(\Xc)$,
    \begin{align}
        \label{eq:well-posedness-probability-flow}
        \mu_{t+1}^\theta(B) = \int_\Xc \int_\Ac P_t(B\mid x,\mu_t^\theta,a)\pi_t^\theta(\d a\mid x,\mu_t^\theta)\mu_t^\theta(\d x).
    \end{align}
    Starting from $\mu_0^\theta = \mu_0$, equation~\eqref{eq:well-posedness-probability-flow} determines the probability flow $(\mu_t^\theta)_{t\in\Tc}$.
    
    The joint law of the states and actions is then
    \begin{align}
        \mu_0(\d x_0) \prod_{t=0}^{T-1} \big[\pi_t^\theta(\d a_t\mid x_t,\mu_t^\theta) P_t(\d x_{t+1}\mid x_t,\mu_t^\theta, a_t) \big].
    \end{align}

    We prove by induction that each $\mu_t^\theta$ has finite second moment. Whenever $\mu_t^\theta\in\Pc_2(\Xc)$, write $u_t^\theta\coloneqq M_2(\mu_t^\theta)$. Assumption~\ref{assumption:well-posedness} gives
    \begin{align}
        \E\big[\|\alpha_t^\theta\|^2 \big] \leq C(1+2u_t^\theta),\qquad
        u_{t+1}^\theta \leq C\big(1+2u_t^\theta+C(1+2u_t^\theta)\big).
    \end{align}
    Thus $\mu_{t+1}^\theta \in \Pc_2(\Xc)$ and the second inequality bounds $u_{t+1}^\theta$. Since $u_0^\theta = M_2(\mu_0)$ is independent of $\theta$, $C$ is uniform in $t$ and $\theta$, and $T$ is finite, induction gives the uniform second-moment bounds. Moreover,
    \begin{align}
        \E\big[\big|r_t(X_t^\theta, \mu_t^\theta,\alpha_t^\theta)\big| \big] \leq C\big(1 + 2u_t^\theta + \E\big[\|\alpha_t^\theta\|^2 \big]\big)<\infty,\qquad
        \E\big[\big|g(X_T^\theta,\mu_T^\theta)\big|\big] \leq C\big(1 + 2u_T^\theta\big) < \infty.
    \end{align}
    Hence $J(\theta)$ is finite and the uniform bound over $\theta\in\Theta$ follows from the uniform constant $C$.

    To prove continuity, let $\theta_n \to \theta$ in $\Theta$. We prove $\mu_t^{\theta_n} \to \mu_t^\theta$ in $\Wc_2$ by induction on $t$. The case $t=0$ is immediate because all initial laws equal $\mu_0$. Suppose the claim holds at some $t<T$.

    We use the characterization of convergence in the Wasserstein space in \citet[Definition~6.6 and Theorem~6.7]{villani2009optimal}. Fix a continuous function $f : \Xc \to \R$ satisfying $|f(y)|\leq C_f(1+\|y\|^2)$ for some $C_f<\infty$. Applying this characterization successively to $P_t$ and $\pi_t^\theta$, together with their continuity and quadratic-growth bounds, shows that the map
    \begin{align}
        \Theta \times \Xc \times \Pc_2(\Xc) \ni (\theta,x,m) \longmapsto h(\theta,x,m) \coloneqq \int_\Ac \int_\Xc f(y) P_t(\d y\mid x,m,a) \pi_t^\theta(\d a\mid x,m) \in \R
    \end{align}
    is continuous and satisfies
    \begin{align}
        \label{eq:well-posedness-growth-h}
        |h(\theta,x,m)| \leq C_f\Big(1 + C(1+C)\big(1+\|x\|^2 + M_2(m)\big)\Big). 
    \end{align}
    Passage to the limit follows by uniform convergence on compact sets and uniform integrability. Indeed, by~\eqref{eq:well-posedness-probability-flow},
    \begin{align}
        \int_\Xc f\d\mu_{t+1}^{\theta_n} - \int_\Xc f\d\mu_{t+1}^\theta = \int_\Xc\big(h(\theta_n, x,\mu_t^{\theta_n}) - h(\theta,x,\mu_t^\theta)\big) \mu_t^{\theta_n}(\d x) + \bigg(\int_\Xc h(\theta,x,\mu_t^\theta)\mu_t^{\theta_n}(\d x) - \int_\Xc h(\theta,x,\mu_t^\theta) \mu_t^\theta(\d x)\bigg).
    \end{align}
    The second term vanishes by the induction hypothesis and the characterization, since $h(\theta,\cdot,\mu_t^\theta)$ is continuous with quadratic growth by~\eqref{eq:well-posedness-growth-h}. For the first term, fix $R>0$. On the compact set $\Xc\cap\{\|x\|\leq R\}$, $h(\theta_n,\cdot,\mu_t^{\theta_n})\to h(\theta,\cdot,\mu_t^\theta)$ uniformly, because $h$ is uniformly continuous on the compact set $(\{\theta_n\}_n\cup\{\theta\})\times(\Xc\cap\{\|x\|\leq R\})\times(\{\mu_t^{\theta_n}\}_n\cup\{\mu_t^\theta\})$. On $\{\|x\|>R\}$, the bound~\eqref{eq:well-posedness-growth-h} and $\sup_nM_2(\mu_t^{\theta_n})<\infty$ give $|h(\theta_n,x,\mu_t^{\theta_n})-h(\theta,x,\mu_t^\theta)|\le K(1+\|x\|^2)$ for some $K$ independent of $n$. Hence
    \begin{align}
        \limsup_{n\to\infty} \bigg|\int_\Xc \big(h(\theta_n,x,\mu_t^{\theta_n})-h(\theta,x,\mu_t^\theta)\big) \mu_t^{\theta_n}(\d x)\bigg|
        \leq K \limsup_{n\to\infty} \int_{\{\|x\|>R\}}(1 + \|x\|^2) \mu_t^{\theta_n}(\d x),
    \end{align}
    and the right-hand side vanishes as $R\to\infty$ by the uniform integrability part of the characterization. Therefore $\int f\,\d\mu_{t+1}^{\theta_n}\to\int f\,\d\mu_{t+1}^\theta$. Since $f$ is arbitrary, the characterization yields $\mu_{t+1}^{\theta_n}\to\mu_{t+1}^\theta$ in $\Wc_2$, which completes the induction.

    Finally, for $t\in\llbracket0,T-1\rrbracket$,
    \begin{align}
        \E\big[r_t(X_t^\theta,\mu_t^\theta,\alpha_t^\theta)\big]=\int_\Xc\int_\Ac r_t(x,\mu_t^\theta,a)\,\pi_t^\theta(\d a\mid x,\mu_t^\theta)\,\mu_t^\theta(\d x).
    \end{align}
    Since $r_t$ is continuous with quadratic growth by Assumption~\ref{assumption:well-posedness}, the argument above, with $r_t(x,m,a)$ in place of $\int_\Xc f(y)P_t(\d y\mid x,m,a)$, shows that $(\theta,x,m)\mapsto\int_\Ac r_t(x,m,a)\,\pi_t^\theta(\d a\mid x,m)$ is continuous and bounded by $C(1+C)(1+\|x\|^2+M_2(m))$. The same splitting then gives $\E[r_t(X_t^{\theta_n},\mu_t^{\theta_n},\alpha_t^{\theta_n})]\to\E[r_t(X_t^\theta,\mu_t^\theta,\alpha_t^\theta)]$. The terminal term is handled identically with $g(\cdot,\mu_T^{\theta_n})$ in place of $h(\theta_n,\cdot,\mu_t^{\theta_n})$. Hence $J(\theta_n)\to J(\theta)$, and $J$ is continuous on $\Theta$.
\end{proof}

\section{Proofs}
\label{sec:proofs}
\subsection{Common definitions}
\label{subsec:proof-common-definitions}

Under Assumption~\ref{assumption:well-posedness} (or Assumption~\ref{assumption:discrete-uniform-regularity} below), the following maps are well defined. For $t \in \llbracket 0,T-1 \rrbracket $ and $\theta\in\Theta$, define the policy-averaged transition kernel and reward by
\begin{align}
    K_t^\theta &: \Xc\times\Pc_2(\Xc)\longrightarrow\Pc_2(\Xc),
    &K_t^\theta(\d x'\mid x,m)
    &\coloneqq\int_\Ac P_t(\d x'\mid x,m,a)\pi_t^\theta(\d a\mid x,m),\nonumber\\
    R_t^\theta &: \Xc\times\Pc_2(\Xc)\longrightarrow\R,
    &R_t^\theta(x,m)
    &\coloneqq\int_\Ac r_t(x,m,a)\pi_t^\theta(\d a\mid x,m).
    \label{eq:common-averaged-kernel-reward}
\end{align}
Thus $K_t^\theta(\cdot\mid x,m)$ is the conditional law of the next state given the current state $x$ and population argument $m$, after averaging over the randomized action. Define also
\begin{align}
    \Phi_t^\theta &: \Pc_2(\Xc)\times\Pc_2(\Xc)\longrightarrow\Pc_2(\Xc),
    &\Phi_t^\theta(\nu,m)
    &\coloneqq\int_\Xc K_t^\theta(\cdot\mid x,m)\nu(\d x),\nonumber\\
    \overline R_t^\theta &: \Pc_2(\Xc)\times\Pc_2(\Xc)\longrightarrow\R,
    &\overline R_t^\theta(\nu,m)
    &\coloneqq\int_\Xc R_t^\theta(x,m)\nu(\d x),\nonumber\\
    \overline g &: \Pc_2(\Xc)\times\Pc_2(\Xc)\longrightarrow\R,
    &\overline g(\nu,m)
    &\coloneqq\int_\Xc g(x,m)\nu(\d x).
    \label{eq:common-population-maps}
\end{align}
The first argument is the state law being integrated, while the second is the population argument supplied to the policy, transition kernel, and rewards.

\subsection{Finite-state space}
\label{appendix:finite-proofs}

This appendix gives the assumptions and proofs for Section~\ref{subsec:finite}. Simplex derivatives and densities use the coordinates $(m_1,\ldots,m_{N-1})$, with $m_N=1-\sum_{k<N}m_k$.

\subsubsection{Perturbation estimate}
\label{proof:discrete-perturbation-estimate}

\begin{proof}[Proof of Proposition~\ref{prop:discrete-perturbation-estimate}]
    At $t=0$, the distance is zero because $M_0^{\lambda,\theta} = \mu_0$. For $t\geq1$, since $M_t^{\lambda,\theta}(\omega)=(1-\lambda)\mu_t^\theta+\lambda Q_t(\omega)$ for any $\omega \in \Omega$, and $d_\mathrm{TV}(\mu,\nu) \leq 1$ for any $\mu,\nu\in\Pc(\Xc)$, it follows that 
    \begin{align}
        d_{\mathrm{TV}}(M_t^{\lambda,\theta}(\omega),\mu_t^\theta)
        =\lambda d_{\mathrm{TV}}(Q_t(\omega),\mu_t^\theta)\leq\lambda, \text{ for any $\omega \in \Omega$.}
    \end{align}
\end{proof}

\subsubsection{Perturbation consistency}
\label{proof:discrete-perturbation-consistency}

\begin{Assumption}[Boundedness and Lipschitz continuity]
\label{assumption:discrete-uniform-regularity}
The maps $r_t$ and $g$ are uniformly bounded in all their arguments by $\overline R$ and $\overline g$, respectively. For all admissible $t,\theta,i,m,m'$, we assume that the following bounds hold with constants independent of these arguments:
    \begin{align}
    \begin{cases}
d_{\mathrm{TV}}\bigl(K_t^\theta(\cdot\mid i,m),K_t^\theta(\cdot\mid i,m')\bigr)
        &\leq L_K d_{\mathrm{TV}}(m,m'),\nonumber\\
        |R_t^\theta(i,m)-R_t^\theta(i,m')|&\leq L_R d_{\mathrm{TV}}(m,m'),\nonumber\\
        |g(i,m)-g(i,m')|&\leq L_g d_{\mathrm{TV}}(m,m').
    \end{cases}
    \end{align}
\end{Assumption}

\begin{Lemma}[Stability of the state law]
    \label{lemma:discrete-stability-state-marginal}
    Under Assumption~\ref{assumption:discrete-uniform-regularity}, the law $\nu_t^{\lambda,\theta}=\Lc(X_t^{\lambda,\theta})$ satisfies
    \begin{align}
        d_{\mathrm{TV}}(\nu_t^{\lambda,\theta},\mu_t^\theta)\leq L_Kt\lambda,
        \qquad t \in \llbracket 1,T \rrbracket.
    \end{align}
\end{Lemma}

\begin{proof}
    For $t\in\llbracket0,T-1\rrbracket$, define the averaged kernel $\overline K_t^{\lambda,\theta} : \Xc \to \Pc(\Xc)$ by
    $\overline K_t^{\lambda,\theta}(\cdot\mid x) \coloneqq \E[K_t^\theta(\cdot\mid x,M_t^{\lambda,\theta})]$. By definition of $K_t^\theta$, conditionally on $(X_t^{\lambda,\theta}, M_t^{\lambda,\theta}) = (x,m)$, the next state has law $K_t^\theta(\cdot\mid x,m)$. For $t\geq1$, $X_t^{\lambda,\theta}$ depends on the perturbation randomizers only through $U_1,\ldots,U_{t-1}$, whereas $M_t^{\lambda,\theta}$ depends on the randomizer $U_t$. Hence $X_t^{\lambda,\theta}$ and $M_t^{\lambda,\theta}$ are independent. At $t=0$, $M_0^{\lambda,\theta}=\mu_0$ is deterministic. Therefore,
    \begin{align}
        \Lc\big(X_{t+1}^{\lambda,\theta} \mid X_t^{\lambda,\theta}=x\big) = \overline K_t^{\lambda,\theta}(\cdot\mid x).
    \end{align}
    It follows that $\nu_{t+1}^{\lambda,\theta} = \nu_t^{\lambda,\theta} \overline K_t^{\lambda,\theta}$. On the other hand, the unperturbed population flow satisfies
    \begin{align}
        \label{eq:population-recursion}
        \mu_{t+1}^\theta = \mu_t^\theta K_t^\theta(\cdot\mid \cdot,\mu_t^\theta).
    \end{align}

    For a probability measure $\nu$ and a Markov kernel $K$, we use the notation $(\nu K)(i):=\sum_{j=1}^N K(i\mid j)\nu(j)$. We bound $d_{\mathrm{TV}}(\nu_{t+1}^{\lambda,\theta},\mu_{t+1}^\theta)$ by inserting the intermediate law $\nu_t^{\lambda,\theta}K_t^\theta(\cdot\mid\cdot,\mu_t^\theta)$. By the triangle inequality,
    \begin{align}
        d_{\mathrm{TV}}(\nu_{t+1}^{\lambda,\theta},\mu_{t+1}^\theta)
        &\leq \underbrace{d_{\mathrm{TV}}\big(\nu_t^{\lambda,\theta}\overline K_t^{\lambda,\theta},\,\nu_t^{\lambda,\theta}K_t^\theta(\cdot\mid\cdot,\mu_t^\theta)\big)}_{(\mathrm{I})}
        +\underbrace{d_{\mathrm{TV}}\big(\nu_t^{\lambda,\theta}K_t^\theta(\cdot\mid\cdot,\mu_t^\theta),\,\mu_t^\theta K_t^\theta(\cdot\mid\cdot,\mu_t^\theta)\big)}_{(\mathrm{II})}.
        \label{eq:discrete-stability-split}
    \end{align}
    
    \emph{Bounding $(\mathrm{I})$:} Fix $j\in\llbracket 1,N\rrbracket$. Since $\overline K_t^{\lambda,\theta}(\cdot\mid j)=\E[K_t^\theta(\cdot\mid j,M_t^{\lambda,\theta})]$ and $\mu\mapsto d_{\mathrm{TV}}(\mu,\nu)=\sup_A|\mu(A)-\nu(A)|$ is a supremum of functionals affine in $\mu$ (hence convex), Jensen's inequality gives
    \begin{align}
        d_{\mathrm{TV}}\big(\overline K_t^{\lambda,\theta}(\cdot\mid j),\,K_t^\theta(\cdot\mid j,\mu_t^\theta)\big)
        &=d_{\mathrm{TV}}\Big(\E\big[K_t^\theta(\cdot\mid j,M_t^{\lambda,\theta})\big],\,K_t^\theta(\cdot\mid j,\mu_t^\theta)\Big)\nonumber\\
        &\leq \E\Big[d_{\mathrm{TV}}\big(K_t^\theta(\cdot\mid j,M_t^{\lambda,\theta}),\,K_t^\theta(\cdot\mid j,\mu_t^\theta)\big)\Big]\nonumber\\
        &\leq L_K\,\E\big[d_{\mathrm{TV}}(M_t^{\lambda,\theta},\mu_t^\theta)\big]
        \;\leq\;L_K\lambda,
        \label{eq:discrete-kernel-rows-close}
    \end{align}
    using Assumption~\ref{assumption:discrete-uniform-regularity} in the second-to-last step and the perturbation estimate of Proposition~\ref{prop:discrete-perturbation-estimate} in the last step. Since $\nu_t^{\lambda,\theta}$ is a probability measure on $\llbracket 1,N\rrbracket$, for any $A$
    \begin{align}
        \big|\nu_t^{\lambda,\theta}\overline K_t^{\lambda,\theta}(A)-\nu_t^{\lambda,\theta}K_t^\theta(\cdot\mid\cdot,\mu_t^\theta)(A)\big|
        &=\Big|\sum_{j=1}^N\nu_t^{\lambda,\theta}(j)\big(\overline K_t^{\lambda,\theta}(A\mid j)-K_t^\theta(A\mid j,\mu_t^\theta)\big)\Big|\nonumber\\
        &\leq\sum_{j=1}^N\nu_t^{\lambda,\theta}(j)\,d_{\mathrm{TV}}\big(\overline K_t^{\lambda,\theta}(\cdot\mid j),K_t^\theta(\cdot\mid j,\mu_t^\theta)\big)
        \;\leq\;L_K\lambda,
    \end{align}
    and taking the supremum over $A$ yields $(\mathrm{I})\leq L_K\lambda$.
    
    \emph{Bounding $(\mathrm{II})$:} Here the same kernel row $K_t^\theta(\cdot\mid\cdot,\mu_t^\theta)$ is applied to the two laws $\nu_t^{\lambda,\theta}$ and $\mu_t^\theta$, so it suffices to show the general contraction property
    \begin{align}
        d_{\mathrm{TV}}(\nu K,\mu K)\leq d_{\mathrm{TV}}(\nu,\mu)
        \label{eq:markov-kernel-contraction}
    \end{align}
    for any probability measures $\nu,\mu$ on $\llbracket 1,N\rrbracket$ and any Markov kernel $K$ on $\llbracket 1,N\rrbracket$. Let $A^\star=\{j:\nu(j)>\mu(j)\}$; since $\sum_j(\nu(j)-\mu(j))=0$, the standard identity \citep[Proposition~4.2 and Remark~4.3]{levin2026markov} $d_{\mathrm{TV}}(\nu,\mu)=\nu(A^\star)-\mu(A^\star)=\sum_{j\in A^\star}(\nu(j)-\mu(j))=-\sum_{j\notin A^\star}(\nu(j)-\mu(j))$ holds. Then, for any $A$,
    \begin{align}
        (\nu K)(A)-(\mu K)(A)
        &=\sum_{j=1}^N\big(\nu(j)-\mu(j)\big)K(A\mid j)\nonumber\\
        &=\sum_{j\in A^\star}\big(\nu(j)-\mu(j)\big)K(A\mid j)+\sum_{j\notin A^\star}\big(\nu(j)-\mu(j)\big)K(A\mid j)\nonumber\\
        &\leq\sum_{j\in A^\star}\big(\nu(j)-\mu(j)\big)\cdot 1+\sum_{j\notin A^\star}\big(\nu(j)-\mu(j)\big)\cdot 0
        =d_{\mathrm{TV}}(\nu,\mu),
    \end{align}
    where the inequality uses $K(A\mid j)\leq 1$ on $A^\star$ (where $\nu(j)-\mu(j)\geq0$) and $K(A\mid j)\geq 0$ off $A^\star$ (where $\nu(j)-\mu(j)\leq 0$). Applying the same argument to $(\mu K)(A)-(\nu K)(A)$ gives the reverse bound, so $|(\nu K)(A)-(\mu K)(A)|\leq d_{\mathrm{TV}}(\nu,\mu)$ for every $A$, and \eqref{eq:markov-kernel-contraction} follows by taking the supremum over $A$. Applying \eqref{eq:markov-kernel-contraction} with $K=K_t^\theta(\cdot\mid\cdot,\mu_t^\theta)$ gives $(\mathrm{II})\leq d_{\mathrm{TV}}(\nu_t^{\lambda,\theta},\mu_t^\theta)$.
    Consequently, combining the bounds on $(\mathrm{I})$ and $(\mathrm{II})$ in \eqref{eq:discrete-stability-split},
    \begin{align}
        d_{\mathrm{TV}}(\nu_{t+1}^{\lambda,\theta},\mu_{t+1}^\theta)
        \leq d_{\mathrm{TV}}(\nu_t^{\lambda,\theta},\mu_t^\theta)+L_K\lambda.
    \end{align}
    The result follows by induction from $\nu_0^{\lambda,\theta}=\mu_0$.
\end{proof}

\begin{proof}[Objective estimate in Theorem~\ref{theorem:discrete-perturbation-consistency}]
    Fix $t\in\llbracket 0,T-1\rrbracket$. We add and subtract $\E[R_t^\theta(X_t^{\lambda,\theta},\mu_t^\theta)]$. By the Lipschitz property in Assumption~\ref{assumption:discrete-uniform-regularity} and Proposition~\ref{prop:discrete-perturbation-estimate}, $|R_t^\theta(X_t^{\lambda,\theta},M_t^{\lambda,\theta})-R_t^\theta(X_t^{\lambda,\theta},\mu_t^\theta)|\leq L_R\,d_{\mathrm{TV}}(M_t^{\lambda,\theta},\mu_t^\theta)\leq L_R\lambda$ point-wise. Moreover, since $|R_t^\theta|\leq\overline R$ and $\nu_t^{\lambda,\theta}=\Lc(X_t^{\lambda,\theta})$,
    \begin{align}
        \Big|\E\big[R_t^\theta(X_t^{\lambda,\theta},\mu_t^\theta)\big]-\sum_i\mu_t^\theta(i)R_t^\theta(i,\mu_t^\theta)\Big|
        =\Big|\sum_i\big(\nu_t^{\lambda,\theta}(i)-\mu_t^\theta(i)\big)R_t^\theta(i,\mu_t^\theta)\Big|
        \leq 2\overline R\,d_{\mathrm{TV}}(\nu_t^{\lambda,\theta},\mu_t^\theta),
    \end{align}
    using $\sum_i|\nu(i)-\mu(i)|=2d_{\mathrm{TV}}(\nu,\mu)$. Hence
    \begin{align}
        \Big|\E\big[R_t^\theta(X_t^{\lambda,\theta},M_t^{\lambda,\theta})\big] - \sum_i\mu_t^\theta(i)R_t^\theta(i,\mu_t^\theta)\Big|
        \leq 2\overline R\,d_{\mathrm{TV}}(\nu_t^{\lambda,\theta},\mu_t^\theta)+L_R\lambda.
    \end{align}
    The terminal contribution is bounded by $2\overline g\,d_{\mathrm{TV}}(\nu_T^{\lambda,\theta},\mu_T^\theta)+L_g\lambda$. Summing and applying the lemma yields
    \begin{align}
        |J^\lambda(\theta)-J(\theta)|\leq C_T\lambda,\qquad
        C_T=\overline R L_KT(T-1)+L_RT+2\overline gL_KT+L_g.
    \end{align}
\end{proof}

For the gradient estimate, use the coordinate simplex $\Zc_N \coloneqq \Big\{u\in[0,1]^{N-1} : \sum_{k=1}^{N-1} u_k \leq1\Big\}$. 
Write $z_t^\theta=(\mu_t^\theta(1),\ldots,\mu_t^\theta(N-1))$, $z^\theta=(z_t^\theta)_{t=0}^T \in \R^{(T+1) \times (N-1)}$, and $D^\theta=\nabla_\theta z^\theta$. The time-$t$ block is $D_t^\theta\in\R^{(N-1)\times d_\theta}$. For a prescribed coordinate flow $z=(z_t)_{t=0}^T$, let $m_t=(z_t,1-\sum_{k<N}z_t(k))$ and define $v_t : \Theta \times \Zc_N^{T+1}\to \Delta_N$ and $F_N : \Theta \times \Zc_N^{T+1} \to \R$ by
\begin{align}
    v_0^\theta(z)&=\mu_0,\qquad
    v_{t+1}^\theta(z)=v_t^\theta(z)K_t^\theta(\cdot\mid\cdot,m_t), \nonumber\\
    F_N(\theta,z)&=\sum_{t=0}^{T-1}\sum_i v_t^\theta(z)(i)R_t^\theta(i,m_t) + \sum_i v_T^\theta(z)(i)g(i,m_T).
    \label{eq:discrete-prescribed-flow-objective}
\end{align}
Here, the product of a probability vector and a transition kernel is defined component-wise by
\begin{align}
    \big[v_t^\theta(z) K_t^\theta(\cdot\mid \cdot,m_t) \big](j) \coloneqq \sum_{i=1}^N v_t^\theta(z)(i) K_t^\theta(x_j \mid x_i,m_t),\qquad
    j \in \llbracket1,N\rrbracket.
\end{align}
Equivalently,
\begin{align}
    F_N(\theta,z)
    &= \int_{\Xc^{T+1}\times\Ac^T}
    \Big[\sum_{t=0}^{T-1}r_t(x_t,m_t,a_t)+g(x_T,m_T)\Big]
    \mu_0(\d x_0)
    \prod_{t=0}^{T-1}
    \Big[
        \pi_t^\theta(\d a_t\mid x_t,m_t)
        P_t(\d x_{t+1}\mid x_t,m_t,a_t)
    \Big].
\end{align}
For fixed $z$, the flow $(m_t)_{t=0}^T$ is fixed, so only
the policy factors depend on $\theta$. Thus differentiating
with respect to $\theta$ produces only the policy score terms.
Thus $F_N$ is the objective for a prescribed population flow. For $\omega\in\Omega$, write $q_t(\omega)=(Q_t(1)(\omega),\ldots,Q_t(N-1)(\omega))$ and $q(\omega)=(q_t(\omega))_{t=1}^T$. With $z^{\lambda,\theta}(\omega)=(1-\lambda)z^\theta+\lambda q(\omega)$, and setting $z_0^{\lambda,\theta} = z_0$, we hence have that
$J(\theta)=F_N(\theta,z^\theta)$ and $J^\lambda(\theta)=\E[F_N(\theta,z^{\lambda,\theta})]$.

\begin{Assumption}[Gradient regularity]
    \label{assumption:discrete-gradient-regularity}
    There is an open set $\Theta^\circ\supset\Theta$ such that the following conditions hold. Derivatives in $m$ are taken in the coordinates $(m_1,\ldots,m_{N-1})$, with $m_N = 1-\sum_{k=1}^{N-1} m_k$.
    \begin{enumerate}
        \item [(1)] For every $t \in \llbracket0,T-1\rrbracket$, $i,j\in\llbracket1,N\rrbracket$ and $\nu_A$-a.e.\ $a \in \Ac$, the map $(\theta,m) \mapsto p_t^\theta(a\mid x_i,m)$ is $\Cc^1$ on $\Theta^\circ \times \overset{\circ}{\Delta}_N$, and the maps $m\mapsto P_t(\{x_j\}\mid x_i,m,a)$, $m\mapsto r_t(x_i,m,a)$ and $m\mapsto g(x_i,m)$ are $\mathcal{C}^1$ on $\overset{\circ}{\Delta}_N$. These maps are continuous on their full domains $\Theta^\circ \times \Delta_N$ and $\Delta_N$, respectively, and their first derivatives extend continuously from the interior to these domains. 
        
        \item [(2)]  The functions $P_t$, $r_t$, $g$ and their first derivatives are bounded and Lipschitz in $m$, uniformly in $(t,i,j,a) \in \llbracket 0, T-1 \rrbracket \times \llbracket 1, N \rrbracket \times \llbracket 1, N \rrbracket \times \Ac$.
        
        \item [(3)]   
    The function $p_t^\theta$, its first derivatives in $(\theta,m)$, and the Lipschitz constants in $m$ of these derivatives are bounded by a measurable map $\ell : \Ac\to[0,\infty)$, uniformly in $(t,i,\theta) \in \llbracket 0, T-1 \rrbracket \times \llbracket 1, N \rrbracket \times  \Theta^\circ,m \in \Delta_N$, for some $\ell\in L^1(\nu_A)$.
    
       \item [(4)] 
    The initial law $\mu_0$ is independent of $\theta$.
    \end{enumerate}
\end{Assumption}

\begin{proof}[Gradient estimate in Theorem~\ref{theorem:discrete-perturbation-consistency}]
    Starting from the constant map $\theta \mapsto \mu_0$, the population recursion defined in~\ref{eq:population-recursion} as $\mu_{t+1}^\theta(j) = \sum_{i=1}^N \mu_t^\theta(i) K_t^\theta(j\mid i,\mu_t^\theta)$ and Assumption~\ref{assumption:discrete-gradient-regularity} give $\Cc^1$ maps $\mu_t : \Theta^\circ \to \Delta_N$ by induction for any $t \in \llbracket 1, T \rrbracket$. 
    Differentiating this finite sum by the product and chain rules gives
    \begin{align}
        \nabla_\theta\mu_{t+1}^\theta(j) = \sum_{i=1}^N K_t^\theta(j\mid i,\mu_t^\theta) \nabla_\theta\mu_t^\theta(i) + \sum_{i=1}^N\mu_t^\theta(i) \Big[\nabla_\theta K_t^\theta(j\mid i,\mu_t^\theta)+(D_t^\theta)^\top
      \nabla_m K_t^\theta(j\mid i,\mu_t^\theta) \Big]
      \label{eq:discrete-exact-sensitivity-recursion}
    \end{align}
    Here the last term is the chain rule through $\mu_t^\theta$, whose first $N-1$ coordinates have derivative $D_t^\theta$.

    Since $\mu_t^\theta(N) = 1-\sum_{i<N} \mu_t^\theta(i)$, and by Cauchy--Schwarz
    \begin{align}
        \sum_{i=1}^N\|\nabla_\theta\mu_t^\theta(i)\|
        \leq 2\sum_{i<N}\|\nabla_\theta\mu_t^\theta(i)\|
        \leq 2\sqrt{N-1}\,\|D_t^\theta\|_{\mathrm F},
    \end{align}
    where we recall that $\|D_t^\theta\|_\mathrm{F}^2 = \sum_{i=1}^{N-1} \|\nabla_\theta \mu_t^\theta(i)\|^2$. The derivatives of $K_t^\theta$ are uniformly bounded by Assumption~\ref{assumption:discrete-gradient-regularity}, and $\sum_i\mu_t^\theta(i)=1$. Thus, summing \eqref{eq:discrete-exact-sensitivity-recursion} over $j<N$ gives $\|D_{t+1}^\theta\|_{\mathrm F} \leq C_{N,T}\bigl(1+\|D_t^\theta\|_{\mathrm F}\bigr)$.
    Because $\mu_0$ does not depend on $\theta$, $D_0^\theta=0$. Induction over the finite horizon therefore yields
    \begin{align}
        \sup_{\theta\in\Theta}\max_{t\in\Tc} \|D_t^\theta\|_{\mathrm F}\leq C_{N,T}.
        \label{eq:discrete-exact-sensitivity-bound}
    \end{align}

    We next record the regularity of $F_N$. Recall from~\eqref{eq:discrete-prescribed-flow-objective} that
    \begin{align}
        F_N(\theta,z)=\sum_{t=0}^{T-1}\sum_{i=1}^N v_t^\theta(z)(i)\,R_t^\theta(i,m_t)+\sum_{i=1}^N v_T^\theta(z)(i)\,g(i,m_T),
    \end{align}
    where $m_t$ is affine in $z_t$, with constant Jacobian $A_N\in\R^{N\times(N-1)}$. By the product and chain rules,
    \begin{align}
        \nabla_\theta F_N(\theta,z)&=\sum_{t=0}^{T-1}\sum_{i}\Big(\nabla_\theta v_t^\theta(z)(i)\,R_t^\theta(i,m_t)+v_t^\theta(z)(i)\,\nabla_\theta R_t^\theta(i,m_t)\Big)+\sum_{i}\nabla_\theta v_T^\theta(z)(i)\,g(i,m_T),\\
        \nabla_{z_t} F_N(\theta,z)&=\sum_{s=0}^{T-1}\sum_{i}\nabla_{z_t}v_s^\theta(z)(i)\,R_s^\theta(i,m_s)+\sum_{i}\nabla_{z_t} v_T^\theta(z)(i)\,g(i,m_T)+\sum_{i}v_t^\theta(z)(i) \nabla_m\widetilde R_t(i,m_t), 
    \end{align}
    
    for $t\in\llbracket1,T\rrbracket$, where $\widetilde R_t=R_t^\theta$ for $t<T$ and $\widetilde R_T=g$. By Assumption~\ref{assumption:discrete-gradient-regularity} and dominated convergence, $K_t^\theta$, $R_t^\theta$, $g$ and their first derivatives in $(\theta,m)$ are bounded and Lipschitz in $m$, uniformly in $\theta$. Differentiating the recursion $v_{t+1}^\theta(z)=v_t^\theta(z)K_t^\theta(\cdot\mid\cdot,m_t)$ expresses $\nabla_\theta v_{t+1}^\theta$ and $\nabla_zv_{t+1}^\theta$ through $v_t^\theta$, $\nabla v_t^\theta$, $K_t^\theta$ and its derivatives. Since $v_t^\theta(z)\in\Delta_N$ and $v_0^\theta=\mu_0$ does not depend on $(\theta,z)$, induction over the finite horizon shows that $v_t^\theta$, $\nabla_\theta v_t^\theta$ and $\nabla_zv_t^\theta$ are bounded and Lipschitz in $z$, uniformly in $\theta$. Each term in the two gradients above is a product of such factors, which gives, uniformly in $\theta$ and in probability flows $z,z'$,
    \begin{align}
        \|\nabla_\theta F_N(\theta,z)\| + \|\nabla_z F_N(\theta,z)\|&\leq C_{N,T},\\
        \|\nabla_\theta F_N(\theta,z)-\nabla_\theta F_N(\theta,z')\| +\|\nabla_z F_N(\theta,z)-\nabla_z F_N(\theta,z')\| &\leq C_{N,T}\|z-z'\|.
        \label{eq:discrete-objective-derivative-bounds}
    \end{align}

    We use the Euclidean norm $\|z-z'\|^2 \coloneqq \sum_{t=0}^T \sum_{k=1}^{N-1} |z_t(k) - z'_t(k)|^2$.
    Under this identification, $D^\theta = \nabla_\theta z^\theta \in \R^{(T+1)(N-1)\times d_\theta}$ is the Jacobian of the coordinate flow.
    
    For each realization of $q$, set $z^{\lambda,\theta}=(1-\lambda)z^\theta+\lambda q$. Since each $q_t$ and $\mu_t^\theta$ is a probability vector,
    \begin{align}
        \|z^{\lambda,\theta}-z^\theta\|\leq C_{N,T}\lambda,\qquad
        \nabla_\theta z^{\lambda,\theta}=(1-\lambda)D^\theta.
        \label{eq:discrete-perturbed-flow-bounds}
    \end{align}
    The chain rule for $J(\theta) = F_N(\theta,z^\theta)$ gives
    \begin{align}
        \label{eq:discrete-gradient-identity}
        \nabla_\theta J(\theta) =  \nabla_\theta F_N(\theta,z^\theta) + (D^\theta)^\top \nabla_z F_N(\theta,z^\theta).
    \end{align}
    For the perturbed composition the same rule gives the integrand $\nabla_\theta F_N(\theta,z^{\lambda,\theta}) + (1-\lambda)(D^\theta)^\top \nabla_z F_N(\theta,z^{\lambda,\theta})$.
    The derivatives in this expression are uniformly bounded by \eqref{eq:discrete-exact-sensitivity-bound} and \eqref{eq:discrete-objective-derivative-bounds}. Thus they provide an integrable bound independent of $q$. Differentiation under the expectation defining $J^\lambda$ is justified by dominated convergence, and hence
    \begin{align}
        \nabla_\theta J^\lambda(\theta) = \E\Big[\nabla_\theta F_N(\theta,z^{\lambda,\theta}) + (1-\lambda)(D^\theta)^\top \nabla_z F_N(\theta, z^{\lambda,\theta}) \Big].
        \label{eq:discrete-perturbed-gradient-identity}
    \end{align}
    Subtracting these identities, and adding and subtracting $(1-\lambda)(D^\theta)^\top\nabla_zF_N(\theta,z^\theta)$, yields
    \begin{align}
        \nabla_\theta J^\lambda(\theta)-\nabla_\theta J(\theta) &= \E\Big[\nabla_\theta F_N(\theta, z^{\lambda,\theta}) - \nabla_\theta F_N(\theta,z^\theta)\\ &\quad+ (1-\lambda)(D^\theta)^\top \big(\nabla_z F_N(\theta,z^{\lambda,\theta}) - \nabla_z F_N(\theta,z^\theta) \big) - \lambda(D^\theta)^\top\nabla_z F_N(\theta,z^\theta) \Big].
    \end{align}
    Applying \eqref{eq:discrete-exact-sensitivity-bound}, \eqref{eq:discrete-objective-derivative-bounds}, and \eqref{eq:discrete-perturbed-flow-bounds}, and using $1-\lambda\leq1$, gives
    \begin{align}
        \sup_{\theta\in\Theta} \|\nabla_\theta J^\lambda(\theta)-\nabla_\theta J(\theta)\| \leq C_T^\nabla\lambda,\qquad C_T^\nabla\coloneqq C_{N,T}.
    \end{align}
    This proves the gradient estimate.
\end{proof}

\begin{proof}[Proof of Corollary~\ref{corollary:discrete-maximizers}]
    By the proof of Theorem~\ref{theorem:discrete-perturbation-consistency}, $J$ and $J^\lambda$ are differentiable on $\Theta^\circ\supset\Theta$, hence continuous on the compact set $\Theta$. Maximizers $\theta^\star$ of $J$ and $\theta_\lambda^\star$ of $J^\lambda$ therefore exist, and the objective estimate gives
    \begin{align}
        J(\theta_\lambda^\star)\geq J^\lambda(\theta_\lambda^\star)-C_T\lambda\geq J^\lambda(\theta^\star)-C_T\lambda\geq J(\theta^\star)-2C_T\lambda.
    \end{align}
    If $\theta_{\lambda_n}^\star\to\bar\theta$ with $\lambda_n\to0$, continuity of $J$ gives $J(\bar\theta)=\lim_nJ(\theta_{\lambda_n}^\star)\geq J(\theta^\star)$, so $\bar\theta$ maximizes $J$.
\end{proof}

\subsubsection{Density and score for the affine perturbation of the weights}
\label{app:discrete-density-score}

\begin{Assumption}[Assumptions on the law $\rho$ of the randomizer $(U_t)_{t\in\llbracket 1,T\rrbracket}$]
\label{assumption:randomizer_rho_discrete}
We will assume the following:
\begin{enumerate}
    \item[(1)]
    The random variables $(U_t)_{t\in\llbracket 1,T\rrbracket}$ are
    independent and identically distributed over $\R^{N-1}$, with
    density $\rho$ with respect to the Lebesgue measure, with $\rho\in \Cc^1(\R^{N-1};(0,+\infty))$.
    \item[(2)] There exist constants $C_\rho>0$ and $\beta>N$ such that, for every $u\in\R^{N-1}$,
    \begin{align}
        \rho(u)\bigl(1+\|\nabla\log\rho(u)\|_1\bigr)
        \leq C_\rho e^{-\beta\|u\|_1}.
        \label{eq:discrete-rho-decay}
    \end{align}
\end{enumerate}
\end{Assumption}
This condition concerns only $\rho$ and its score. Since $\rho>0$, its left-hand side equals $\rho(u)+\|\nabla\rho(u)\|_1$. Every nondegenerate Gaussian density satisfies it: its density has quadratic exponential decay and its score grows at most linearly.

\begin{Lemma}[Law of the perturbed population]
\label{lemma:discrete-law-M}
For every $z\in\Zc_N$ and $\lambda\in(0,1)$, the random variable $\Psi_\lambda(z,Q_t)$ admits the density $h^\lambda(\cdot\mid z)$ defined in~\eqref{eq:discrete-main-density} with respect to Lebesgue measure in simplex coordinates:
\begin{align}
    h^\lambda(y\mid z)
    =\frac{1}{\lambda^{N-1}}
    f_Q\big(q^\lambda(y\mid z)\big),
    \qquad y\in Y_\lambda(z),
\end{align}
and the density is zero elsewhere. In particular, for every $t\in\llbracket1,T\rrbracket$ and $\theta\in\Theta$, the random variable $M_t^{\lambda,\theta}$ defined in~\eqref{eq:discrete-main-transport} has density $h^\lambda(\cdot\mid z_t^\theta)$ in these coordinates.
\end{Lemma}

\begin{proof}
Fix $z\in\Zc_N$ and $\lambda\in(0,1)$. The affine map $\Psi_\lambda(z,\cdot)$ defined in~\eqref{eq:discrete-perturbation-map} is a $\Cc^1$-diffeomorphism from the simplex interior onto $Y_\lambda(z)$, with inverse $q^\lambda(\cdot\mid z)$ and Jacobian determinant $\lambda^{N-1}$.

Since $Q_t=\varphi(U_t)$ and $U_t$ has density $\rho$, the change-of-variables formula gives
\begin{align}
    f_Q(q) = \frac{\rho(\varphi^{-1}(q))}{\prod_{i=1}^N q_i}.
\end{align}
Applying the change-of-variables formula again gives
\begin{align}
    h^\lambda(y\mid z) =\frac{1}{\lambda^{N-1}}  f_Q\big(q^\lambda(y\mid z)\big),
    \qquad y\in Y_\lambda(z),
\end{align}
with density zero elsewhere. Taking $z=z_t^\theta$ gives the result for $M_t^{\lambda,\theta}$.
\end{proof}

\begin{Lemma}[Extension at the simplex boundary]
\label{lemma:discrete-density-extension}
Under Assumption~\ref{assumption:randomizer_rho_discrete}, identify the simplex with its first $N-1$ coordinates and define
\begin{align}
    \widetilde f_Q:\R^{N-1}\longrightarrow[0,\infty),\qquad
    \widetilde f_Q(q)\coloneqq
    \begin{cases}
        f_Q(q),&q\in\Zc_N^\circ,\\
        0,&q\notin\Zc_N^\circ.
    \end{cases}
\end{align}
Then $\widetilde f_Q\in\Cc_c^1(\R^{N-1})$ and both its value and its gradient vanish on $\partial\Zc_N$.
\end{Lemma}

\begin{proof}
For $q\in\Zc_N^\circ$, set $q_N=1-\sum_{i=1}^{N-1}q_i$, $u=\varphi^{-1}(q)$, $r=\|u\|_1$ and $u_N=0$. Writing $S=\sum_{i=1}^Ne^{u_i}$ gives $q_i=e^{u_i}/S$. If $a=\max_{1\leq i\leq N}u_i\geq0$, then $S\leq Ne^a$, $a-u_i\leq r$ and $Na-\sum_{i=1}^{N-1}u_i\leq(N-1)r$. Hence
\begin{align}
    \max_{1\leq i\leq N}\frac1{q_i}
    &\leq Ne^r,\qquad
    \frac1{\prod_{i=1}^Nq_i} = S^N e^{-\sum_{i=1}^{N-1}u_i}
    \leq N^N e^{(N-1)r}.
\end{align}
Using $f_Q(q)=\rho(u)/\prod_{i=1}^Nq_i$ and the formula for $H_k=\partial_{q_k}\log f_Q$, these bounds give
\begin{align}
    |f_Q(q)|+\|\nabla_q f_Q(q)\|_1
    &\leq C_N e^{Nr}\rho(u)
    \bigl(1+\|\nabla\log\rho(u)\|_1\bigr)\\
    &\leq C_NC_\rho e^{-(\beta-N)r},
\end{align}
where $C_N$ depends only on $N$.

As $q\in\Zc_N^\circ$ approaches any point of $\partial\Zc_N$, we have $\min_{1\leq i\leq N}q_i\to0$. The bound on $\max_{1\leq i\leq N}q_i^{-1}$ therefore gives
\begin{align}
    r=\|\varphi^{-1}(q)\|_1
    \geq\log\frac{1}{N\min_{1\leq i\leq N}q_i}
    \longrightarrow\infty.
\end{align}
Since $\beta>N$, the preceding estimate shows that $f_Q(q)\to0$ and $\nabla_q f_Q(q)\to0$ as $q$ approaches the boundary from the interior. In particular, $\widetilde f_Q$ is continuous at every boundary point.

It remains to prove differentiability there. Define
\begin{align}
    \partial\Zc_N &\coloneqq \bigg\{ z\in\Zc_N : z_k = 0 \text{ for some } k \text{ or } \sum_{k=1}^{N-1} z_k = 1\bigg\},\quad
    \Zc_N^\circ \coloneqq \bigg\{z\in(0,\infty)^{N-1} : \sum_{k=1}^{N-1}z_k < 1\bigg\}.
\end{align}
Fix $\bar q\in\partial\Zc_N$ and let $q\in\Zc_N^\circ$. The segment from $\bar q$ to $q$ lies in the interior except at $\bar q$. Applying the one-dimensional mean value theorem along this segment gives a point $\xi_q$ on its open part such that
\begin{align}
    \widetilde f_Q(q)-\widetilde f_Q(\bar q)
    =\nabla_q f_Q(\xi_q)\cdot(q-\bar q).
\end{align}
Since $\xi_q$ lies between $\bar q$ and $q$, $\|\xi_q-\bar q\|\leq\|q-\bar q\|$. Thus $\xi_q\to\bar q$ as $q\to\bar q$, and the boundary limit of the interior gradient gives $\nabla_q f_Q(\xi_q)\to0$. By the Cauchy--Schwarz inequality,
\begin{align}
    \frac{|\widetilde f_Q(q)-\widetilde f_Q(\bar q)|}
    {\|q-\bar q\|}
    \leq\|\nabla_q f_Q(\xi_q)\|\longrightarrow0.
\end{align}
For $q\notin\Zc_N^\circ$, the numerator is zero by definition. Thus the difference quotient tends to zero from every direction, proving that $\nabla\widetilde f_Q(\bar q)=0$.

The gradient is continuous across the boundary: it tends to zero from the interior and is zero outside. Therefore $\widetilde f_Q\in\Cc^1(\R^{N-1})$. Finally, its support is contained in the closed bounded set $\Zc_N$, so $\widetilde f_Q\in\Cc_c^1(\R^{N-1})$.
\end{proof}

\begin{Lemma}[Stein identity on the simplex]
    \label{lemma:discrete-stein}
    Suppose Assumption~\ref{assumption:randomizer_rho_discrete} holds and let $Q = \varphi(U)$, with $U\sim\rho$. Let $\kappa\in(0,1)$, $c\in\Delta_N$, and let $h:\Delta_N \to \R$ be continuous, $\Cc^1$ on $\overset{\circ}{\Delta}_N$ in the coordinates $(m_1,\ldots,m_{N-1})$, with bounded gradient. Then, for every $k\in\llbracket1,N-1\rrbracket$,
    \begin{align}
        \E\big[h\big((1-\kappa)c+\kappa Q\big)H_k(Q)\big]=-\kappa\,\E\big[\partial_k h\big((1-\kappa)c+\kappa Q\big)\big].
        \label{eq:discrete-stein}
    \end{align}
    In particular, $\E[H(Q)]=0$.
\end{Lemma}

\begin{proof}
    Define $g : \Delta_N\to\R$ by $g(q) \coloneqq h\big((1-\kappa)c+\kappa q\big)$. Since $\kappa> 0$, for $q\in\overset{\circ}{\Delta}_N$ every coordinate of $(1-\kappa) c + \kappa q$ is strictly positive, hence $(1-\kappa)c+\kappa q\in\overset{\circ}{\Delta}_N$, even if $c\in\partial\Delta_N$. Therefore $g$ is continuous on $\Delta_N$, $\Cc^1$ on $\overset{\circ}{\Delta}_N$, and
    \begin{align}
        \label{eq:stein-chain-rule}
        \partial_k g(q) = \kappa\,\partial_k h\big((1-\kappa)c+ \kappa q\big)
    \end{align}
    is bounded. Using Lemma~\ref{lemma:discrete-density-extension}, since $H_k = \partial_k \log f_Q$ on $\overset{\circ}{\Delta}_N$, we have $H_k f_Q = \partial_k f_Q = \partial_k \widetilde f_Q$ there, and consequently
    \begin{align}
        \E[g(Q)H_k(Q)] = \int_{\overset{\circ}{\Delta}_N} g(q)\, \partial_k \widetilde f_Q(q)\, \d q,
        \label{eq:stein-start}
    \end{align}
    where the integrand is bounded, because $g$ is continuous on the compact set $\Delta_N$ and $\widetilde f_Q\in\Cc_c^1(\R^{N-1})$.

    \emph{Splitting the integral.} Throughout, points of $\overset{\circ}{\Delta}_N$ are written in the coordinates $q=(q_1,\ldots,q_{N-1})$, with $q_N=1-\sum_{i=1}^{N-1}q_i$. Fix $k$ and separate the $k$-th coordinate from the others: write $q_{-k}\coloneqq(q_i)_{i\neq k,\, i\leq N-1}$, and denote by $(s,q_{-k})$ the point whose $k$-th coordinate is $s$ and whose other coordinates are $q_{-k}$. A point $(s,q_{-k})$ lies in $\overset{\circ}{\Delta}_N$ if and only if
    \begin{align}
        q_{-k}\in S_{-k}\coloneqq\Big\{q_{-k}:q_i>0\ \text{for all } i\neq k,\ \sum_{i\neq k}q_i<1\Big\}
        \quad\text{and}\quad
        0<s<b(q_{-k})\coloneqq1-\sum_{i\neq k}q_i.
    \end{align}
    By Fubini's theorem, which applies because the integrand in~\eqref{eq:stein-start} is bounded on a bounded set,
    \begin{align}
        \E[g(Q)H_k(Q)]=\int_{S_{-k}}\bigg(\int_0^{b(q_{-k})}g(s,q_{-k})\,\partial_k\widetilde f_Q(s,q_{-k})\,\d s\bigg)\d q_{-k}.
        \label{eq:stein-fubini}
    \end{align}

    \emph{Integration by parts in one variable.} Fix $q_{-k}\in S_{-k}$ and write $b=b(q_{-k})$. The functions $s\mapsto g(s,q_{-k})$ and $s\mapsto\widetilde f_Q(s,q_{-k})$ are continuous on $[0,b]$ and $\Cc^1$ on $(0,b)$, with bounded derivatives. Both endpoints lie on $\partial\Delta_N$: at $s=0$, the $k$-th coordinate is $q_k=0$, and at $s=b$, the last coordinate is $q_N=1-b-\sum_{i\neq k}q_i=0$. A point of $\Delta_N$ with a zero coordinate lies on $\partial\Delta_N$. By Lemma~\ref{lemma:discrete-density-extension}, $\widetilde f_Q=0$ on the simplex boundary, so $\widetilde f_Q(0,q_{-k})=\widetilde f_Q(b,q_{-k})=0$. Integrating by parts on $[0,b]$ therefore gives
    \begin{align}
        \int_0^{b}g(s,q_{-k})\,\partial_k\widetilde f_Q(s,q_{-k})\,\d s
        &=\Big[g(s,q_{-k})\,\widetilde f_Q(s,q_{-k})\Big]_{s=0}^{s=b}-\int_0^{b}\partial_kg(s,q_{-k})\,\widetilde f_Q(s,q_{-k})\,\d s\\
        &=-\int_0^{b}\partial_kg(s,q_{-k})\,\widetilde f_Q(s,q_{-k})\,\d s.
    \end{align}

    \emph{Conclusion.} Substituting this into~\eqref{eq:stein-fubini}, recombining the two integrals by Fubini's theorem, and using $\widetilde f_Q=f_Q$ on $\overset{\circ}{\Delta}_N$ and~\eqref{eq:stein-chain-rule},
    \begin{align}
        \E[g(Q)H_k(Q)]
        = -\int_{\overset{\circ}{\Delta}_N} \partial_k g(q)\, f_Q(q)\,\d q
        = -\kappa\, \E\big[\partial_k h\big((1-\kappa)c+\kappa Q\big)\big].
    \end{align}
    Since $g(Q)=h((1-\kappa)c+\kappa Q)$, this is~\eqref{eq:discrete-stein}. Taking $h$ constant gives $\E[H_k(Q)]=0$ for every $k\in\llbracket1,N-1\rrbracket$, hence $\E[H(Q)]=0$.
\end{proof}

\subsubsection{Model-free policy-gradient representation}
\label{proof:discrete-pg-representation}

We now turn to the computation of the policy gradient representation.

\begin{proof}[Proof of Theorem~\ref{theorem:discrete-pg-representation}]
    Fix $\lambda\in(0,1)$ and write $Q=(Q_1,\ldots,Q_T)$. Recall that the return $\Rc_\theta^\lambda$ and the policy score $\Sc_{\mathrm{pol}}^{\lambda,\theta}$ are defined in~\eqref{eq:perturbed-objective}, and the population score $\Sc_{\mathrm{mf}}^{\lambda,\theta}$ is defined in~\eqref{eq:discrete-main-scores}. Throughout the proof, $\nabla_\theta F_N(\theta,z)$ denotes the derivative with respect to the first argument, with $z$ held fixed.
    
    Recall that $z_0^{\lambda,\theta} = z_0$ and $z_t^{\lambda,\theta} = (1-\lambda)z_t^\theta + \lambda q_t$ for $t\in\llbracket1,T\rrbracket$,
    where $q_t$ denotes the first $N-1$ coordinates of $Q_t$. Since $D_0^\theta=0$, equation~\eqref{eq:discrete-perturbed-gradient-identity} gives
    \begin{align}
        \label{eq:discrete-pg-proof-target}
        \nabla_\theta J^\lambda(\theta)
        &= \E\big[\nabla_\theta F_N(\theta, z^{\lambda,\theta})\big] + (1-\lambda)\sum_{t=1}^T (D_t^\theta)^\top \E\big[\nabla_{z_t} F_N(\theta,z^{\lambda,\theta})\big].
    \end{align}
    We identify the two terms separately.
    
    \paragraph{Policy-score term.}
    Fix a prescribed coordinate flow $z$ and write $m_t = (z_t,1 - \sum_{k=1}^{N-1} z_t(k))$. For a trajectory $(x,a) = (x_0,\ldots,x_T,a_0,\ldots,a_{T-1})$, set
    \begin{align}
        \Rc(x,a;m) \coloneqq \sum_{t=0}^{T-1} r_t(x_t,m_t,a_t) + g(x_T,m_T).
    \end{align}
    The integral representation of the prescribed-flow objective is
    \begin{align}
        F_N(\theta, z) = \int_{\Xc^{T+1} \times \Ac^T} \Rc(x,a;m) \mu_0(\d x_0) \prod_{t=0}^{T-1} \big[p_t^\theta(a_t\mid x_t,m_t) \nu_A(\d a_t) P_t(\d x_{t+1}\mid x_t,m_t,a_t) \big].
    \end{align}
    With $z$ fixed, the return and transition kernels are independent of $\theta$, only the policy densities depend on $\theta$. Assumption~\ref{assumption:discrete-gradient-regularity} justifies differentiation under the integral, giving
    \begin{align}
        \nabla_\theta F_N(\theta,z) &= \int_{\Xc^{T+1} \times \Ac^T} \Rc(x,a;m) \bigg[\sum_{s=0}^{T-1} \nabla_\theta \log p_s^\theta(a_s\mid x_s, m_s)  \bigg] \mu_0(\d x_0) \\
        &\qquad \times \prod_{t=0}^{T-1} \bigg[p_t^\theta(a_t\mid x_t,m_t)\nu_A(\d a_t) P_t(\d x_{t+1}\mid x_t,m_t,a_t) \bigg].
    \end{align}
    We now evaluate this partial derivative at $z = z^{\lambda,\theta}$. Conditionally on $Q$, the product of kernels above is the trajectory law of the perturbed process. Hence
    \begin{align}
        \left. \nabla_\theta F_N(\theta,z) \right|_{z=z^{\lambda,\theta}} = \E\big[\Rc_\theta^\lambda\Sc_\mathrm{pol}^{\lambda,\theta} \mid Q \big].
    \end{align}
    Taking expectations yields
    \begin{align}
        \E\big[\nabla_\theta F_N(\theta,z^{\lambda,\theta})\big] =\E\big[ \Rc_\theta^\lambda\Sc_{\mathrm{pol}}^{\lambda,\theta} \big].
    \label{eq:discrete-policy-score-identification}
    \end{align}
    \paragraph{Population-score term.}
    Fix $t \in \llbracket1,T\rrbracket$, and condition on $Q_{-t} \coloneqq (Q_s)_{s\neq t}$. For fixed $Q_{-t}$, define $h_t: \Delta_N \to \R$ by 
    \begin{align}
        h_t(m) \coloneqq F_N\big(\theta, z_0^{\lambda,\theta},\ldots,z_{t-1}^{\lambda,\theta}, (m_1,\ldots,m_{N-1}), z_{t+1}^{\lambda,\theta},\ldots,z_T^{\lambda,\theta}\big)
    \end{align}
    Derivatives of $h_t$ are taken in the coordinates $(m_1,\ldots,m_{N-1})$, with $m_N = 1-\sum_{k=1}^{N-1} m_k$. Thus, for every $k \in \llbracket1,N-1\rrbracket$,
    \begin{align}
        h_t \big((1-\lambda)\mu_t^\theta + \lambda Q_t\big) &= F_N(\theta, z^{\lambda,\theta}),\\
        \partial_{m_k} h_t\big((1-\lambda)\mu_t^\theta + \lambda Q_t\big) &= \partial_{z_t(k)} F_N(\theta,z^{\lambda,\theta}).
    \end{align}
    The required regularity follows from Assumption~\ref{assumption:discrete-gradient-regularity} and~\eqref{eq:discrete-objective-derivative-bounds}. Since $Q_t$ is independent of $Q_{-t}$, Lemma~\ref{lemma:discrete-stein}, applied component-wise, gives
    \begin{align}
        \E\big[F_N(\theta,z^{\lambda,\theta}) H(Q_t)\mid Q_{-t}] = -\lambda\E\big[\nabla_{z_t} F_N(\theta,z^{\lambda,\theta}) \mid Q_{-t}\big].
    \end{align}
    Taking expectations,
    \begin{align}
        \label{eq:discrete-stein-objective}
        \E\big[F_N(\theta,z^{\lambda,\theta})H(Q_t)\big] = -\lambda\E\big[\nabla_{z_t} F_N(\theta,z^{\lambda,\theta})\big]. 
    \end{align}

    Moreover, since $H(Q_t)$ is $\sigma(Q)$-measurable,
    \begin{align}
        \E\big[\Rc_\theta^\lambda H(Q_t) \big]
        = \E\big[\E[\Rc_\theta^\lambda\mid Q]H(Q_t)\big]
        = \E\big[F_N(\theta,z^{\lambda,\theta})H(Q_t)].
    \end{align}
    Combining this with~\eqref{eq:discrete-stein-objective} yields
    \begin{align}
        -\frac{1-\lambda}{\lambda}(D_t^\theta)^\top \E\big[\Rc_\theta^\lambda H(Q_t)\big] = (1-\lambda)(D_t^\theta)^\top \E\big[\nabla_{z_t}F_N(\theta,z^{\lambda,\theta})\big].
    \end{align}
    Summing over $t\in\llbracket 1, T \rrbracket$ gives
    \begin{align}
        \label{eq:discrete-population-score-identification}
        \E\big[\Rc_\theta^\lambda\Sc_\mathrm{mf}^{\lambda,\theta}\big] = (1-\lambda) \sum_{t=1}^T (D_t^\theta)^\top \E\big[\nabla_{z_t} F_N(\theta,z^{\lambda,\theta})\big].
    \end{align}

    Finally, combining \eqref{eq:discrete-policy-score-identification} and \eqref{eq:discrete-population-score-identification} with \eqref{eq:discrete-pg-proof-target} gives $\nabla_\theta J^\lambda(\theta) = \E\big[\Rc_\theta^\lambda(\Sc_\mathrm{pol}^{\lambda,\theta} + \Sc_\mathrm{mf}^{\lambda,\theta})\big]$.
\end{proof}

\subsubsection{Bias and mean-square error}
\label{proof:discrete-bias-mse}

\begin{Assumption}[Score regularity]
    \label{assumption:discrete-score-regularity}
    The policy scores and the population score satisfy
    \begin{align}
        \sup_{t,\theta,i,m} \int_\Ac \|\nabla_\theta\log p_t^\theta(a\mid x_i,m)\|^2 \pi_t^\theta(\d a\mid x_i,m) &<\infty,\\
        \E\big[\|H(Q_t)\|^2\big] &<\infty.
    \end{align}
\end{Assumption}

\begin{proof}[Proof of Theorem~\ref{theorem:discrete-bias-mse}]
Fix $\theta$. Throughout, $C$ may change from line to line and depend on the fixed dimensions, horizon, and regularity bounds, but not on \(\theta,\lambda,\eta,B,M,n\). Write $\Zc_N = \bigg\{u\in[0,1]^{N-1}:\sum_{k=1}^{N-1} u_k \leq 1\bigg\}$.

For $s\in\llbracket1,T\rrbracket$, write
\begin{align}
    z_s^\theta \coloneqq (\mu_s^\theta(1), \ldots,\mu_s^\theta(N-1)),\qquad
    \widehat z_s \coloneqq (\widehat\mu_s^M(1), \ldots,\widehat\mu_s^M(N-1)),
\end{align}
and set $z^\theta = (z_s^\theta)_{s=1}^T$ and $\widehat z = (\widehat z_s)_{s=1}^T$. Here $\widehat \mu_s^M = \frac{1}{M}\sum_{\ell=1}^M \delta_{X_s^\ell}$ is the empirical population law. We use the Euclidean norm for vectors and stacked coordinate flows. In particular, $\|\widehat z-z^\theta\|^2 = \sum_{s=1}^T \sum_{k=1}^{N-1} |\widehat \mu_s^M(k) - \mu_s^\theta(k)|^2$. 
The time-zero population argument remains fixed at $\mu_0$.

\emph{Conditional mean and its stability.}
For deterministic $c = (c_t)_{t=1}^T\in \Zc_N^T$, $D \in \big(\R^{(N-1)\times d_\theta}\big)^T$, and $\kappa\in[0,1)$, set $z^\kappa=(z_0,(1-\kappa)c_1+\kappa q_1, \ldots, (1-\kappa)c_T+\kappa q_T)$,
where $q_s$ consists of the $N-1$ coordinates of $Q_s$. Define
\begin{align}
    \Gamma^\kappa : \Zc_N^T \times \big(\R^{(N-1)\times d_\theta}\big)^T \to \R^{d_\theta},\qquad
    \Gamma^\kappa(c,D) \coloneqq \E\big[\nabla_\theta F_N(\theta,z^\kappa)\big] + (1-\kappa)\sum_{s=1}^T D_s^\top \E\big[\nabla_{z_s} F_N(\theta,z^\kappa)\big].
\end{align}
For $s\in\llbracket1,T\rrbracket$, let $\bar c_s = (c_s,1-\sum_{k=1}^{N-1}c_s(k))$ and set
\begin{align}
    m_0^\kappa=\mu_0,\qquad
    m_s^\kappa = (1-\kappa)\bar c_s + \kappa Q_s.
\end{align}
Define
\begin{align}
    X_0^\kappa \sim \mu_0,\qquad
    \alpha_s^\kappa \sim \pi_s^\theta(\cdot\mid X_s^\kappa,m_s^\kappa),\qquad
    X_{s+1}^\kappa\sim P_s(\cdot\mid X_s^\kappa, m_s^\kappa, \alpha_s^\kappa),\quad
    s\in\llbracket0,T-1\rrbracket.
\end{align}
Its return and policy score are
\begin{align}
    \Rc^\kappa \coloneqq \sum_{s=0}^{T-1} r_s(X_s^\kappa,m_s^\kappa,\alpha_s^\kappa) + g(X_T^\kappa,m_T^\kappa),\qquad
    \Sc_\mathrm{pol}^\kappa \coloneqq \sum_{s=0}^{T-1} \nabla_\theta \log p_s^\theta(\alpha_s^\kappa\mid X_s^\kappa, m_s^\kappa),
\end{align}
where each policy derivative holds the action, state and population argument fixed.

For $\kappa > 0$, the policy-score identity~\eqref{eq:discrete-policy-score-identification} and the Stein identity~\eqref{eq:discrete-stein-objective} give
\begin{align}
    \E\big[\Rc^\kappa\Sc_{\mathrm{pol}}^\kappa\big]=\E\big[\nabla_\theta F_N(\theta,z^\kappa)\big],\qquad
        \E\big[\Rc^\kappa H(Q_s)\big]=-\kappa\,\E\big[\nabla_{z_s}F_N(\theta,z^\kappa)\big].
\end{align}
Multiplying the second identity by $-\frac{1-\kappa}{\kappa}D_s^\top$ and summing over $s$ shows that
\begin{align}
    \E\bigg[\Rc^\kappa\bigg(\Sc_{\mathrm{pol}}^\kappa-\frac{1-\kappa}{\kappa}\sum_{s=1}^TD_s^\top H(Q_s)\bigg)\bigg]=\Gamma^\kappa(c,D).
\end{align}
Consequently, $\Gamma^\kappa(c,D)$ is the expectation of the score estimator using population coordinates $c$ and sensitivity matrices $D$.

Since $\Gamma^\kappa(c,D)-\Gamma^\kappa(c,D^\theta)=(1-\kappa)\sum_s(D_s-D_s^\theta)^\top\E[\nabla_{z_s}F_N(\theta,z^\kappa)]$, the derivative bounds~\eqref{eq:discrete-objective-derivative-bounds} imply
\begin{align}
    \big\|\Gamma^\kappa(c,D) - \Gamma^\kappa(c,D^\theta)\big\| \leq C\sum_{s=1}^T \big\|D_s - D_s^\theta\|_\mathrm{F}.
\end{align}
Moreover, the time-zero block of $z^\kappa-(z_0,z^\theta)$ vanishes and, for $s\geq1$,
\begin{align}
    z_s^\kappa-z_s^\theta=(1-\kappa)(c_s-z_s^\theta)+\kappa(q_s-z_s^\theta).
\end{align}
Since $q_s,z_s^\theta\in\Zc_N$ have Euclidean norm at most one, $\|q_s-z_s^\theta\|\leq2$, and the triangle inequality gives
\begin{align}
    \|z^\kappa-(z_0,z^\theta)\|\leq\|c-z^\theta\| + 2\kappa\sqrt{T}.
\end{align}
Write $\bar z^\theta = (z_0,z^\theta)$. Using $\Gamma^0(z^\theta, D^\theta) = \nabla_\theta J(\theta)$, we obtain
\begin{align}
    \Gamma^\kappa(c,D^\theta) - \nabla_\theta J(\theta) 
    &= \E\big[\nabla_\theta F_N(\theta,z^\kappa) - \nabla_\theta F_N(\theta,\bar z^\theta) \big] \\
    &\qquad + (1-\kappa)\sum_{s=1}^T (D_s^\theta)^\top \E\big[ \nabla_{z_s}F_N(\theta,z^\kappa) - \nabla_{z_s} F_N(\theta,\bar z^\theta) \big]
    - \kappa\sum_{s=1}^T (D_s^\theta)^\top \nabla_{z_s} F_N(\theta, \bar z^\theta).
\end{align}
The first two terms are controlled by the Lipschitz continuity in $z$ of the derivatives of $F_N$, while the last term is bounded by $C\kappa$. Indeed,~\eqref{eq:discrete-objective-derivative-bounds} and~\eqref{eq:discrete-exact-sensitivity-bound} give
\begin{align}
    \|\Gamma^\kappa(c,D^\theta) - \nabla_\theta J(\theta)\| &\leq C\E\big[ \|z^\kappa - \bar z^\theta \| \big] + C\kappa \\
    &\leq C\big(\|c-z^\theta\| + \kappa\big)
\end{align}
Combining these estimates yields
\begin{align}
    \label{eq:bias-stability}
    \big\|\Gamma^\kappa(c,D) - \nabla_\theta J(\theta)\big\| \leq C\bigg(\kappa + \|c-z^\theta\| + \sum_{s=1}^T \big\|D_s - D_s^\theta\big\|_\mathrm{F}\bigg).
\end{align}

\emph{Population-flow error.}
Let $\Fc_t^\mathrm{pop} \coloneqq \sigma\big(X_s^\ell : s\in\llbracket0,t\rrbracket, \ell\in\llbracket1,M\rrbracket\big)$ where $\sigma$ refers to the generated sigma-algebra.
The particles use independent simulation randomness at each time and rollout step. Conditionally on $\Fc_t^\mathrm{pop}$, the variables $(X_{t+1}^\ell)_{l \in \llbracket 1, M \rrbracket}$ are therefore independent, with laws $K_t^\theta(\cdot\mid X_t^\ell, \widehat \mu_t^M)$.
Therefore, $\E[\widehat\mu_{t+1}^M\mid\Fc_t^{\mathrm{pop}}]=\Phi_t^\theta(\widehat\mu_t^M,\widehat\mu_t^M)$, with $\Phi_t^\theta$ from~\eqref{eq:common-population-maps}. 

For each $j \in \llbracket1,N\rrbracket$,  writing $p_{\ell j}\coloneqq K_t^\theta(\{x_j\}\mid X_t^\ell,\widehat\mu_t^M)$,
\begin{align}
    \widehat\mu_{t+1}^M(j) - \Phi_t^\theta(\widehat \mu_t^M, \widehat \mu_t^M)(j) = \frac1M \sum_{\ell=1}^M \big(\mathbf 1_{\{X_{t+1}^\ell = x_j\} } - p_{\ell j} \big).
\end{align}
Conditionally on $\Fc_t^\mathrm{pop}$, the summands are independent and centered, with variances $p_{\ell j}(1-p_{\ell j})$. Since the squared Euclidean norm is the sum of the squared coordinates,
\begin{align}
    \E\big[\|\widehat \mu_{t+1}^M - \Phi_t^\theta(\widehat \mu_t^M, \widehat \mu_t^M)\|^2 \mid \Fc_t^\mathrm{pop}\big] = \frac{1}{M^2} \sum_{j=1}^N \sum_{\ell=1}^M p_{\ell j}(1-p_{\ell j}) \leq \frac1{M^2} \sum_{\ell= 1}^M\sum_{j=1}^N p_{\ell j} = \frac1M,
\end{align}
where we used $1-p_{\ell j} \leq 1$ and $\sum_{j=1}^N p_{\ell j} = 1$.


For $m,m' \in \Delta_N$, insert the intermediate law $\Phi_t^\theta(m,m')$ to obtain
\begin{align}
    d_\mathrm{TV}\big(\Phi_t^\theta(m,m), \Phi_t^\theta(m',m')\big)
    &\leq d_\mathrm{TV}\big(\Phi_t^\theta(m,m), \Phi_t^\theta(m,m')\big) + d_\mathrm{TV}\big(\Phi_t^\theta(m,m'),\Phi_t^\theta(m',m')\big)\\
    &\leq \sum_{i=1}^N m(i) d_\mathrm{TV}\big(K_t^\theta(\cdot\mid x_i,m),K_t^\theta(\cdot\mid x_i,m')\big) + d_\mathrm{TV}(m,m')\\
    &\leq (1+L_K) d_\mathrm{TV}(m,m')
\end{align}
The first term uses the Lipschitz bound of Assumption~\ref{assumption:discrete-uniform-regularity} on the kernel rows. For the second term, the population argument is fixed at $m'$, so the total-variation contraction~\eqref{eq:markov-kernel-contraction} applies.

Since $\mu_{t+1}^\theta=\Phi_t^\theta(\mu_t^\theta,\mu_t^\theta)$, write
\begin{align}
    \widehat\mu_{t+1}^M-\mu_{t+1}^\theta
    =\underbrace{\widehat\mu_{t+1}^M-\Phi_t^\theta(\widehat\mu_t^M,\widehat\mu_t^M)}_{\xi_{t+1}}
     +\underbrace{\Phi_t^\theta(\widehat\mu_t^M,\widehat\mu_t^M)-\Phi_t^\theta(\mu_t^\theta,\mu_t^\theta)}_{\zeta_t}.
\end{align}
Since $\E[\xi_{t+1}\mid \Fc_t^\mathrm{pop}]=0$ and $\zeta_t$ is $\Fc_t^\mathrm{pop}$-measurable,
\begin{align}
    \E\big[ \langle\xi_{t+1},\zeta_t\rangle\big] = \E\big[\langle\E[\xi_{t+1}\mid \Fc_t^\mathrm{pop}], \zeta_t\rangle\big]=0.
\end{align}
Moreover, the preceding total-variation estimate gives
\begin{align}
    \|\zeta_t\|
    &\leq \|\zeta_t\|_1= 2 d_\mathrm{TV}\big(\Phi_t^\theta(\widehat\mu_t^M, \widehat\mu_t^M),\Phi_t^\theta(\mu_t^\theta,\mu_t^\theta)\big)\\
    &\leq 2(1+L_K)d_\mathrm{TV}(\widehat\mu_t^M,\mu_t^\theta)= (1+L_K)\|\widehat\mu_t^M-\mu_t^\theta\|_1\\
    &\leq\sqrt N(1+L_K) \|\widehat \mu_t^M - \mu_t^\theta\|,
\end{align}
where the first and second equalities use the total variation definition of \citet{levin2026markov} and the last inequality follows from Cauchy-Schwarz. Consequently,
\begin{align}
     \E\big[\|\widehat\mu_{t+1}^M-\mu_{t+1}^\theta\|^2\big] &= \E\big[ \|\xi_{t+1}\|^2\big] + \E\big[ \|\zeta_t\|^2\big] \\
     &\leq \frac1M + N(1+L_K)^2 \E\big[\|\widehat \mu_t^M - \mu_t^\theta\|^2 \big].
\end{align}
The initial particles are i.i.d. with law $\mu_0$, so the same variance calculation gives $\E[\|\widehat \mu_0^M - \mu_0\|^2] \leq 1/M$. Induction over the finite horizon gives
\begin{align}
    \max_{s\in\Tc}
    \E\big[\|\widehat\mu_s^M - \mu_s^\theta\|^2\big] \leq \frac{C}{M}.
\end{align}
Since $\widehat z_s$ and $z_s^\theta$ contain the first $N-1$ coordinates of these probability vectors,
\begin{align}
    \E\big[\|\widehat z-z^\theta\|^2 \big] &\leq \sum_{s=1}^T \E\big[ \|\widehat\mu_s^M-\mu_s^\theta\|^2\big] \leq \frac{C}{M}.
    \label{eq:bias-population-error}
\end{align}


\emph{Sensitivity error.}
Let $\Bc_t$ collect the random inputs used to simulate auxiliary batch $t$. These inputs are independent across batches and independent of the population particle system. Set $\Fc_{t-1}^\mathrm{off} \coloneqq \Fc_T^\mathrm{pop}\vee \sigma(\Bc_1,\ldots,\Bc_{t-1})$. Conditionally on $\Fc_{t-1}^\mathrm{off}$, simulate $n$ independent rollouts of length $t$. For rollout $r\in\llbracket1,n\rrbracket$, draw independent randomizers $\overline Q_s^{(t,r)}$, $s\in\llbracket1,t-1\rrbracket$, with the same law as $Q_s$, and set
\begin{align}
    \overline m_0^{(t,r)} =\mu_0,\qquad
    \overline m_s^{(t,r)} = (1-\eta)\widehat \mu_s^M + \eta\overline Q_s^{(t,r)},\quad
    s\in\llbracket1,t-1\rrbracket.
\end{align}
Define
\begin{align}
    \overline X_0^{(t,r)}&\sim\mu_0,\qquad
    \overline\alpha_s^{(t,r)} \sim \pi_s^\theta(\cdot\mid \overline X_s^{(t,r)},\overline m_s^{(t,r)}
    ),\qquad
    \overline X_{s+1}^{(t,r)} \sim P_s(\cdot\mid \overline X_s^{(t,r)},\overline m_s^{(t,r)}, \overline\alpha_s^{(t,r)}),\quad
    s\in\llbracket0,t-1\rrbracket.
\end{align}
Thus $\overline X_t^{(t,r)}$ is the terminal state of rollout $r$ in auxiliary batch $t$. Set $\widehat D_0 = 0$. For $t\in\llbracket1,T\rrbracket$, define $\widehat D_t\in\R^{(N-1)\times d_\theta}$ recursively by its rows of size $d_\theta$:
\begin{align}
    (\widehat D_t)_{k,:} \coloneqq \frac1n \sum_{r=1}^n \mathbf 1_{\{\overline X_t^{(t,r)}=x_k\} } \bigg[ \sum_{s=0}^{t-1} \nabla_\theta \log p_s^\theta \big(\overline \alpha_s^{(t,r)}\mid \overline X_s^{(t,r)}, \overline m_s^{(t,r)} \big) -\frac{1-\eta}{\eta} \sum_{s=1}^{t-1} \widehat D_s^\top H(\overline Q_s^{(t,r)}) \bigg]^\top,\qquad
    k\in\llbracket1,N-1\rrbracket.
\end{align}
The policy derivatives hold the action, state, and population argument fixed. We write $\widehat D=(\widehat D_t)_{t=1}^T$.


For $k\in\llbracket1,N-1\rrbracket$, define $F_{t,k}(\theta,z) \coloneqq v_t^\theta(z)(k)$ for $z\in\Zc_N^{T+1}$, where $v_t^\theta(z)$ is the prescribed-flow state law from~\eqref{eq:discrete-prescribed-flow-objective}. This is the objective with horizon $t$, zero running reward, and terminal reward $x\mapsto \mathbf 1_{\{x=x_k\} }$. It depends only on $z_0,\ldots,z_{t-1}$ and $F_{t,k}(\theta,\bar z^\theta) = \mu_t^\theta(k)$ for $\bar z^\theta = (z_0,z^\theta)$. The same finite recursion used for $F_N$ gives bounded first derivatives of $F_{t,k}$ that are Lipschitz in $z$, uniformly in $\theta$.

Conditionally on $\Fc_{t-1}^\mathrm{off}$, the population coordinates and earlier sensitivity estimates are fixed. Applying the conditional-mean identity and \eqref{eq:bias-stability} to $F_{t,k}$, and combining the $N-1$ rows, gives
\begin{align}
    \label{eq:bias-sensitivity-mean}
    \Big\|\E\big[\widehat D_t\mid \Fc_{t-1}^\mathrm{off}\big] - D_t^\theta\Big\|_\mathrm{F} \leq C\bigg(\eta + \|\widehat z - z^\theta\| + \sum_{s=1}^{t-1} \|\widehat D_s - D_s^\theta\|_\mathrm{F}\bigg).
\end{align}

For the conditional fluctuation, let $Y_{t,r}$ be the matrix contributed by rollout $r$ in the definition of $\widehat D_t$, so that $\widehat D_t=n^{-1}\sum_{r=1}^nY_{t,r}$. Write the policy scores as $\overline\Psi_s^{(t,r)}$, as defined in~\eqref{eq:algorithm-auxiliary-policy-score}. Since at most one row of $Y_{t,1}$ is nonzero,
\begin{align}
    \|Y_{t,1}\|_\mathrm{F}^2
    &\leq\bigg\|\sum_{s=0}^{t-1}\overline\Psi_s^{(t,1)}-\frac{1-\eta}{\eta}\sum_{s=1}^{t-1}\widehat D_s^\top H(\overline Q_s^{(t,1)})\bigg\|^2\\
    &\leq2\bigg\|\sum_{s=0}^{t-1}\overline\Psi_s^{(t,1)}\bigg\|^2
    +\frac{2(1-\eta)^2}{\eta^2}\bigg\|\sum_{s=1}^{t-1}\widehat D_s^\top H(\overline Q_s^{(t,1)})\bigg\|^2\\
    &\leq2t\sum_{s=0}^{t-1}\|\overline\Psi_s^{(t,1)}\|^2
    +\frac{2(t-1)}{\eta^2}\sum_{s=1}^{t-1}\|\widehat D_s\|_\mathrm{F}^2\|H(\overline Q_s^{(t,1)})\|^2.
\end{align}
The second inequality uses $\|a-b\|^2\leq2\|a\|^2+2\|b\|^2$. The last uses $\|\sum_{j=1}^L v_j\|^2\leq L\sum_{j=1}^L\|v_j\|^2$ for sums of $t$ and $t-1$ terms, respectively, together with $\|D^\top h\|\leq\|D\|_\mathrm{F}\|h\|$ and $(1-\eta)^2\leq1$. Taking conditional expectations gives
\begin{align}
    \E\big[\|Y_{t,1}\|_\mathrm{F}^2\mid\Fc_{t-1}^\mathrm{off}\big]
    &\leq 2t\sum_{s=0}^{t-1}\E\big[\|\overline\Psi_s^{(t,1)}\|^2\mid\Fc_{t-1}^\mathrm{off}\big]
    +\frac{2(t-1)}{\eta^2}\sum_{s=1}^{t-1}\|\widehat D_s\|_\mathrm{F}^2\E\big[\|H(\overline Q_s^{(t,1)})\|^2\big]\\
    &\leq\frac{C}{\eta^2}\bigg(1+\sum_{s=1}^{t-1}\|\widehat D_s\|_\mathrm{F}^2\bigg).
\end{align}
The last inequality follows from Assumption~\ref{assumption:discrete-score-regularity} and $0<\eta<1$: the policy-score bound is uniform, the earlier $\widehat D_s$ are fixed under the conditioning, and the new randomizers are independent of $\Fc_{t-1}^\mathrm{off}$. Since the $Y_{t,r}$ are conditionally i.i.d.,
\begin{align}
    \label{eq:bias-sensitivity-fluctuation}
    \E\bigg[\big\|\widehat D_t - \E\big[\widehat D_t\mid \Fc_{t-1}^\mathrm{off} \big] \big\|_\mathrm{F}^2 \Big| \Fc_{t-1}^\mathrm{off} \bigg]
    &=\frac1n\E\bigg[\big\|Y_{t,1}-\E[Y_{t,1}\mid\Fc_{t-1}^\mathrm{off}]\big\|_\mathrm{F}^2\Big|\Fc_{t-1}^\mathrm{off}\bigg]\notag\\
    &\leq\frac1n\E\big[\|Y_{t,1}\|_\mathrm{F}^2\mid\Fc_{t-1}^\mathrm{off}\big]\notag\\
    &\leq \frac{C}{n\eta^2} \bigg(1+\sum_{s=1}^{t-1}\|\widehat D_s\|_\mathrm{F}^2\bigg).
\end{align}

Set $e_t\coloneqq\E\big[\|\widehat D_t-D_t^\theta\|_{\mathrm F}^2\big]$, with $e_0=0$. Write $\widehat D_t-D_t^\theta=A_t+B_t$, with the fluctuation $A_t\coloneqq\widehat D_t-\E[\widehat D_t\mid\Fc_{t-1}^{\mathrm{off}}]$ and the conditional error $B_t\coloneqq\E[\widehat D_t\mid\Fc_{t-1}^{\mathrm{off}}]-D_t^\theta$. Since $B_t$ is $\Fc_{t-1}^{\mathrm{off}}$-measurable and $\E[A_t\mid\Fc_{t-1}^{\mathrm{off}}]=0$, the cross term vanishes and $e_t=\E\big[\|A_t\|_{\mathrm F}^2\big]+\E\big[\|B_t\|_{\mathrm F}^2\big]$. By $\|a+b\|^2\leq2\|a\|^2+2\|b\|^2$ and~\eqref{eq:discrete-exact-sensitivity-bound},
\begin{align}
    \E\big[\|\widehat D_s\|_{\mathrm F}^2\big]\leq2e_s+2\|D_s^\theta\|_{\mathrm F}^2\leq2e_s+C.
\end{align}
Squaring~\eqref{eq:bias-sensitivity-mean}, taking expectations with~\eqref{eq:bias-population-error}, and taking expectations in~\eqref{eq:bias-sensitivity-fluctuation}, we obtain
\begin{align}
    e_t \leq C\bigg(\eta^2 + \frac{1}{M} + \frac{1}{n\eta^2}\bigg) + C\bigg(1+\frac{1}{n\eta^2}\bigg)\sum_{s=1}^{t-1}e_s.
\end{align}
Since $n\eta^2 \geq 1$, induction yields
\begin{align}
    \label{eq:bias-sensitivity-error}
    \max_{t\in\llbracket1,T\rrbracket} e_t \leq C\bigg(\eta^2 + \frac{1}{M} + \frac{1}{n\eta^2}\bigg),\qquad
    \max_{t\in\llbracket1,T\rrbracket} \E\big[\|\widehat D_t\|_\mathrm{F}^2\big] \leq C.
\end{align}

\emph{Bias and mean-square error of the online estimator.}
The online stage uses new simulation randomness. Define $\widehat G_{B,M,n}^{\lambda,\eta}$ by equation~\eqref{eq:discrete-main-estimator} and set for simplicity $\widehat G \coloneqq \widehat G^{\lambda,\eta}_{B,M,n}$ Conditional on $(\widehat z,\widehat D)$, its $B$ contributions are i.i.d., and $\overline G \coloneqq \E\big[\widehat G\mid \widehat z,\widehat D\big] = \Gamma^\lambda(\widehat z,\widehat D)$. Squaring~\eqref{eq:bias-stability} with $\kappa=\lambda$, taking expectations, and applying~\eqref{eq:bias-population-error} and~\eqref{eq:bias-sensitivity-error} gives
\begin{align}
    \E\big[\|\overline G- \nabla_\theta J(\theta)\|^2\big] \leq C\bigg(\lambda^2 + \eta^2 + \frac{1}{n\eta^2} + \frac{1}{M}\bigg).
\end{align}
Since $\E[\widehat G] = \E[\overline G]$, Jensen's inequality and $\sqrt{a+b}\leq\sqrt a+\sqrt b$ imply
\begin{align}
    \big\|\E[\widehat G] -\nabla_\theta J(\theta)\big\| \leq C\bigg(\lambda + \eta + \frac{1}{\eta\sqrt{n}} + \frac{1}{\sqrt{M}}\bigg).
\end{align}
Finally, write $Y_b\coloneqq\widehat\Rc^{(b)}(\widehat\Sc_\mathrm{pol}^{(b)}+\widehat\Sc_\mathrm{mf}^{(b)})$, so that $\widehat G=\frac1B\sum_{b=1}^B Y_b$. The returns are bounded by $T\overline R+\overline g$ (Assumption~\ref{assumption:discrete-uniform-regularity}). The same sum-of-squares and matrix bounds as above, together with Assumption~\ref{assumption:discrete-score-regularity}, give
\begin{align}
    \E\big[\|\widehat\Sc_\mathrm{pol}^{(1)}\|^2\mid\widehat z,\widehat D\big]&\leq C,\\
    \E\big[\|\widehat\Sc_\mathrm{mf}^{(1)}\|^2\mid\widehat z,\widehat D\big]
    &\leq\frac{T(1-\lambda)^2}{\lambda^2}\sum_{t=1}^T\|\widehat D_t\|_\mathrm{F}^2\E\big[\|H(Q_t^{(1)})\|^2\big]
    \leq C\lambda^{-2}\sum_{t=1}^T\|\widehat D_t\|_\mathrm{F}^2.
\end{align}
Since the online randomizers are independent of $(\widehat z,\widehat D)$ and the $Y_b$ are conditionally i.i.d.,
\begin{align}
    \E\Big[\|\widehat G - \overline G\|^2 \Big| \widehat z,\widehat D \Big]
    &=\frac1B\E\big[\|Y_1-\E[Y_1\mid\widehat z,\widehat D]\|^2\mid\widehat z,\widehat D\big]\\
    &\leq\frac1B\E\big[\|Y_1\|^2\mid\widehat z,\widehat D\big]\\
    &\leq \frac{C}{B}\bigg(1 + \lambda^{-2}\sum_{t=1}^T \|\widehat D_t\|_\mathrm{F}^2\bigg).
\end{align}
Taking expectations and using~\eqref{eq:bias-sensitivity-error} bounds this term by $C(1+\lambda^{-2})/B$. Since $\widehat G-\overline G$ has conditional mean zero and $\overline G$ is $\sigma(\widehat z,\widehat D)$-measurable, the cross term vanishes and
\begin{align}
    \E\big[\|\widehat G - \nabla_\theta J(\theta)\|^2\big] &= \E\big[\|\widehat G-\overline G\|^2\big] + \E\big[\|\overline G - \nabla_\theta J(\theta)\|^2\big]\\
    &\leq C\bigg(\lambda^2 + \frac{1+\lambda^{-2}}{B} + \eta^2 + \frac{1}{n\eta^2} + \frac{1}{M}\bigg).
\end{align}
This proves the theorem.
\end{proof}

\subsection{Continuous-state space}
\label{subsec:continuous-proofs}

This appendix gives the mixture construction, assumptions, and proofs for Section~\ref{subsec:continuous}.

\subsubsection{Encoder, decoder, and numerical fitting}
\label{app:encoder-decoder}

Fix $K$. Let $\Kc_K\subset\Zc_K$ be a fixed non-empty compact fitting set. In practice, $\Kc_K$ imposes lower and upper bounds on the weights and covariance eigenvalues, bounds the means, and keeps the projected means separated so that component labels remain ordered. The decoder
\begin{align}
    \Dc_K:\widetilde\Zc_K\longrightarrow\Pc_2(\R^d),\qquad
    \Dc_K(p,m,\Sigma)=\sum_{j=1}^Kp_j\Nc(m_j,\Sigma_j),
\end{align}
is continuous in $\Wc_2$ and satisfies
\begin{align}
    M_2(\Dc_K(p,m,\Sigma))
    =\sum_{j=1}^Kp_j\big(\|m_j\|^2+\operatorname{tr}\Sigma_j\big).
    \label{eq:continuous-decoder-moment}
\end{align}
Thus the compact fitting constraints give the uniform moment bound required by the perturbation estimates. No inverse identity $\Dc_K\circ\Ec_K=\operatorname{Id}$ is needed.

\begin{Assumption}[Encoder regularity]
\label{assumption:continuous-encoder}
The encoder $\Ec_K:\Pc_2(\R^d)\to\Zc_K$ is Borel measurable and satisfies $z_{K,t}^\theta=\Ec_K(\mu_t^\theta)\in\Kc_K$. The paths $\theta\mapsto z_{K,t}^\theta$ are continuous for objective consistency, $\Cc^1$ on an open neighborhood of $\Theta$ for the policy-gradient representation, and $\Cc^3$ with bounded derivatives when the centered-difference bias is estimated.
\end{Assumption}

\paragraph{Practical encoder and numerical fit.}
Let $f_z$ denote the density of $\Dc_K(z)$. A standard choice is the constrained likelihood encoder
\begin{align}
    \Ec_K(\mu)&\in\underset{z\in\Kc_K}{\operatorname{argmax}}\;
    \ell_K(\mu,z),
    &\ell_K(\mu,z)&\coloneqq\int_{\R^d}\log f_z(x)\,\mu(\d x),
    \label{eq:continuous-likelihood-encoder}
\end{align}
with a fixed measurable tie-breaking rule. Replacing $\mu$ by an empirical law gives the constrained EM problem described by~\citet{dempster1977maximum}. The compact constraints make this fitting problem well posed; the continuity and differentiability required by Assumption~\ref{assumption:continuous-encoder} are imposed on the resulting fitted path.

For the numerical algorithm, let $\widehat\Ec_{K,L}:\Pc_2(\R^d)\to\Kc_K$ denote the fixed constrained EM implementation with $L$ particles, including its initialization, stopping rule, and label ordering. We summarize its statistical and optimization error by the following assumption.

\begin{Assumption}[Accuracy of fitted coordinates]
\label{assumption:continuous-fitting-error}
The numerical fitting maps are measurable. For the interacting particle laws $\widehat\mu_t^{\theta,L}$ used in the offline stage, the computed coordinates $\widehat z_{K,t}^{\theta,L}\coloneqq\widehat\Ec_{K,L}(\widehat\mu_t^{\theta,L})$ satisfy
\begin{align}
    \sup_{\theta\in\Theta}\max_{t\in\Tc}\E\big[\|\widehat z_{K,t}^{\theta,L}-z_{K,t}^\theta\|^2\big]\leq a_L^2,
    \qquad a_L\longrightarrow0.
    \label{eq:continuous-fitting-error}
\end{align}
\end{Assumption}
The error $a_L$ includes both particle and optimization error; the bound is also required at the shifted parameters $\theta\pm\eta e_\ell$. The centered-difference calculation in the proof of Theorem~\ref{theorem:continuous-bias-mse} gives
\begin{align}
    \E\big[\|\widehat D_t-\nabla_\theta z_{K,t}^\theta\|_F^2\big]
    \leq C_K\left(\eta^4+\frac{a_{n_0}^2}{\eta^2}\right),\qquad n_0=\frac{n}{2d_\theta}.
    \label{eq:continuous-fitting-sensitivity-error}
\end{align}
Thus consistency requires $\eta\to0$ and $a_{n_0}/\eta\to0$.

Finally, the mixture approximation error is measured by $\varepsilon_K$ in~\eqref{eq:continuous-projection-error}. Its convergence as $K$ grows is a separate approximation question and is not needed for the fixed-$K$ results.

\subsubsection{Mixture coordinates and transport map}
\label{app:continuous-construction}

The fitted coordinate space is $\Zc_K=\Delta_K^\circ\times\M_K^e\times(\S_{++}^d)^K$, and randomized coordinates belong to $\widetilde\Zc_K=\Delta_K^\circ\times(\R^d)^K\times(\S_{++}^d)^K$. A point $(p,m,\Sigma)$ represents the mixture $\sum_jp_j\Nc(m_j,\Sigma_j)$, with $\Sigma_j$ a covariance matrix. The randomizer is $R_t=(Q_t,A_t,B_t)$ and the perturbation is $\Psi_\lambda$ from~\eqref{eq:continuous-perturbation-map}.

Let $\phi_i$ be the density of $\Nc(m_i,\Sigma_i)$ and set
\[
 r_i(x)=\frac{p_i\phi_i(x)}{\sum_{k=1}^{K} p_k\phi_k(x)}.
\]
Given $a \in \Delta_K^\circ$, we define the map
\begin{align}
    [0,1] \ni u \mapsto F_a(u) := \text{ min } \big \lbrace i \in \llbracket 1,K \rrbracket : u \leq \sum_{k=1}^{i} a_k \big \rbrace . 
\end{align}
For $0<\lambda<1$, define
\begin{align}
    \iota(x,u)=F_{r(x)}(u),\qquad
v(x,u)=\frac{u-\sum_{k<\iota(x,u)}r_k(x)}{r_{\iota(x,u)}(x)},
\end{align}
\begin{align}
    \jmath^\lambda(\omega,x,u)=
\begin{cases}
\iota(x,u),&v(x,u)\leq1-\lambda,\\
F_{Q_t(\omega)}\big((v(x,u)-1+\lambda)/\lambda\big),&v(x,u)>1-\lambda.
\end{cases}
\end{align}
Define $T_{K,t}^{0,\theta}(\omega,x,u)=x$. For $0<\lambda<1$, with $i=\iota(x,u)$ and $j=\jmath^\lambda(\omega,x,u)$, set
\begin{align}
    T_{K,t}^{\lambda,\theta}(\omega,x,u)
=m_{j,K,t}^{\lambda,\theta}(\omega)+\Gamma_{j,K,t}^{\lambda,\theta}(\omega)(\Gamma_{i,K,t}^\theta)^{-1}(x-m_{i,K,t}^\theta).
\label{eq:transport_map_randomized}
\end{align}

\begin{proof}[Proof of Proposition~\ref{prop:continuous-state-transport}]

All maps in the definition of $T_{K,t}^{\lambda,\theta}$ in \eqref{eq:transport_map_randomized} are jointly measurable by construction. 
 First, at $\lambda=0$, by construction it is clear that $T_{K,t}^{0,\theta}(\omega,x,u) = x$.
 
 Fix $\omega \in \Omega$ and let $X\sim\nu_{K,t}^\theta$ and $U\sim\Uc([0,1])$ be independent. Set $I=\iota(X,U)$ and $V=v(X,U)$. Let $f : \llbracket 1, K \rrbracket \times \R^d \times [0,1] \to \R$ be a bounded measurable test map. Set $s_k(x)=\sum_{i=1}^kr_i(x)$ and $s_0(x)=0$. On each interval $[s_{k-1}(x),s_k(x)]$, substitute $u=s_{k-1}(x)+r_k(x)v$ to obtain
 \begin{align}
     \int_{\R^d}\int_0^1 f(\iota(x,u),x,v(x,u))\,\d u\,\nu_{K,t}^\theta(\d x)
&=\sum_{k=1}^K\int_{\R^d}\int_0^1 f(k,x,v)r_k(x)\,\d v\,\nu_{K,t}^\theta(\d x).
 \end{align}
 Since $r_k(x)\nu_{K,t}^\theta(\d x)=p_k\phi_k(x)\,\d x$,
 \begin{align}
     \label{eq: joint_law}
\Lc(I,X,V)(\d k,\d x,\d v)
=\sum_{i=1}^K r_i(x)\nu_{K,t}^\theta(\d x)\delta_i(\d k)\,\d v
=\sum_{i=1}^K p_i\phi_i(x)\delta_i(\d k)\,\d x\,\d v.
 \end{align}
    It follows directly from \eqref{eq: joint_law} that $\P_{I} = \sum_{k=1}^{K} p_k \delta_{k}$ that the conditional law of $X|I=k$ for any $k \in \llbracket 1, K \rrbracket$ is given by $\Lc \big(X|I=k \big)(\d x) =  \phi_k(x) \d x = \Nc(m_k,\Sigma_k)(\d x)$ and that $V \sim \Uc([0,1])$ while being independent of $(I,X)$. Hence, it follows that for fixed $\omega \in \Omega$ 
    the law of the target index $J:=\jmath^\lambda(\omega,X,U)$ satisfies
    \begin{align}
        \P(J=j|I=i)= (1-\lambda) \mathbf{1}_{j=i} + \lambda q_{j,t}(\omega).
    \end{align}
    Conditionally on $I=i$ and $J=j$, the vector $(\Gamma_{i,K,t}^{\theta})^{-1}(X-m_{i,K,t}^{\theta})$ is standard Gaussian. Consequently, the transported vector has law $\Nc(m_{j,K,t}^{\lambda,\theta}(\omega),\Sigma_{j,K,t}^{\lambda,\theta}(\omega))$ under this conditioning. Summing over $i$ gives the weight $(1-\lambda)p_j+\lambda q_{j,t}(\omega)$,  proving the claim.
\end{proof}

\subsubsection{Perturbation estimate}
\label{proof:continuous-perturbation-estimate}

\begin{Assumption}[Moment bounds]
\label{assumption:continuous-moment-bounds}
There exists a uniform constant $C_K<\infty$  in $(t,\theta,j ) \in \Tc \times \Theta \times \llbracket 1, K \rrbracket $ such that the following holds
\begin{align}
\begin{cases}
    \lVert m_{j,K,t}^\theta \rVert^2+ \lVert \Sigma _{j,K,t}^\theta\|_F^2&\leq C_K, \\
    \E \Big[ \underset{j \in \llbracket 1, K \rrbracket}{\text{ max }} (\lVert A_{j,t} \Vert^2 + \lVert B_{j,t} \rVert^2_{F} ) \Big ] &\leq C_K.
\end{cases}
\end{align}
\end{Assumption}

\begin{proof}[Proof of Proposition~\ref{proposition:continuous-perturbation-estimate}]
Fix $t,\theta$ and suppress these indices. We use a coupling of the same source and target mixtures; it need not be the symmetric-square-root transport used in Proposition~\ref{prop:continuous-state-transport}. Independently of $R_t$, draw $I$ with weights $p$, a standard Gaussian vector $Z$, and $V\sim\Uc([0,1])$. If $V\leq1-\lambda$, set $J=I$; otherwise draw $J$ with weights $Q_t$, using additional independent randomness. Put
\begin{align}
 X&=m_I+\Sigma_I^{1/2}Z,\qquad
 Y=m_J^\lambda+\Sigma_J^{1/2}((1-\lambda)I_d+\lambda B_{J,t})Z.
\end{align}
Conditionally on $R_t$, this couples $\nu_{K,t}^\theta$ with $M_{K,t}^{\lambda,\theta}$. Indeed, $J$ has weights $(1-\lambda)p+\lambda Q_t$, and its conditional component covariance is precisely the one in~\eqref{eq:continuous-perturbation-map}.

On $\{V\leq1-\lambda\}$,
\begin{align}
 Y-X=\lambda(A_{I,t}-m_I)+\lambda\Sigma_I^{1/2}(B_{I,t}-I_d)Z.
\end{align}
The moment bounds and $\|\Sigma_I^{1/2}(B_{I,t}-I_d)\|_{\mathrm F}^2
\leq\|\Sigma_I\|_{\mathrm{op}}\|B_{I,t}-I_d\|_{\mathrm F}^2$ give a conditional mean squared displacement bounded by $C\lambda^2$. Here and below $C$ depends only on the stated moment bounds and the fixed dimension. On $\{V>1-\lambda\}$, those same bounds, using the maximum over components for $A_t,B_t$, give $\E[\|X\|^2+\|Y\|^2\mid V>1-\lambda]\leq C$. Hence
\begin{align}
 \E[\Wc_2^2(M_{K,t}^{\lambda,\theta},\nu_{K,t}^\theta)]
 \leq\E[\|Y-X\|^2]\leq C((1-\lambda)\lambda^2+\lambda)\leq C\lambda.
\end{align}
The Wasserstein triangle inequality followed by Minkowski's inequality yields
\begin{align}
 \big(\E[\Wc_2^2(M_{K,t}^{\lambda,\theta},\mu_t^\theta)]\big)^{1/2}
 \leq\Wc_2(\nu_{K,t}^\theta,\mu_t^\theta)
 +\big(\E[\Wc_2^2(M_{K,t}^{\lambda,\theta},\nu_{K,t}^\theta)]\big)^{1/2},
\end{align}
which proves the claim. When $K=1$, the component-switching construction is unnecessary, and the same coupling gives the  bound $\big(\E[\Wc_2^2(M_{1,t}^{\lambda,\theta},\nu_{1,t}^\theta)]\big)^{1/2}\leq C\lambda$.
\end{proof}


\subsubsection{Objective perturbation consistency}
\label{proof:continuous-objective-consistency}

\begin{Assumption}[Objective regularity]
    \label{assumption:continuous-objective-regularity}
    We make the following assumptions
    \begin{enumerate}
        \item [(1)] The randomizers $(R_t)_{t \in \Tc}$ are independent across time and independent of the simulation randomness
        \item [(2)] The initial law satisfies $\mu_0 \in \Pc_2(\R^d)$. Furthermore, for any $t\in\Tc,\theta\in\Theta$, and $\eta,m\in\Pc_2(\R^d)$, the measure $\Phi_t^\theta(\eta,m)$ belongs to $\Pc_2(\R^d)$.
        \item [(3)] There exists a positive constant $L$ such that
        \begin{align} \begin{cases}\Wc_2\big(\Phi_t^\theta(\eta,m), \Phi_t^\theta(\eta',m') \big) &\leq L\big(\Wc_2(\eta,\eta') + \Wc_2(m,m') \big),\\
        \big|\overline R_t^\theta(\eta,m) - \overline R_t^\theta(\eta',m') \big| &\leq L\left(\Wc_2(\eta,\eta') + \Wc_2(m,m') \right),\\
        \big|\overline g(\eta,m) - \overline g(\eta',m') \big| &\leq L\left(\Wc_2(\eta,\eta') + \Wc_2(m,m') \right).
    \end{cases}
    \end{align}
      uniformly in $\theta,t,\eta,\eta',m,m'$.
    \end{enumerate}
\end{Assumption}
Note that the construction in Appendix \ref{appendix:canonical-construction} ensures $(1)$.

The quadratic rewards of the linear--quadratic and portfolio benchmarks are not globally Wasserstein-Lipschitz, so these benchmarks lie outside the scope of the objective consistency, projection-error, and optimality-gap guarantees under Assumption~\ref{assumption:continuous-objective-regularity}.

\begin{proof}[Proof of Theorem~\ref{theorem:continuous-objective-perturbation-consistency}]
    Fix $\theta \in \Theta$, let $\eta_{K,t}^{\lambda,\theta} \coloneqq\Lc(X_t^{K,\lambda,\theta})$, and set $e_t :=\Wc_2(\eta_{K,t}^{\lambda,\theta},\eta_{K,t}^{0,\theta})$. Independence of the current randomizer and state gives
    \begin{align}
    \begin{cases}
       \eta_{K,t+1}^{\lambda,\theta} &= \E \big[ \Phi_t^{\theta}(\eta_{K,t}^{\lambda,\theta}, M_{K,t}^{\lambda,\theta}) \big], \\
       \eta_{K,t+1}^{0,\theta} &= \Phi_t^{\theta}(\eta_{K,t}^{0,\theta}, \nu_{K,t}^{\theta})
    \end{cases}
    \end{align}

    The squared Wasserstein distance is convex under mixtures. To verify this, let $\gamma_\omega$ be an optimal coupling between $\Phi_t^\theta(\eta_{K,t}^{\lambda,\theta}, M_{K,t}^{\lambda,\theta}(\omega))$ and $\eta_{K,t+1}^{0,\theta}$ for a given realization $\omega$. Taking the expectation $\E[\gamma_\omega]$ yields a valid coupling between the expected measure $\eta_{K,t+1}^{\lambda,\theta}$ and the deterministic measure $\eta_{K,t+1}^{0,\theta}$. Because the Wasserstein distance is the infimum over all valid couplings, it is bounded by the cost of this averaged coupling:
    \begin{align}
        \Wc_2^2 \big(\eta_{K,t+1}^{\lambda,\theta}, \eta_{K,t+1}^{0,\theta} \big) \leq \E \Big[ \Wc_2^2 \big(\Phi_{t}^{\theta}(\eta_{K,t}^{\lambda,\theta},M_{K,t}^{\lambda,\theta}), \Phi_t^{\theta}(\eta_{K,t}^{0,\theta}, \nu_{K,t}^{\theta}) \big) \Big].
    \end{align}
    Using the Lipschitz continuity of $\Phi_t^\theta$ (Assumption~\ref{assumption:continuous-objective-regularity}), this implies:
    \begin{align}
        e_{t+1}^2 \leq L^2 \E \Big[ \big(e_t + \Wc_2(M_{K,t}^{\lambda,\theta}, \nu_{K,t}^{\theta}) \big)^2 \Big]. 
    \end{align}
    Taking the square root and applying Minkowski's inequality with $p=2$ separates the terms:
    \begin{align}
        e_{t+1} \leq L \Big( e_t + \big( \E\big[\Wc_2^2(M_{K,t}^{\lambda,\theta}, \nu_{K,t}^{\theta})\big] \big)^{1/2} \Big).
    \end{align}
    Setting $\delta \coloneqq C_K\sqrt{\lambda}$, the perturbation estimate from Proposition~\ref{proposition:continuous-perturbation-estimate} gives $e_{t+1} \leq L(e_t + \delta)$. Since $e_0 = 0$, a direct induction yields $e_t \leq \delta\sum_{s=1}^T L^s$ for any $t\geq1$
.   
    Finally, applying the Lipschitz bounds on the rewards $\overline R_t^\theta$ and terminal cost $\overline g$ from Assumption~\ref{assumption:continuous-objective-regularity}, and summing over time, yields:
    \begin{align}
        |J_K^\lambda(\theta)-J_K(\theta)| \leq L \sum_{t=0}^{T} \Big(e_t + \E\big[\Wc_2(M_{K,t}^{\lambda,\theta}, \nu_{K,t}^{\theta})\big]\Big) \leq L \sum_{t=0}^{T} (e_t + \delta) \leq C_{K,T}\sqrt\lambda,
    \end{align}
    uniformly in $\theta$, for some constant $C_{K,T} > 0$ depending only on $L, T$, and $C_K$.
\end{proof}

\paragraph{Projection error and maximizers.}
To prove~\eqref{eq:continuous-projection-error}, set $a_t\coloneqq\Wc_2(\Lc(X_t^{K,0,\theta}),\mu_t^\theta)$ and $\delta_t\coloneqq\Wc_2(\nu_{K,t}^\theta,\mu_t^\theta)$. The original and projected dynamics give
\begin{align}
    a_{t+1}\leq L(a_t+\delta_t),\qquad a_0=0.
\end{align}
Since $\delta_t\leq\varepsilon_K$, induction yields $a_t\leq\varepsilon_K\sum_{s=1}^tL^s$. Applying the reward bounds gives
\begin{align}
    |J_K(\theta)-J(\theta)|
    \leq L\sum_{t=0}^T(a_t+\delta_t)
    \leq C_T^{\mathrm{proj}}\varepsilon_K,\qquad
    C_T^{\mathrm{proj}}\coloneqq L\sum_{t=0}^T\left(1+\sum_{s=1}^tL^s\right).
\end{align}
The second bound in~\eqref{eq:continuous-projection-error} follows by the triangle inequality and Theorem~\ref{theorem:continuous-objective-perturbation-consistency}.

\begin{proof}[Proof of Proposition~\ref{prop:continuous-optimality-gap}]
    For any $\theta \in \Theta$, we have
    \begin{align}
        J(\theta) - J(\theta^{\star}_{K,\lambda}) &= J(\theta) -J_K^{\lambda}(\theta)  + J_K^{\lambda}(\theta) -  J_K^{\lambda}(\theta^{\star}_{K,\lambda}) + J_K^{\lambda}(\theta^{\star}_{K,\lambda} ) - J(\theta^{\star}_{K,\lambda}) \\
        &\leq 2 \underset{\theta \in \Theta}{\text{ sup }} |J(\theta) - J_K^{\lambda}(\theta)|,
    \end{align}
    where we used that $J_K^{\lambda}(\theta) \leq J_K^{\lambda}(\theta^{\star}_{K,\lambda})$. 
    Hence, by taking the supremum over $\theta \in \Theta$, in the left hand-side, it follows that
    \begin{align}
        0 \leq \underset{\theta \in \Theta}{\text{sup } } J(\theta) - J(\theta^{\star}_{K,\lambda}) \leq  2 \big( C_T^{\text{proj}} \varepsilon_K + C_{K,T} \sqrt{\lambda} \big).
    \end{align}
\end{proof}

\subsubsection{Objective gradient convergence}
\label{proof:continuous-gradient-convergence}

Fix $K$. For a prescribed coordinate flow $z= (z_1,\ldots,z_T) \in \widetilde \Zc_K^T$, set $m_0 = \mu_0$ and $m_t = \Dc_K(z_t)$ for $t\geq 1$. Define $v_t : \Theta^\circ \times \widetilde \Zc_K^T \to \Pc_2(\R^d)$ recursively by
\begin{align}
    v_0^\theta(z) = \mu_0,\qquad
    v_{t+1}^\theta(z) = \Phi_t^\theta(v_t^\theta(z),m_t).
\end{align}
The objective for this prescribed population flow is $F_K:\Theta^\circ \times \widetilde \Zc_K^T \to \R$, defined by
\begin{align}
    F_K(\theta,z) = \sum_{t=0}^{T-1} \overline R_t^\theta (v_t^\theta(z),m_t) + \overline g(v_T^\theta(z),m_T).
\end{align}
Write $z^\theta = (z_{K,t}^\theta)_{t=1}^T$, $Z^{\lambda,\theta} = \big(\Psi_\lambda(z_{K,t}^\theta,R_t)\big)_{t=1}^T$, and $D^\theta = \nabla_\theta z^\theta$. Conditioning on the randomizers gives
\begin{align}
    J_K(\theta) = F_K(\theta,z^\theta),\qquad
    J_K^\lambda(\theta) = \E\big[F_K(\theta,Z^{\lambda,\theta})\big].
\end{align}
Throughout, \(\nabla_\theta F_K\) denotes differentiation with the prescribed flow held fixed.

\begin{Assumption}[Continuous-state gradient regularity]
    \label{assumption:continuous-gradient-regularity}
    Assume that:
    \begin{enumerate}
        \item [(1)] \emph{Regularity of the fitted flow.} The parameter set $\Theta$ is compact, and $\theta \mapsto z^\theta = (\Ec_K(\mu_t^\theta))_{t=1}^T$ extends to a $\Cc^1$ map on an open neighborhood $\Theta^\circ$ of $\Theta$, with values in $\Zc_K^T$. The initial law \(\mu_0\) is independent of \(\theta\).
        
        \item [(2)] \emph{Regularity of the prescribed-flow objective.}  The map $F_K$ is $\Cc^1$ on the open set $\Theta^\circ \times \widetilde \Zc_K^T \subset \R^{d_\theta + Tq_K}$. There exist \(L<\infty\) and \(r\geq0\) such that, uniformly in \(\theta\in\Theta^\circ\) and \(z,z'\in\widetilde{\mathcal Z}_K^T\),
        \begin{align}
            \|\nabla F_K(\theta,z)\| &\leq L(1+\|z\|)^r,\\
            \|\nabla F_K(\theta,z) - \nabla F_K(\theta,z')\| &\leq L(1+\|z\|+\|z'\|)^r\|z-z'\|.
        \end{align}
        Here \(\nabla F_K\) contains the derivatives with respect to both \(\theta\) and \(z\).
        \item [(3)] \emph{Randomizer moments.} Assume Assumption~\ref{assumption:continuous-moment-bounds} and condition~(1) of Assumption~\ref{assumption:continuous-objective-regularity}. The randomizers have a common law $\rho$, independent of $\theta$ and $\lambda$. For some $\lambda_0\in(0,1)$, this law satisfies condition~(1) of Assumption~\ref{assumption: randomizer_rho} for every $0<\lambda\leq\lambda_0$. Moreover, for each $t\in\llbracket1,T\rrbracket$,
        \begin{align}
            \E\Big[\max_{1\leq j\leq K} \big(\|A_{j,t}\|^{r+1} + \|B_{j,t}\|_\mathrm{F}^{2(r+1)}\big) \Big] < \infty.
        \end{align}
    \end{enumerate}
\end{Assumption}
Condition~(2) is an additional regularity assumption on the prescribed-flow objective; it does not follow from the Wasserstein Lipschitz bounds of Assumption~\ref{assumption:continuous-objective-regularity}. It requires polynomial control at infinity and uniform control as weights or covariance eigenvalues approach zero.
The case \(r=0\) is the bounded, Lipschitz derivative assumption used in the discrete proof.

\begin{proof}[Proof of Theorem~\ref{theorem:continuous-gradient-convergence}]
    By compactness and the \(\Cc^1\) regularity of the fitted flow, its coordinates and first derivatives are uniformly bounded. Moreover, because its covariance matrices are positive definite, there exists \(c_K>0\) such that $\Sigma_{j,K,t}^\theta \succeq c_K I_d$ uniformly in $\theta,t,j$. These bounds also hold on a sufficiently small neighborhood of $\Theta$.

    Set 
    \begin{align}
        W \coloneqq 1 + \sum_{t=1}^T \sum_{j=1}^K \big(\|A_{j,t}\| + \|B_{j,t}\|_\mathrm{F}^2\big).
    \end{align}
    The moment assumption gives $\E\big[W^{r+1}\big] < \infty$. We first establish
    \begin{align}
        \label{eq:continuous-gradient-bound1}
        \|Z^{\lambda,\theta} - z^\theta \|+ \|\nabla_\theta Z^{\lambda,\theta} - D^\theta\| \leq C\lambda W.
    \end{align}
    For the weights and means, this follows directly from their affine perturbations. For the covariance block $(j,t)$, recall $\Gamma_{j,K,t}^\theta = (\Sigma_{j,K,t}^\theta)^{1/2}$ and set
    \begin{align}
        V_{j,t}^\lambda \coloneqq 2\lambda(B_{j,t}-I_d) + \lambda^2(B_{j,t} - I_d)^2.
    \end{align}
    By~\eqref{eq:continuous-perturbation-map} and $\Sigma_{j,K,t}^\theta=\Gamma_{j,K,t}^\theta I_d\Gamma_{j,K,t}^\theta$,
    \begin{align}
        \Sigma_{j,K,t}^{\lambda,\theta} - \Sigma_{j,K,t}^\theta
        &=\Gamma_{j,K,t}^\theta\Big(\big((1-\lambda)I_d+\lambda B_{j,t}\big)^2-I_d\Big)\Gamma_{j,K,t}^\theta\\
        &=\Gamma_{j,K,t}^\theta\Big(2\lambda(B_{j,t}-I_d)+\lambda^2(B_{j,t}-I_d)^2\Big)\Gamma_{j,K,t}^\theta\\
        &=\Gamma_{j,K,t}^\theta V_{j,t}^\lambda\Gamma_{j,K,t}^\theta.
    \end{align}
    Since $\lambda^2\leq\lambda$ and $\|(B_{j,t}-I_d)^2\|_\mathrm{F}\leq\|B_{j,t}-I_d\|_\mathrm{F}^2$,
    \begin{align}
        \|V_{j,t}^\lambda\|_\mathrm{F} \leq C\lambda(1 + \|B_{j,t}\|_\mathrm{F}^2).
    \end{align}
    For each parameter coordinate $\theta_\ell$, differentiating $(\Gamma_{j,K,t}^\theta)^2 = \Sigma_{j,K,t}^\theta$ gives
    \begin{align}
        \Gamma_{j,K,t}^\theta \partial_{\theta_\ell} \Gamma_{j,K,t}^\theta + (\partial_{\theta_\ell}\Gamma_{j,K,t}^\theta)\Gamma_{j,K,t}^\theta = \partial_{\theta_\ell}\Sigma_{j,K,t}^\theta.
    \end{align}
    The covariance lower bound implies
    \begin{align}
        \|\partial_{\theta_\ell}\Gamma_{j,K,t}^\theta\|_\mathrm{F} \leq \frac{1}{2\sqrt{c_K}} \|\partial_{\theta_\ell}\Sigma_{j,K,t}^\theta\|_\mathrm{F}.
    \end{align}
    Consequently, both \(\Gamma_{j,K,t}^\theta\) and its first derivatives are uniformly bounded. Since $B_{j,t}$, and hence $V_{j,t}^\lambda$, does not depend on $\theta$, the product rule gives
    \begin{align}
        \partial_{\theta_\ell}\Sigma_{j,K,t}^{\lambda,\theta} - \partial_{\theta_\ell}\Sigma_{j,K,t}^\theta
        = (\partial_{\theta_\ell}\Gamma_{j,K,t}^\theta)V_{j,t}^\lambda\Gamma_{j,K,t}^\theta + \Gamma_{j,K,t}^\theta V_{j,t}^\lambda(\partial_{\theta_\ell}\Gamma_{j,K,t}^\theta),
    \end{align}
    so that, using $\|XYX'\|_\mathrm{F}\leq\|X\|_\mathrm{F}\|Y\|_\mathrm{F}\|X'\|_\mathrm{F}$,
    \begin{align}
        \|\partial_{\theta_\ell}\Sigma_{j,K,t}^{\lambda,\theta} - \partial_{\theta_\ell}\Sigma_{j,K,t}^\theta\|_\mathrm{F}
        \leq 2\|\partial_{\theta_\ell}\Gamma_{j,K,t}^\theta\|_\mathrm{F}\|\Gamma_{j,K,t}^\theta\|_\mathrm{F}\|V_{j,t}^\lambda\|_\mathrm{F}.
    \end{align}
    Together with $\|\Sigma_{j,K,t}^{\lambda,\theta} - \Sigma_{j,K,t}^\theta\|_\mathrm{F}\leq\|\Gamma_{j,K,t}^\theta\|_\mathrm{F}^2\|V_{j,t}^\lambda\|_\mathrm{F}$ and the bound on $V_{j,t}^\lambda$, this yields
    \begin{align}
        \|\Sigma_{j,K,t}^{\lambda,\theta} - \Sigma_{j,K,t}^\theta\|_\mathrm{F} + \|\nabla_\theta \Sigma_{j,K,t}^{\lambda,\theta} - \nabla_\theta \Sigma_{j,K,t}^\theta\|\leq C\lambda(1+  \|B_{j,t}\|_\mathrm{F}^2).
    \end{align}
    Summing over the finitely many blocks proves~\eqref{eq:continuous-gradient-bound1}.

    In particular,
    \begin{align}
        \|Z^{\lambda,\theta}\| + \|\nabla_\theta Z^{\lambda,\theta}\| \leq CW.
    \end{align}
    By the chain rule and the growth assumption in Assumption~\ref{assumption:continuous-gradient-regularity}(2), using $W\geq1$,
    \begin{align}
        \big\|\nabla_\theta\big[F_K(\theta,Z^{\lambda,\theta})\big]\big\|
        &= \big\|\nabla_\theta F_K(\theta,Z^{\lambda,\theta}) + (\nabla_\theta Z^{\lambda,\theta})^\top \nabla_z F_K(\theta,Z^{\lambda,\theta})\big\|\\
        &\leq \big(1+\|\nabla_\theta Z^{\lambda,\theta}\|\big)\,L\big(1+\|Z^{\lambda,\theta}\|\big)^r
        \leq L(1+CW)^{r+1}\leq CW^{r+1},
    \end{align}
    which is integrable. This bound is uniform in $\theta$ in a neighborhood of $\Theta$. By the mean value theorem, the difference quotients of $\theta\mapsto F_K(\theta,Z^{\lambda,\theta})$ are therefore also dominated by $CW^{r+1}$. Dominated convergence then justifies differentiation under the expectation, giving
    \begin{align}
        \nabla_\theta J_K^\lambda(\theta) &= \E\big[\nabla_\theta F_K(\theta,Z^{\lambda,\theta}) + (\nabla_\theta Z^{\lambda,\theta})^\top \nabla_z F_K(\theta,Z^{\lambda,\theta})\big],\\
        \nabla_\theta J_K(\theta) &= \nabla_\theta F_K(\theta,z^\theta) + (D^\theta)^\top \nabla_z F_K(\theta,z^\theta).
    \end{align}
    Subtracting and adding $(D^\theta)^\top \E\big[\nabla_zF_K(\theta,Z^{\lambda,\theta})\big]$ gives
    \begin{align}
        \nabla_\theta J_K^\lambda - \nabla_\theta J_K &= \E\big[\nabla_\theta F_K(\theta,Z^{\lambda,\theta}) - \nabla_\theta F_K(\theta,z^\theta)\big] + (D^\theta)^\top\E\big[\nabla_z F_K(\theta,Z^{\lambda,\theta}) - \nabla_z F_K(\theta,z^\theta) \big]\\
        &\quad + \E\big[(\nabla_\theta Z^{\lambda,\theta} - D^\theta)^\top \nabla_z F_K(\theta,Z^{\lambda,\theta}) \big].
    \end{align}
    Using~\eqref{eq:continuous-gradient-bound1}, the derivative bounds, and the uniform bound on $D^\theta$, each term is bounded by $C\lambda \E\big[W^{r+1}\big]$. The estimate is uniform in $\theta$, proving the result.
\end{proof}

Under the stronger Assumption~\ref{assumption:continuous-gradient-regularity}, the objective error is also $O(\lambda)$: the mean value theorem and~\eqref{eq:continuous-gradient-bound1} give $|J_K^\lambda(\theta)-J_K(\theta)|\leq\E|F_K(\theta,Z^{\lambda,\theta})-F_K(\theta,z^\theta)|\leq C\lambda\E[W^{r+1}]$, uniformly in $\theta$.

\subsubsection{Density and score for the randomized perturbation}
\label{proof:continuous-perturbation-density}

\begin{Assumption}[Assumptions on  $\rho$ the law of the randomizer $(R_t)_{t \in \llbracket 1, T \rrbracket}$]\label{assumption: randomizer_rho}
Let $\Bc_{\lambda} := \big \lbrace B \in \S_{d} : B \succ - \frac{1-\lambda}{\lambda} I_d \big \rbrace $. We make the following assumptions.
\begin{enumerate}
    \item [(1)] $\rho \big( \Delta_K^{\circ} \times (\R^d)^K \times (\Bc_{\lambda})^K \big)=1$.
    \item [(2)] $\rho$ admits a density $f_{(Q,A,B)}$ with respect the Lebesgue measure over $\Delta_K^\circ \times (\R^d)^K \times (\Bc_{\lambda})^K$. Moreover, we assume the following factorization
    \begin{align}
        f_{(Q,A,B)}(q,a,b) = f_Q(q) f_A(a) f_B(b),
    \end{align}
    for some maps $f_Q,f_A,f_B$
    implying that $(Q_t,A_t,B_t)$ are mutually independent.
    \item [(3)] For $K\geq2$, $Q_t=\varphi(U_t)$, where the density of $U_t$ satisfies Assumption~\ref{assumption:randomizer_rho_discrete} with $N=K$. For $K=1$, set $Q_t=1$.
\end{enumerate}
\end{Assumption}

\begin{Lemma}[Jacobian determinant]
\label{lemma:covariance-jacobian}
Fix $\Sigma\in\S_{++}^d$ and $\lambda\in(0,1)$. Define $F:\Bc_\lambda\to\S_{++}^d$ by
\[
    F(B)\coloneqq\Sigma^{1/2}\big((1-\lambda)I_d+\lambda B\big)^2
    \Sigma^{1/2},
    \qquad B\in\Bc_\lambda.
\]
For $B\in\Bc_\lambda$, set $C\coloneqq(1-\lambda)I_d+\lambda B$. The differential of $F$ at $B$ is the linear map
\begin{align}
    DF(B):\S^d\to\S^d,\qquad H\mapsto\lambda\Sigma^{1/2}(CH+HC)\Sigma^{1/2}.
\end{align}
Its determinant, computed in the same independent-entry coordinates on the domain and codomain, is
\begin{align}
    \det DF(B) = \lambda^{d(d+1)/2}(\det\Sigma)^{(d+1)/2} \prod_{i=1}^d(2c_i)  \prod_{1\leq i<j\leq d}(c_i+c_j).
\end{align}
where $c_1,\ldots,c_d>0$ are the eigenvalues of $C$.
\end{Lemma}

\begin{proof}
    Set $S\coloneqq\Sigma^{1/2}$. For symmetric $H$ small enough that $B+H\in\Bc_\lambda$, expansion gives
    \begin{align}
        F(B+H)-F(B)=\lambda S(CH+HC)S+\lambda^2SH^2S.
    \end{align}
    The last term is $O(\|H\|_\mathrm{F}^2)$, which proves the formula for $DF(B)$. Define the linear maps $\Lc_C,\Cc:\S^d\to\S^d$ by $\Lc_C(H)\coloneqq CH+HC$ and $\Cc(H)\coloneqq SHS$. Then
    \begin{align}
        DF(B)(H)=\lambda\,\Cc\big(\Lc_C(H)\big),\qquad
        \det DF(B)=\lambda^{d(d+1)/2}\det\Cc\,\det\Lc_C,
    \end{align}
    since $\dim\S^d=d(d+1)/2$.
    Let $(u_1,\ldots,u_d)$ be an orthonormal eigenbasis of $C$, with $Cu_i=c_iu_i$, and set $E_{i,j}\coloneqq u_iu_j^\top$. The matrices $E_{i,i}$ and $E_{i,j}+E_{j,i}$, for $i<j$, form a basis of $\S^d$. Since $CE_{i,j}=c_iE_{i,j}$ and $E_{i,j}C=c_jE_{i,j}$,
    \begin{align}
        \Lc_C(E_{i,i})=2c_iE_{i,i},\qquad
        \Lc_C(E_{i,j}+E_{j,i})=(c_i+c_j)(E_{i,j}+E_{j,i}).
    \end{align}
    Thus $\Lc_C$ is diagonal in this basis, and, since the determinant of a linear map does not depend on the basis,
    \begin{align}
        \det\Lc_C=\prod_{i=1}^d(2c_i)
        \prod_{1\leq i<j\leq d}(c_i+c_j).
    \end{align}
    For $\Cc$, let $a_1,\ldots,a_d>0$ be the eigenvalues of $S$, and define $E_{i,j}$ as above from an orthonormal eigenbasis of $S$. Since $SE_{i,j}S=a_ia_jE_{i,j}$,
    \begin{align}
        \Cc(E_{i,i})=a_i^2E_{i,i}, \qquad
        \Cc(E_{i,j}+E_{j,i}) =a_ia_j(E_{i,j}+E_{j,i}).
    \end{align}
    Hence
    \begin{align}
        \det\Cc=\prod_{i=1}^da_i^2\prod_{1\leq i<j\leq d}a_ia_j=\prod_{i=1}^d a_i^{d+1}=(\det S)^{d+1}=(\det\Sigma)^{(d+1)/2},
    \end{align}
    since each index appears twice in the first product and $d-1$ times in the second. Substituting the two determinants into the composition formula proves the result.
\end{proof}

\begin{Lemma}[Explicit perturbation density]
\label{lemma: Law_z_k}
Under the conditions of Lemma~\ref{lemma:continuous-main-density}, fix $z=(p,m,\Sigma)\in\widetilde\Zc_K$ and $\lambda\in(0,1)$. The inverse coordinates are given by~\eqref{eq:continuous-inverse-coordinates}.
For $b\in(\Bc_\lambda)^K$, let $c_{j,1},\ldots,c_{j,d}$ denote the eigenvalues of $(1-\lambda)I_d+\lambda b_j$. Define the Jacobian map $\mathcal J^\lambda:\widetilde\Zc_K\times(\Bc_\lambda)^K\to(0,\infty)$ by
\begin{align}
    \mathcal J^\lambda(z,b)
    \coloneqq\lambda^{q_K}\prod_{j=1}^K
    \bigg[(\det\Sigma_j)^{(d+1)/2}
    \prod_{i=1}^d(2c_{j,i})\prod_{1\leq i<\ell\leq d}(c_{j,i}+c_{j,\ell})\bigg].
    \label{eq:continuous-explicit-jacobian}
\end{align}
Then, for $y\in Y_\lambda(z)$,
\begin{align}
    h^\lambda(y\mid z)
    =\frac{f_Q(q^\lambda(y\mid z))f_A(a^\lambda(y\mid z))f_B(b^\lambda(y\mid z))}
    {\mathcal J^\lambda(z,b^\lambda(y\mid z))},
    \label{eq:continuous-explicit-density}
\end{align}
and the density is zero elsewhere. For $K=1$, the factor $f_Q$ and the free weight coordinates are omitted. In particular, $Z_{K,t}^{\lambda,\theta}$ has density $h^\lambda(\cdot\mid z_{K,t}^\theta)$.
\end{Lemma}

\begin{proof}
    For the fixed $z=(p,m,\Sigma)$, consider the map $\Psi_\lambda(z,\cdot)$ from~\eqref{eq:continuous-perturbation-map}:
    \begin{align}
        \Psi_\lambda(z,(q,a,b))=\Big(
        (1-\lambda)p+\lambda q,\,
        (1-\lambda)m+\lambda a,
        \big(
            \Sigma_j^{1/2}
            ((1-\lambda)I_d+\lambda b_j)^2
            \Sigma_j^{1/2}
        \big)_{j=1}^K
    \Big).
    \end{align}
    For $b_k\in\Bc_\lambda$, the matrix $(1-\lambda)I_d+\lambda b_k$ is positive definite. The square map is a smooth diffeomorphism of $\S_{++}^d$ onto itself, with inverse given by the positive definite square root. Thus the displayed map is a smooth diffeomorphism from $\Delta_K^\circ\times(\R^d)^K\times(\Bc_\lambda)^K$ onto $\Yc_t^{\lambda,\theta}\times(\R^d)^K\times(\S_{++}^d)^K$, with the inverse given by~\eqref{eq:continuous-inverse-coordinates}.

    For each $j$, let $F_j$ be the covariance map $F$ of Lemma~\ref{lemma:covariance-jacobian} with $\Sigma=\Sigma_j$. Each output block depends only on its corresponding input block. Hence, in the free weight, mean, and independent covariance coordinates, the Jacobian matrix is block diagonal:
    \begin{align}
        D_{(q,a,b)}\Psi_\lambda(z,(q,a,b))
        =\operatorname{diag}\big(\lambda I_{K-1},\lambda I_{Kd},DF_1(b_1),\ldots,DF_K(b_K)\big).
    \end{align}
    Its determinant is therefore the product of the block determinants. By Lemma~\ref{lemma:covariance-jacobian},
    \begin{align}
        \det D_{(q,a,b)}\Psi_\lambda(z,(q,a,b))
        &=\lambda^{K-1}\lambda^{Kd}\prod_{j=1}^K\det DF_j(b_j)\\
        &=\lambda^{(K-1)+Kd+Kd(d+1)/2}\prod_{j=1}^K
        \bigg[(\det\Sigma_j)^{(d+1)/2}\prod_{i=1}^d(2c_{j,i})\prod_{1\leq i<\ell\leq d}(c_{j,i}+c_{j,\ell})\bigg]\\
        &=\mathcal J^\lambda(z,b)>0,
    \end{align}
    since $q_K=(K-1)+Kd+Kd(d+1)/2$. For $K=1$, the weight block is omitted. The change-of-variables formula, together with $f_{(Q,A,B)}=f_Qf_Af_B$, now gives~\eqref{eq:continuous-explicit-density}.
\end{proof}

\subsubsection{Policy-gradient representation}
\label{proof:continuous-pg-representation}

\begin{Assumption}[Likelihood-ratio representation]
\label{assumption:continuous-likelihood-ratio-regularity}
\
\begin{enumerate}
    \item [(1)]  For every $t\in\llbracket 0,T-1\rrbracket$ and for every $(x,\mu,a) \in \R^d \times \Pc_2(\R^d) \times \Ac$, and every admissible $z \in \widetilde{\Zc}_K$, the maps
$
\theta\longmapsto p_t^\theta(a\mid x,\mu)$ and 
$\theta\longmapsto h^\lambda(y\mid z_{K,t+1}^\theta)$ 
are of class $\Cc^1(\Theta^\circ)$, where $\Theta \subset \Theta^{\circ} $, is an open set and strictly positive. This requires the map $\theta \mapsto z_{K,t}^{\theta} \in \widetilde{\Zc}_K$ to be $\Cc^1(\Theta^{\circ})$. Moreover, they satisfy sufficient regularity conditions to justify differentiation under the integral sign and the corresponding score-function identities, namely existence of an integrable function $G_\theta$ such that the relevant derivatives are uniformly dominated by $G_\theta$. 
   \item [(2)] Moreover, we assume that
    \begin{align}
       \underset{\theta \in \Theta}{\text{ sup }} \E \Big[\sum_{t=0}^{T-1} \lVert \nabla_{\theta} \text{ log } p_t^{\theta}(X_t, \Dc_{K}(Q_t,M_t,\Sigma_t), A_t) \rVert^2 + \sum_{t=1}^{T} \lVert \nabla_{\theta} \text{ log } h^\lambda(Q_t, M_t,\Sigma_t \mid z_{K,t}^\theta) \rVert^2 \Big] < \infty
    \end{align}
    \item[(3)]
    Denote by $S$ the trajectory random variable
    \[S :=
        \big(
    X_0^{K,\lambda,\theta},
            \alpha_0^{K,\lambda,\theta},
            X_1^{K,\lambda,\theta},
            \ldots,
            \alpha_{T-1}^{K,\lambda,\theta},
            X_T^{K,\lambda,\theta},
            Z_{K,1}^{\lambda,\theta},
            \ldots,
            Z_{K,T}^{\lambda,\theta}
        \big),\]
    and define its law by $\P^{\lambda,\theta}:=\Lc(S)$.
    We assume that $\Rc(S) \in L^2(\P^{\lambda,\theta})$.
\end{enumerate}
\end{Assumption}

\begin{proof}[Proof of Theorem~\ref{theorem:continuous-pg-representation}]
Fix $\lambda\in(0,1)$. Write a trajectory as $s=(x_0,a_0,x_1,\ldots,a_{T-1},x_T,z_1,\ldots,z_T)$. For $t\in\Tc$, define the population-argument map $m_t : \widetilde \Zc_K^T \to \Pc_2(\R^d)$ by
\begin{align}
    m_0(z_1,\ldots,z_T) = \mu_0,\qquad
    m_t(z_1,\ldots,z_T) = \Dc_K(z_t),\quad
    t\in\llbracket1,T\rrbracket
\end{align}
We abbreviate these arguments as $m_t$. In particular, the initial population argument is kept exact and is not projected onto the Gaussian-mixture family.

Recalling the policy reference measure $\nu_A$ introduced in Section~\ref{sec:background}, define
\[
    \nu(\d s)
    :=
    \mu_0(\d x_0)
    \prod_{t=1}^T \d z_t
    \prod_{t=0}^{T-1}
    \left[
        \nu_A(\d a_t)
        P_t(\d x_{t+1}\mid x_t,m_t,a_t)
    \right].
\]
By independence of the randomizers, the trajectory law is
\begin{align}
    \P^{\lambda,\theta}(\d s)
    :=\Lc(S)(\d s)
    =\ell^{\lambda,\theta}(s)\nu(\d s),
\end{align}
where $  \ell^{\lambda,\theta}(s)
    =
    \prod_{t=0}^{T-1}
        p_t^\theta(a_t\mid x_t,m_t)
    \prod_{t=1}^T h^\lambda(z_t\mid z_{K,t}^\theta)$.
Here the return uses the same population arguments,
\[
    \Rc(s)=\sum_{t=0}^{T-1}r_t(x_t,m_t,a_t)+g(x_T,m_T),
\]
and is independent of $\theta$ when $s$ is held fixed. The reference measure $\nu$ is also independent of $\theta$. Hence, we have $ J_K^\lambda(\theta)
    =
    \int \Rc(s)\P^{\lambda,\theta}(\d s)
    =
    \int \Rc(s)\ell^{\lambda,\theta}(s)\nu(\d s)$. 
Now, under the Assumptions \ref{assumption: randomizer_rho}, \ref{assumption:continuous-likelihood-ratio-regularity}, we get with dominated convergence theorem that
\begin{align}
    \nabla_\theta J_K^\lambda(\theta)
    =
    \int
        \Rc(s)\nabla_\theta\ell^{\lambda,\theta}(s)
        \nu(\d s),
\end{align}
and we end using the \eqref{eq:continuous-main-score} term and $
    \nabla_\theta\log h^\lambda(z_t\mid z_{K,t}^\theta) =
    (D_t^\theta)^\top
    s^\lambda(z_t,z_{K,t}^\theta)$. The probabilistic representation follows by definition of the quantities.
\end{proof}

\subsubsection{Sensitivity estimation and statistical bounds}
\label{proof:continuous-bias-mse}

\begin{Assumption}[Statistical regularity]
    \label{assumption:continuous-statistical-regularity}
    Assume Assumptions~\ref{assumption:continuous-moment-bounds},~\ref{assumption:continuous-objective-regularity}, and~\ref{assumption:continuous-gradient-regularity}. Write $z^\theta = (z_{K,t}^\theta)_{t=1}^T \in \R^{T q_K}$ and $D^\theta = \nabla_\theta z^\theta \in \R^{Tq_K \times d_\theta}$. Assume additionally:
    \begin{enumerate}
        \item [(1)] \emph{Regular fitted coordinates.} There exists a compact set $\Cc_K \subset \widetilde \Zc_K$ containing all exact and fitted coordinates used below. The fitted coordinates use the same component ordering as the exact coordinates. The map $\theta \mapsto z^\theta$ is $\Cc^3$, with uniformly bounded third derivatives on an admissible neighborhood containing all shifted parameters $\theta \pm \eta e_\ell$.
        
        \item [(2)] \emph{Fitting accuracy.} Assumption~\ref{assumption:continuous-fitting-error} holds uniformly over the base parameters $u=\theta$ and shifted parameters $u=\theta\pm\eta e_\ell$, for $\theta\in\Theta$ and $0<\eta\leq\eta_0$. Writing $\widehat z^{N,u}=(\widehat z_{K,t}^{N,u})_{t=1}^T$, summing~\eqref{eq:continuous-fitting-error} over time gives $\sup_u\E[\|\widehat z^{N,u}-z^u\|^2]\leq Ca_N^2$. The base fit uses $N=M$ particles and each shifted fit uses $N=n_0$.
        
        \item [(3)] \emph{Score identities and moments.} The likelihood-ratio identities hold for every prescribed coordinate flow \(c\in\mathcal C_K^T\), both for differentiation in \(\theta\) with \(c\) fixed and for differentiation in \(c\), with no additional boundary term. For one rollout using population arguments $m_0 = \mu_0$ and $m_t = \Dc_K(\Psi_\lambda(c_t,R_t))$, denote its return and policy score by $\Rc^\lambda(c)$ and $\Sc_\mathrm{pol}^\lambda(c)$. Uniformly in $\theta$, $c$ and $0 < \lambda \leq \lambda_0$,
        \begin{align}
            \E\big[ |\Rc^\lambda(c)|^2 \|\Sc_\mathrm{pol}^\lambda(c)\|^2\big] \leq C,
        \end{align}
        and
        \begin{align}
            \E\Big[|\Rc^\lambda (c)|^2 \sum_{t=1}^T \|s^\lambda(\Psi_\lambda(c_t, R_t), c_t)\|^2 \Big] \leq \frac{C}{\lambda^2}.
        \end{align}
        \item [(4)] \emph{Online simulations.} Conditional on the offline estimates, the \(B\) online trajectories are independent and identically distributed, using fresh simulation randomness.
    \end{enumerate}
\end{Assumption}

\begin{proof}[Proof of Theorem~\ref{theorem:continuous-bias-mse}]
    Fix $\theta$ and set $n_0 \coloneqq n/(2d_\theta)$. For each offline system of $N$ particles simulated under parameter $u$, denote its states by $(X_t^{i,N,u})_{i=1}^N$ and define
    \begin{align}
        \widehat\mu_t^{N,u}\coloneqq\frac1N\sum_{i=1}^N\delta_{X_t^{i,N,u}},\qquad
        \widehat z^{N,u}\coloneqq\big(\widehat\Ec_{K,N}(\widehat\mu_t^{N,u})\big)_{t=1}^T\in\Cc_K^T.
    \end{align}
    The base system uses $(N,u)=(M,\theta)$, and the shifted systems use $(N,u)=(n_0,\theta\pm\eta e_\ell)$, where $e_\ell$ is the $\ell$-th canonical basis vector of $\R^{d_\theta}$. Set
    \begin{align}
        \widehat z \coloneqq\widehat z^{M,\theta},\qquad
        \widehat D \coloneqq\begin{pmatrix}\widehat D_1\\\vdots\\\widehat D_T\end{pmatrix}\in\R^{Tq_K\times d_\theta},\qquad
        \widehat D e_\ell\coloneqq\frac{\widehat z^{n_0,\theta+\eta e_\ell}-\widehat z^{n_0,\theta-\eta e_\ell}}{2\eta},\quad \ell\in\llbracket1,d_\theta\rrbracket,
    \end{align}
    where $\widehat D_t\in\R^{q_K\times d_\theta}$ is the time-$t$ block, as in~\eqref{eq:continuous-algorithm-sensitivity}. Write $\widehat G\coloneqq\widehat G_{B,M,n}^{\lambda,\eta}(\theta)$. Throughout, the time-zero population argument remains fixed at $\mu_0$.

    \emph{Conditional mean and its stability.} For a prescribed coordinate flow $c$, define the randomized flow by
    \begin{align}
        Z^\lambda : \widetilde \Zc_K^T \times \Omega \longrightarrow \widetilde \Zc_K^T,\qquad
        Z^\lambda(c,\omega) = \big(\Psi_\lambda(c_t, R_t(\omega))\big)_{t=1}^T.
    \end{align}
    We omit $\omega$ for simplicity. Using the prescribed-flow objective $F_K$, define $\Gamma^\lambda : \Cc_K^T \times \R^{Tq_K \times d_\theta} \to \R^{d_\theta}$ by
    \begin{align}
        \label{eq:continuous-biasmse-gamma}
        \Gamma^\lambda(c,D) = \E\big[\nabla_\theta F_K(\theta,Z^\lambda(c))\big] + D^\top \E\big[(\nabla_c Z^\lambda(c))^\top \nabla_z F_K(\theta,Z^\lambda(c))\big].
    \end{align}
    Here \(\nabla_c Z^\lambda(c)\) is the Jacobian with respect to the prescribed coordinates.

    The policy-score identity identifies the first term in~\eqref{eq:continuous-biasmse-gamma}. For the population score, differentiation with respect to $c$, followed by the chain rule, gives
    \begin{align}
\E\left[
\mathcal R^\lambda(c)
\begin{pmatrix}
s^\lambda(\Psi_\lambda(c_1,R_1),c_1)\\
\vdots\\
s^\lambda(\Psi_\lambda(c_T,R_T),c_T)
\end{pmatrix}
\right]&=
\nabla_c\mathbb E\!\left[F_K(\theta,Z^\lambda(c))\right]\\
&=
\mathbb E\!\left[
(\nabla_c Z^\lambda(c))^\top
\nabla_zF_K(\theta,Z^\lambda(c))
\right].
\end{align}
    Consequently, $\Gamma^\lambda(c,D)$ is the expectation of the score estimator using coordinates $c$ and sensitivity matrix $D$.

    The same covariance calculation as in the gradient convergence proof gives, uniformly for $c\in\Cc_K^T$,
    \begin{align}
        \label{eq:continuous-biasmse-bound}
        \|Z^\lambda(c) - c\| + \|\nabla_cZ^\lambda(c) - I\| \leq C\lambda W,
    \end{align}
    where
    \begin{align}
        W = 1 + \sum_{t=1}^T \sum_{j=1}^K \big(\|A_{j,t}\| + \|B_{j,t}\|_\mathrm{F}^2\big),\qquad
        \E\big[W^{r+1}\big] < \infty.
    \end{align}
    The constant $C$ in~\eqref{eq:continuous-biasmse-bound} depends on $c$ only through an upper bound on $\|c\|$ and a lower bound on the covariance eigenvalues of $c$, and both are uniform over the compact set $\Cc_K$. Indeed, $c\mapsto\|c\|$ is continuous, hence bounded on $\Cc_K$. Moreover, $\lambda_{\min}(\Sigma)=\min_{\|u\|=1}u^\top\Sigma u$ is $1$-Lipschitz for $\|\cdot\|_\mathrm{op}$, so $c\mapsto\min_j\lambda_{\min}(\Sigma_j)$ is continuous and positive on $\Cc_K\subset\widetilde\Zc_K$. It therefore attains a positive minimum $c_K$ on $\Cc_K$.

    Since $\|c\|\leq C$ on $\Cc_K^T$, $\lambda\leq1$ and $W\geq1$, the bound~\eqref{eq:continuous-biasmse-bound} gives $\|Z^\lambda(c)\|\leq\|c\|+C\lambda W\leq CW$ and $\|\nabla_cZ^\lambda(c)\|\leq\|I\|+C\lambda W\leq CW$. The growth bound in Assumption~\ref{assumption:continuous-gradient-regularity}(2) gives $\|\nabla_zF_K(\theta,Z^\lambda(c))\|\leq L(1+\|Z^\lambda(c)\|)^r\leq CW^r$. Hence
    \begin{align}
        \big\|\E\big[(\nabla_c Z^\lambda(c))^\top \nabla_z F_K(\theta,Z^\lambda(c)) \big] \big\|
        \leq \E\big[\|\nabla_c Z^\lambda(c)\|\,\|\nabla_z F_K(\theta,Z^\lambda(c))\|\big]
        \leq C\,\E\big[W^{r+1}\big]\leq C.
    \end{align}
    Since $\Gamma^\lambda$ is linear in $D$,
    \begin{align}
        \label{eq:continuous-biasmse-gammalambda-bound}
        \big\|\Gamma^\lambda(c,D) - \Gamma^\lambda(c,D^\theta)\big\| \leq C\|D-D^\theta\|_\mathrm{F}.
    \end{align}
    At $\lambda = 0$,
    \begin{align}
        \Gamma^0(c,D^\theta) = \nabla_\theta F_K(\theta,c) + (D^\theta)^\top \nabla_z F_K(\theta,c).
    \end{align}
    Since $Z^0(c)=c$ and $\nabla_cZ^0(c)=I$, subtracting and adding $(D^\theta)^\top\E[\nabla_zF_K(\theta,Z^\lambda(c))]$ gives
    \begin{align}
        \Gamma^\lambda(c,D^\theta) - \Gamma^0(c,D^\theta)
        &= \E\big[\nabla_\theta F_K(\theta,Z^\lambda(c)) - \nabla_\theta F_K(\theta,c)\big]
        + (D^\theta)^\top\E\big[\nabla_z F_K(\theta,Z^\lambda(c)) - \nabla_z F_K(\theta,c)\big]\\
        &\quad + (D^\theta)^\top\E\big[(\nabla_cZ^\lambda(c) - I)^\top\nabla_z F_K(\theta,Z^\lambda(c))\big].
    \end{align}
    By the Lipschitz bound in Assumption~\ref{assumption:continuous-gradient-regularity}(2) and~\eqref{eq:continuous-biasmse-bound}, $\|\nabla F_K(\theta,Z^\lambda(c))-\nabla F_K(\theta,c)\|\leq L(1+\|Z^\lambda(c)\|+\|c\|)^r\|Z^\lambda(c)-c\|\leq C\lambda W^{r+1}$. The last term is bounded by $\|D^\theta\|\,C\lambda W\cdot CW^r$ using~\eqref{eq:continuous-biasmse-bound} and the growth bound. Since $D^\theta$ is uniformly bounded and $\E[W^{r+1}]<\infty$,
    \begin{align}
        \label{eq:continuous-biasmse-gamma-bound}
        \|\Gamma^\lambda(c,D^\theta) - \Gamma^0(c,D^\theta)\|\leq C\lambda.
    \end{align}
    Moreover, the Lipschitz derivative bounds on the compact set of base coordinates give
    \begin{align}
        \label{eq:continuous-biasmse-gammazero-bound}
        \big\|\Gamma^0(c,D^\theta) - \Gamma^0(z^\theta,D^\theta)\big\| \leq C\|c-z^\theta\|.
    \end{align}
    Since $\Gamma^0(z^\theta,D^\theta) = \nabla_\theta J_K(\theta)$, combining~\eqref{eq:continuous-biasmse-gammalambda-bound},~\eqref{eq:continuous-biasmse-gamma-bound}, and~\eqref{eq:continuous-biasmse-gammazero-bound} yields
    \begin{align}
        \label{eq:continuous-biasmse-gamma-gradient-bound}
        \big\|\Gamma^\lambda(c,D) - \nabla_\theta J_K(\theta)\big\| \leq C\big(\lambda + \|c-z^\theta\| + \|D - D^\theta\|_\mathrm{F}\big).
    \end{align}
    
    \emph{Population-coordinate error.}
    The fitting-accuracy assumption directly gives
    \begin{align}
        \label{eq:continuous-biasmse-z-bound}
        \E\big[\|\widehat z - z^\theta\|^2\big] \leq C a_M^2.
    \end{align}

    \emph{Sensitivity error.}
    Define the exact centered difference $D^\eta \in \R^{Tq_K \times d_\theta}$ column-wise by
    \begin{align}
        D^\eta e_\ell  = \frac{z^{\theta + \eta e_\ell} - z^{\theta-\eta e_\ell}}{2\eta}.
    \end{align}
    Taylor's formula along the $\ell$-th coordinate gives, with Taylor remainder $R_\ell^\pm$,
      \begin{align}
        z^{\theta\pm\eta e_\ell}
        &=z^\theta\pm\eta\,\partial_\ell z^\theta+\frac{\eta^2}{2}\partial_\ell^2z^\theta+R_\ell^\pm,\\
        \|R_\ell^\pm\|
        &\leq\frac{\eta^3}{6}\sup_{|u|\leq\eta}\|\partial_\ell^3z^{\theta+u e_\ell}\|\leq C\eta^3,
    \end{align}
    by Assumption~\ref{assumption:continuous-statistical-regularity}(1). Subtracting the two expansions cancels the constant and quadratic terms. Since $D^\theta e_\ell=\partial_\ell z^\theta$,
    \begin{align}
        \|(D^\eta-D^\theta)e_\ell\|
        =\bigg\|\frac{R_\ell^+-R_\ell^-}{2\eta}\bigg\|\leq C\eta^2.
    \end{align}
    Squaring and summing over the $d_\theta$ columns yields
    \begin{align}
        \label{eq:continuous-biasmse-D-bound}
        \|D^\eta - D^\theta\|_\mathrm{F} \leq C\eta^2.
    \end{align}
    For each shifted fit, set $\epsilon_\ell^\pm\coloneqq\widehat z^{n_0,\theta\pm\eta e_\ell}-z^{\theta\pm\eta e_\ell}\in\R^{Tq_K}$. Subtracting the fitted and exact centered differences gives
    \begin{align}
        (\widehat D-D^\eta)e_\ell
        &=\frac{(\widehat z^{n_0,\theta+\eta e_\ell}-z^{\theta+\eta e_\ell})-(\widehat z^{n_0,\theta-\eta e_\ell}-z^{\theta-\eta e_\ell})}{2\eta}\\
        &=\frac{\epsilon_\ell^+-\epsilon_\ell^-}{2\eta}.
    \end{align}
    The squared Frobenius norm is the sum of the squared column norms. Using $\|a-b\|^2\leq2\|a\|^2+2\|b\|^2$ and the fitting bound $\E[\|\epsilon_\ell^\pm\|^2]\leq Ca_{n_0}^2$ from Assumption~\ref{assumption:continuous-statistical-regularity}(2),
    \begin{align}
        \E\big[\|\widehat D - D^\eta\|_\mathrm{F}^2 \big] = \sum_{\ell=1}^{d_\theta} \E\bigg[\bigg\|\frac{\epsilon_\ell^+ - \epsilon_\ell^-}{2\eta} \bigg\|^2 \bigg] &\leq \frac{1}{2\eta^2} \sum_{\ell=1}^{d_\theta} \big(\E\big[\|\epsilon_\ell^+\|^2\big] + \E\big[\|\epsilon_\ell^-\|^2\big]\big)\\
        &\leq C\frac{a_{n_0}^2}{\eta^2}.
    \end{align}
    Decomposing $\widehat D-D^\theta=(\widehat D-D^\eta)+(D^\eta-D^\theta)$ and using $\|A+B\|_\mathrm{F}^2\leq2\|A\|_\mathrm{F}^2+2\|B\|_\mathrm{F}^2$, together with~\eqref{eq:continuous-biasmse-D-bound}, proves
    \begin{align}
        \label{eq:continuous-biasmse-Dexp-bound}
        \E\big[\|\widehat D -D^\theta\|_\mathrm{F}^2 \big] \leq C\bigg(\eta^4 + \frac{a_{n_0}^2}{\eta^2}\bigg).
    \end{align}
    Because $D^\theta$ is uniformly bounded and $a_{n_0} \leq \eta$,
    \begin{align}
        \E\big[\|\widehat D\|_\mathrm{F}^2\big] \leq C.
    \end{align}

    \emph{Bias and mean-square error.}
    Define the sigma-algebra generated by the offline particle systems by
    \begin{align}
    \Fc^\mathrm{off}\coloneqq\sigma\Big(
        \big(X_t^{i,M,\theta}\big)_{
            t\in\llbracket0,T\rrbracket,\,
            i\in\llbracket1,M\rrbracket},
        \big(X_t^{i,n_0,\theta+\varepsilon\eta e_\ell}\big)_{
            t\in\llbracket0,T\rrbracket,\,
            i\in\llbracket1,n_0\rrbracket,\,
            \ell\in\llbracket1,d_\theta\rrbracket,\,
            \varepsilon\in\{-1,1\}}
    \Big).
    \end{align}
    Since the fitted coordinates are functions of these particles, $\widehat z$ and $\widehat D$ are $\Fc^\mathrm{off}$-measurable. Set $\overline G\coloneqq\E[\widehat G\mid\Fc^\mathrm{off}]$. The online simulation randomness gives $\overline G=\Gamma^\lambda(\widehat z,\widehat D)$. 
    
    Since $\overline G=\Gamma^\lambda(\widehat z,\widehat D)$, the stability estimate~\eqref{eq:continuous-biasmse-gamma-gradient-bound} and $(a+b+c)^2\leq3(a^2+b^2+c^2)$ give
\begin{align}
    \E\big[\|\overline G-\nabla_\theta J_K(\theta)\|^2\big]
    &\leq C\,\E\big[
        \big(\lambda+\|\widehat z-z^\theta\|
        +\|\widehat D-D^\theta\|_\mathrm{F}\big)^2
    \big]\notag\\
    &\leq C\big(
        \lambda^2
        +\E[\|\widehat z-z^\theta\|^2]
        +\E[\|\widehat D-D^\theta\|_\mathrm{F}^2]
    \big)\notag\\
    &\leq C\bigg(
        \lambda^2+a_M^2+\eta^4
        +\frac{a_{n_0}^2}{\eta^2}
    \bigg),
    \label{eq:continuous-biasmse-overlineg-bound}
\end{align}
where the last inequality uses the population-coordinate bound~\eqref{eq:continuous-biasmse-z-bound} and the sensitivity bound~\eqref{eq:continuous-biasmse-Dexp-bound}.
    Since $\E[\widehat G] = \E[\overline G]$, Jensen’s inequality proves the bias bound in the Theorem.

    Conditional on $\Fc^\mathrm{off}$, the online estimator averages $B$ independent contributions. Applying the
same conditional-variance argument as in the paragraph
\emph{Bias and mean-square error of the online estimator} of the discrete state space bias and mean-square error proof in
Appendix~\ref{proof:discrete-bias-mse}, together with the score moment bounds of Assumption~\ref{assumption:continuous-statistical-regularity}(3), gives
    \begin{align}
        \label{eq:continuous-biasmse-widehatg-bound}
        \E\big[\|\widehat G-\overline G\|^2 \mid \Fc^\mathrm{off}\big] \leq \frac{C}{B}\big(1+\lambda^{-2}\|\widehat D\|_\mathrm{F}^2\big).
    \end{align}
    Write
$\widehat G-\nabla_\theta J_K(\theta)
=(\widehat G-\overline G)
+(\overline G-\nabla_\theta J_K(\theta))$.
Since $\overline G$ is $\Fc^\mathrm{off}$-measurable and
$\E[\widehat G-\overline G\mid\Fc^\mathrm{off}]=0$,
the cross term vanishes:
\begin{align}
    \E\big[
        \langle\widehat G-\overline G,\,
        \overline G-\nabla_\theta J_K(\theta)\rangle
    \big]
    &=\E\big[
        \langle
        \E[\widehat G-\overline G\mid\Fc^\mathrm{off}],\,
        \overline G-\nabla_\theta J_K(\theta)
        \rangle
    \big]\\
    &=0.
\end{align}
Expanding the squared norm and taking expectations therefore gives
\begin{align}
    \E\big[\|\widehat G-\nabla_\theta J_K(\theta)\|^2\big]
    &=
    \E\big[\|\widehat G-\overline G\|^2\big]
    +\E\big[\|\overline G-\nabla_\theta J_K(\theta)\|^2\big].
\end{align}
    Combining~\eqref{eq:continuous-biasmse-overlineg-bound} and~\eqref{eq:continuous-biasmse-widehatg-bound} proves the result.
\end{proof}

\section{Linear dynamics: A fully solvable framework}\label{eq : LQ-setting}

We present an application of the continuous-state setting developed in Section \ref{subsec:continuous} in the linear dynamics setting where the problem remains fully tractable.
We will assume that the underlying dynamics is linear in its parameters with Gaussian noise. The literature on this kind of problem is extensive, see \citet{yong2013linear, elliott2013discrete,carmona2019linear} among others, but its tractability makes it well-suited for benchmarking.
For simplicity and readability, we provide the computations in the case $d=1$ but it can easily be adapted to the multi-dimensional setting. Hence, we make the following assumptions
\begin{Assumption}\label{assumption : LQ_model}
For any $(t,x,m,a) \in \llbracket 0,T-1\rrbracket\times\R\times\Pc_2(\R)\times\R$, we assume  
    \begin{align}
    \begin{cases}
    P_t(x,m, a) &= \Nc( \beta_t + b_t x + \bar{b}_t \bar{m} + c_t a , d_t), \\
    X_0 &\sim \Nc(m_0,\Sigma_0), \\
    \pi_t^{\theta}(\cdot |x,m) &= \Nc(e_t^{\theta} x + \bar{e}_t^{\theta}\bar m ,f_t^{\theta})
    \end{cases}
\end{align}
for some $\R$-valued families $(\beta_t)_t,(b_t)_t, (\bar{b}_t)_t, (c_t)_t,(e_t^{\theta})_t, (\bar{e}_t^{\theta})$ and $(d_t)_t, (f_t^{\theta}) \in \R^{+}_{\star}$ and some $m_0 \in \R$ and $\Sigma_0 \in \R_{+}$ and where we denoted  $\bar{m} := \E_{X \sim m}[X]$. 
Moreover, we impose sufficient regularity conditions on the maps \(\theta \mapsto (e_t^{\theta},\bar e_t^{\theta},f_t^{\theta})\) to justify the score-function identity underlying the policy-gradient representation.
\end{Assumption}

\begin{Remark}
Even if we are assuming the underlying transition kernel dynamics, we are still relatively in a model-free setting as we do not have access to its parameters.
\end{Remark}

\begin{Lemma}\textnormal{(Stability of the process $(X_t)_{t \in \Tc}$).}

Let $t \in \llbracket 0, T \rrbracket$ and $\theta \in \Theta$.
Under Assumptions \ref{assumption : LQ_model}, we have 
\begin{align}
    X_{t}^{\theta} \sim  \Nc \big(m_t^{\theta}, \Sigma_t^{\theta} \big),
\end{align}
where the sequences $(m_t^{\theta})_{t \in \Tc}$ and  $(\Sigma_t^{\theta})_{t \in \Tc}$ are solution to the linear system
\begin{align}\label{eq : def_mean_variance_sequence}
\begin{cases}
    m_{t+1}^{\theta} &=  \beta_t + (b_t + \bar{b}_t) m_t^{\theta} + c_t (e_t^{\theta} +\bar{e}_t^{\theta}) m_t^{\theta}, \\ 
    \Sigma_{t+1}^{\theta} &= (b_t + c_t e_t^{\theta})^2 \Sigma_t^{\theta} + c_t^2 (f_t^{\theta}) + d_t 
\end{cases}
\end{align}
starting from $m_0^{\theta} = m_0$ and $\Sigma_0^{\theta} = \Sigma_0$.
\end{Lemma}
\begin{proof}
    The proof is straightforward recalling that $X_{t+1}|(X_t,\P_{X_t}, \alpha_t) \sim P_t(X_t,\P_{X_t},\alpha_t)$ $\P-\text{a.s}$. Hence, it implies that $\P-\text{a.s}$
    \begin{align}\label{eq : identity_LQ}
        X_{t+1} = \beta_t + b_t X_t + \bar{b}_t \E[X_t] + c_t ( e_t^{\theta} X_t + \bar{e}_t^{\theta} \E[X_t] + \sqrt{f_t^{\theta}} \epsilon_{t+1}) + \sqrt{d_t} \tilde{\epsilon}_{t+1},
    \end{align}
    where $\epsilon_{t+1}$ and $\tilde{\epsilon}_{t+1}$ are independent standard Gaussian random variables independent of $X_t$. The result follows taking the expectation and the variance in \eqref{eq : identity_LQ}.
\end{proof}
Hence, under Assumption \ref{assumption : LQ_model} and the specific Gaussian parametrization, it follows that the process $X^{\theta}$ remains in the manifold of Gaussian random variables, namely its law is characterized by a finite-dimensional vector, here the couple $(\E[X_t^{\theta}], \mathbb{V}[X_t^{\theta}])$ for any $t \in \Tc$.
\begin{Proposition}\textnormal{(Transport map representation).}
According to the transport map construction in \ref{prop:continuous-state-transport}, the following holds
\begin{align}
    M_{t}^{\lambda,\theta}(\omega) = \Nc \big(m_t^{\lambda,\theta}(\omega) , \Sigma_t^{\lambda,\theta}(\omega) \big), \quad \omega \in \Omega, \quad t \in \llbracket 1, T \rrbracket.
\end{align}
where we set
\begin{align}
\begin{cases}
    m_t^{\lambda,\theta}(\omega) &:=  (1-\lambda) m_t^{\theta} + \lambda A_t(\omega), \\
    \Sigma_t^{\lambda,\theta}(\omega) &:= \big( (1-\lambda)+ \lambda B_t(\omega) \big) ^2 \Sigma_t^{\theta}
\end{cases}
\end{align}
\end{Proposition}
\begin{proof}
    This is a direct application of Proposition \ref{prop:continuous-state-transport}.
\end{proof}
\begin{Lemma}\textnormal{(Law of the random variables $(m_t^{\lambda,\theta}, \Sigma_t^{\lambda,\theta}))_{t \in \llbracket 1, T \rrbracket}$).}\label{lemma : density_randomized}

Let $U_{\lambda} = \R \times  (-\frac{1-\lambda}{\lambda},+ \infty) $. For any $t \in \llbracket 1, T \rrbracket$, $\theta \in \Theta$ and $\lambda  \in (0,1)$, suppose that $(A_t,B_t) \in U_{\lambda} \quad \P-\text{a.s}$ and admits a  strictly positive $\mathcal{C}^1$ density map $f_{(A_t,B_t)}$ with respect to the Lebesgue measure over $U_{\lambda}$. Then  the random variable $\big(m_t^{\lambda,\theta}, \Sigma_t^{\lambda,\theta} \big)$ admits the following density
    \begin{align}
        h^\lambda((m,\Sigma)\mid z_{1,t}^\theta)= \frac{f_{(A_t,B_t)}\big(a_{t}^{\lambda,\theta}(m,\Sigma),b_t^{\lambda,\theta}(m,\Sigma) \big)}{2 \lambda^2 \sigma_t^{\theta} \sqrt{\Sigma}},
    \end{align}
where $z_{1,t}^\theta \coloneqq (m_t^\theta, \Sigma_t^\theta) = (m_t^\theta,(\sigma_t^\theta)^2)\in\R\times\R_+^\star$, and the maps $a_t^{\lambda,\theta}$ and $b_t^{\lambda,\theta}$ are defined as 
\begin{align}
\begin{cases}
    a_t^{\lambda,\theta}(m,\Sigma) &=\frac{ m -(1-\lambda) m_t^{\theta}}{\lambda}, \\
    b_t^{\lambda,\theta}(m,\Sigma) &= \frac{1}{\lambda} \big( \frac{\sqrt{\Sigma}}{\sigma_t^{\theta}}- (1-\lambda) \big)
\end{cases}
\end{align}
and where the quantities $(m_t^{\theta}, \sigma_t^{\theta})$ are given by  \eqref{eq : def_mean_variance_sequence}.
\end{Lemma}
\begin{proof}
  This is a direct application of Lemma \ref{lemma: Law_z_k} specified to our setting where $K=1$.
\end{proof}

\begin{Theorem}{(Policy gradient in the L-Q setting).}\label{thm : policy_gradient_LQ}

Under the assumptions of Theorem~\ref{theorem:continuous-pg-representation}, for any $\theta\in\Theta$, the map $J^{\lambda}$ defined in \eqref{eq:perturbed-objective} is differentiable with respect to $\theta$ with gradient map given by
\begin{align}\label{eq : policy_gradient_LQ_setting}
    \nabla_{\theta} J^{\lambda}(\theta) &=\E \Big[ \sum_{t=0}^{T-1} R_{t}^{\lambda,\theta} \nabla_{\theta} \text{ log }  \pi_{t}^{\theta}(X_{t}^{\lambda,\theta}, \Nc(m_{t}^{\lambda,\theta}, \Sigma_{t}^{\lambda,\theta}), a_{t}^{\lambda,\theta}) +  \sum_{t=1}^{T} R_t^{\lambda,\theta}\nabla_{\theta} \text{ log } h^\lambda((m_t^{\lambda,\theta},\Sigma_t^{\lambda,\theta})\mid z_{1,t}^\theta) \Big] 
\end{align}
where we denoted the random variables $(R_s^{\lambda,\theta})_{s \in \Tc}$ as 
\begin{align}
\begin{cases}
    R_t^{\lambda,\theta} &:= \sum_{s=t}^{T-1} r_s \big(X_s^{\lambda,\theta}, \Nc(m_s^{\lambda,\theta}, \Sigma_s^{\lambda,\theta}) ,  a_s^{\lambda,\theta} \big) + g(X_T^{\lambda, \theta}, \Nc(m_T^{\lambda, \theta}, \Sigma_T^{\lambda, \theta})), \quad t \in \llbracket 0, T-1 \rrbracket   , \\
    R_T^{\lambda,\theta} &:= g \big(X_T^{\lambda,\theta}, \Nc( m_{T}^{\lambda,\theta} , \Sigma_T^{\lambda,\theta}) \big)
\end{cases}
\end{align}
\end{Theorem}
\begin{proof}
    Apply Theorem~\ref{theorem:continuous-pg-representation} with $K=1$. Each policy score has conditional mean zero given the history up to its action draw, and each population score has conditional mean zero given the history before its independent randomizer is drawn. Rewards at times $s<t$ are measurable with respect to these histories, so their products with the time-$t$ scores have zero expectation. Removing these past rewards gives the stated reward-to-go formula.
\end{proof}
\begin{Corollary}{\textnormal{(Score function of the mean-field density).}}\label{corollary: score_mean_field_density}

For any $\lambda \in (0,1)$, $\theta \in \Theta$ and $t \in \llbracket 1, T \rrbracket$, we have
\begin{align}
    \nabla_{\theta} \text{ log } h^\lambda((m,\Sigma)\mid z_{1,t}^\theta) &= - \frac{1-\lambda}{\lambda} \partial_a \text{ log } f_{(A_{t}, B_{t})}(a_{t}^{\lambda,\theta}(m, \Sigma), b_t^{\lambda,\theta}(m, \Sigma)) \nabla_{\theta} m_t^{\theta} \\
    &\quad - {\frac{1}{\lambda}}\partial_b \text{ log } f_{(A_{t}, B_t)}(a_{t}^{\lambda,\theta}(m, \Sigma), b_{t}^{\lambda,\theta}(m, \Sigma)) \frac{\sqrt{\Sigma}}{\sigma_{t}^{\theta}} \nabla_{\theta} \text{ log } \sigma_{t}^{\theta}   \\
    &\quad - \nabla_{\theta} \text{ log }(\sigma_{t}^{\theta}) 
\end{align}
\end{Corollary}
\begin{proof}
    The proof is straightforward noticing that 
    \begin{align}
        \nabla_{\theta} (\frac{1}{\sigma_t^{\theta}}) &= - \frac{\nabla_{\theta} \sigma_t^{\theta}}{(\sigma_t^{\theta})^2} = - \frac{\nabla_{\theta} \text{ log}(\sigma_t^{\theta})}{\sigma_t^{\theta}}
    \end{align}
    and recalling \eqref{lemma : density_randomized}.
\end{proof}
From the computation of the score function of the mean-field density in Corollary \ref{corollary: score_mean_field_density} , it appears that we need to get estimators of the quantities $(\nabla_{\theta} m_t^{\theta})_{t \in \llbracket 1, T \rrbracket}$ and $(\nabla_{\theta}(\sigma_t^{\theta})^2)_{t \in \llbracket 1,T\rrbracket}$ similarly as we had to compute estimators of $(D_t^{\theta})_{t \in \Tc}$. Unlike the approach developed in the main body, we can look for a probabilistic approximation of theses quantities in the spirit of the discrete-time setting.  Indeed, by noticing that 
\begin{align}
\begin{cases}
    m_t^{\theta} &=\E \big[ \sum_{s=0}^{T-1} \mathbf{1}_{t=s} X_s^{\theta}  + \mathbf{1}_{t=T} X_T^{\theta}], \\
    (\sigma_t^{\theta})^2 &= \E \Big[ \sum_{s=0}^{T-1} \mathbf{1}_{s=t} (X_s^{\theta})^2 + \mathbf{1}_{t=T} (X_T^{\theta})^2 \Big] - (m_t^{\theta})^2,
\end{cases}
\end{align}
We hence are able to provide the following probabilistic approximations of the required gradient quantities.

\begin{align}
\begin{cases}
    \nabla_{\theta} m_t^{\theta} &\approx   \E \Big[ X_t^{\lambda,\theta} \Big( \sum_{\tau=0}^{t-1}  \nabla_{\theta} \text{ log } \pi_{\tau}^{\theta}(X_{\tau}^{\lambda,\theta}, \Nc(M_{\tau}, \Sigma_{\tau}), \alpha_{\tau}^{\lambda,\theta}) +  \sum_{\tau=1}^{t-1} \nabla_{\theta} \text{ log } h^\lambda((M_\tau,\Sigma_\tau)\mid z_{1,\tau}^\theta) \Big) \\
    &\qquad  +\mathbf{1}_{t =T} X_T^{\lambda,\theta} \nabla_{\theta} \text{ log } h^\lambda((M_T,\Sigma_T)\mid z_{1,T}^\theta) \Big], \\
    \nabla_{\theta} (\sigma_t^{\theta})^2   
    &\approx \E \Big[ (X_t^{\lambda,\theta})^2 \Big(\sum_{\tau=0}^{t-1} \nabla_{\theta} \text{ log } \pi_{\tau}^{\theta}(X_{\tau}^{\lambda,\theta}, \Nc(M_{\tau}, \Sigma_{\tau}), \alpha_{\tau}^{\lambda,\theta}) + \sum_{\tau=1}^{t-1} \nabla_{\theta} \text{ log } h^\lambda((M_\tau,\Sigma_\tau)\mid z_{1,\tau}^\theta) \Big)
    \\
    &\qquad + \mathbf{1}_{t =T} (X_T^{\lambda,\theta})^2 \nabla_{\theta} \text{ log } h^\lambda((M_T,\Sigma_T)\mid z_{1,T}^\theta) \Big]- 2 \nabla_{\theta} m_t^{\theta} m_t^{\theta}.
\end{cases}
\end{align}
The approximations follow by similar arguments as in \cite[Section 3.1] {meunier2026model} so we refer to the references therein for full details. Hence, it follows by noticing that
$\nabla_{\theta} \text{ log} (\sigma_t^{\theta}) = \frac{\nabla_{\theta}(\sigma_t^{\theta})^2}{2 (\sigma_t^{\theta})^2}$.

The probabilistic representations of the moment sensitivities $\nabla_\theta m_t^\theta$ and $\nabla_\theta \sigma_t^\theta$ provide an alternative model-free approach to evaluate the mean-field score. By mirroring the likelihood-ratio technique of the finite-state setting, this formulation bypasses the need for the centered policy differences used in the general continuous-state algorithm. We denote this variant as Transport-Proba. Following the structure of Algorithm~\ref{alg:transport-reinforce}, we compare Transport-Proba directly against both standard REINFORCE and the default Transport estimator. The results are presented in Appendix~\ref{experiments:transport-proba-comparison}.

\section{Algorithmic details}
\label{app:algorithm-complexity}

\subsection{Sampling details}

Each gradient step at the current $\theta$ has an offline stage, which estimates the population coordinates and their sensitivities, and an online stage, which simulates the perturbed trajectories entering the estimator~\eqref{eq:discrete-main-estimator} or~\eqref{eq:continuous-main-estimator}. All stages use the same rollout: given population arguments $(m_t)_{t\in\Tc}$,
\begin{align}
    &X_0\sim\mu_0,\qquad
    \alpha_t\sim\pi_t^\theta(\cdot\mid X_t,m_t),\qquad
    X_{t+1}\sim P_t(\cdot\mid X_t,m_t,\alpha_t),\qquad
    t\in\llbracket0,T-1\rrbracket,
    \label{eq:algorithm-rollout}
\end{align}
with policy scores $\nabla_\theta\log p_t^\theta(\alpha_t\mid X_t,m_t)$. The procedure is summarized in Algorithm~\ref{alg:transport-reinforce}.

\paragraph{Offline stage: population coordinates.}
Run $M$ interacting particles $(X_t^\ell)_{\ell=1}^M$ through~\eqref{eq:algorithm-rollout}, with $X_0^\ell$ i.i.d.\ with law $\mu_0$ and population arguments given by their empirical law $m_t=\widehat\mu_t^M\coloneqq\frac1M\sum_{\ell=1}^M\delta_{X_t^\ell}$. The estimated coordinates are
\begin{align}
    \widehat z_t\coloneqq
    \begin{cases}
        \big(\widehat\mu_t^M(1),\ldots,\widehat\mu_t^M(N-1)\big),&\Xc\text{ finite},\\
        \Ec_K(\widehat\mu_t^M),&\Xc=\R^d,
    \end{cases}
    \label{eq:algorithm-population-particles}
\end{align}
where, in continuous state space, $\Ec_K$ is computed by constrained Expectation-Maximization on the particles.

\paragraph{Offline stage: sensitivities in finite state space.}

Set $\widehat D_0^{\eta,M,n}\coloneqq0$. For $t=1,\ldots,T$, run a batch of $n$ rollouts~\eqref{eq:algorithm-rollout} of length $t$, indexed by $r\in\llbracket1,n\rrbracket$, with population arguments
\begin{align}
    &\overline m_0^{(t,r)}=\mu_0,\quad
    \overline m_s^{(t,r)}=(1-\eta)\widehat\mu_s^M+\eta\,\overline Q_s^{(t,r)},\\
    &\overline Q_s^{(t,r)}=\varphi\big(\overline U_s^{(t,r)}\big),\quad
    s\in\llbracket1,t-1\rrbracket,
\end{align}
where the $\overline U_s^{(t,r)}$ are i.i.d.\ with density $\rho$. Denote the states and actions of these rollouts by $\overline X_s^{(t,r)}$ and $\overline\alpha_s^{(t,r)}$, respectively, and define their policy scores by
\begin{align}
    \overline\Psi_s^{(t,r)}\coloneqq\nabla_\theta\log p_s^\theta\big(\overline\alpha_s^{(t,r)}\mid\overline X_s^{(t,r)},\overline m_s^{(t,r)}\big)\in\R^{d_\theta},\qquad s\in\llbracket0,t-1\rrbracket,
    \label{eq:algorithm-auxiliary-policy-score}
\end{align}
where the derivative holds the action, state, and population argument fixed. The $k$-th row of $\widehat D_t \coloneqq \widehat D_t^{\eta,M,n} \in\R^{(N-1)\times d_\theta}$, written $\widehat D_t^{\eta,M,n}$ in~\eqref{eq:discrete-main-estimator} for $k\in\llbracket1,N-1\rrbracket$, is
\begin{align}
    (\widehat D_t)_{k,:}&\coloneqq\frac1n\sum_{r=1}^n\mathbf 1_{\{\overline X_t^{(t,r)}=x_k\}}
    \bigg[\sum_{s=0}^{t-1}\overline\Psi_s^{(t,r)}
    -\frac{1-\eta}{\eta}\sum_{s=1}^{t-1}\widehat D_s^\top H\big(\overline Q_s^{(t,r)}\big)\bigg]^\top.
    \label{eq:discrete-sensitivity-estimator}
\end{align}
Each row is the score-function estimator of Theorem~\ref{theorem:discrete-pg-representation}, at scale $\eta$ and horizon $t$, for the terminal reward $\mathbf 1_{\{x=x_k\}}$, with $\widehat\mu^M$ and the earlier estimates $\widehat D_1,\ldots,\widehat D_{t-1}$ in place of $\mu^\theta$ and $D_1^\theta,\ldots,D_{t-1}^\theta$. It  estimates the sensitivity of the $\eta$-perturbed state law.

\paragraph{Offline stage: sensitivities in continuous state space.}
For each $\ell\in\llbracket1,d_\theta\rrbracket$, run two independent systems of $n_0$ interacting particles as above, at the shifted parameters $\theta\pm\eta e_\ell$, where $e_\ell$ is the $\ell$-th canonical basis vector of $\R^{d_\theta}$ and $\theta\pm\eta e_\ell\in\Theta$. Denote by $\widehat z_t^{\theta\pm\eta e_\ell}$ their fitted coordinates, which stack the weights, means and independent covariance entries. Set
\begin{align}
    \widehat D_t e_\ell\coloneqq\frac{\widehat z_t^{\theta+\eta e_\ell}-\widehat z_t^{\theta-\eta e_\ell}}{2\eta},\ t\in\llbracket1,T\rrbracket,\ \ell\in\llbracket1,d_\theta\rrbracket.
    \label{eq:continuous-algorithm-sensitivity}
\end{align}
These centered differences use $n=2d_\theta n_0$ particles in total.

\paragraph{Online stage.}
Run $B$ rollouts~\eqref{eq:algorithm-rollout} with population arguments $\widehat M_0^{(b)}=\mu_0$ and, for $t\in\llbracket1,T\rrbracket$,
\begin{align}
    \widehat M_t^{(b)}=\begin{cases}
        (1-\lambda)\widehat\mu_t^M+\lambda\varphi(U_t^{(b)}),&\Xc\text{ finite},\\
        \Dc_K\big(\Psi_\lambda(\widehat z_t,R_t^{(b)})\big),&\Xc=\R^d,
    \end{cases}
    \label{eq:algorithm-online-sampling}
\end{align}
where the randomizers $U_t^{(b)}$ or $R_t^{(b)}$ are i.i.d.\ with law $\rho$. Their returns $\widehat\Rc^{(b)}$ and policy scores $\widehat\Sc_{\mathrm{pol}}^{(b)}$, together with $\widehat D$, give the estimator~\eqref{eq:discrete-main-estimator} or~\eqref{eq:continuous-main-estimator}, and $\theta$ is updated by $\theta\leftarrow\theta+\gamma\,\widehat G_{B,M,n}^{\lambda,\eta}(\theta)$.

\subsection{Computational Cost}
The finite-state score estimator uses the representation of population probabilities as expectations of state indicators. The continuous-state centered-difference estimator avoids differentiating the nonlinear mixture-fitting map.

\begin{Proposition}[Simulation cost]
    \label{prop:transport-simulation-cost}
    One policy update uses
    \begin{align}
        C_{\mathrm{sim}}=
        \begin{cases}
            T(M+B)+nT(T+1)/2,
            &\text{finite states, fresh auxiliary batches},\\
            T(M+n+B),
            &\text{finite states, shared auxiliary batch},\\
            T(M+n+B),
            &\text{continuous states}.
        \end{cases}
        \label{eq:simulator-complexity}
    \end{align}
    These counts exclude the arithmetic cost of mixture fitting and score evaluation. The finite-state MSE theorem covers fresh auxiliary batches; the experiments use a shared batch.
\end{Proposition}

\begin{proof}
    The population and online stages use $MT$ and $BT$ transitions, respectively. Fresh finite-state auxiliary batches require $n\sum_{t=1}^T t=nT(T+1)/2$ transitions, whereas a shared batch of $n$ trajectories of length $T$ requires $nT$ transitions.
    In continuous state space, the $2d_\theta$ auxiliary systems contain $n=2d_\theta n_0$ particles in total and require $nT$ transitions.
\end{proof}

\paragraph{Simulator calls.}
In continuous state space, for a fixed total $n$, increasing $d_\theta$ decreases the number $n_0=n/(2d_\theta)$ of particles per shifted system. In finite state space, drawing a fresh auxiliary batch of length $t$ for each $t$, the variant covered by Theorem~\ref{theorem:discrete-bias-mse}, raises the auxiliary cost from $nT$ to $nT(T+1)/2$ transitions. For reference, MF-REINFORCE \citep{meunier2026model}, with two auxiliary trajectories resampled at each terminal time, uses $C_{\mathrm{MF}}=B_{\mathrm{MF}}T+n_{\mathrm{MF}}T(T+1)$ transitions.

\paragraph{Arithmetic and storage.}
Let $q=N-1$ in finite state space and $q=q_K$ in continuous state space. Beyond simulation, policy evaluation and mixture fitting, one update costs $O(BTqd_\theta)$ operations for the main scores, plus $O(nT^2qd_\theta)$ for the finite-state auxiliary recursion or $O(Tqd_\theta)$ for the centered differences, and $O(Tq(1+d_\theta))$ memory for $(\widehat z_t,\widehat D_t)_{t=1}^T$. The $2d_\theta$ shifted particle systems are independent and are simulated as a single batch, whereas the finite-state sensitivities are computed sequentially in $t$.

\section{Experiment details}
\label{app:experiments}

\subsection{Protocol and settings}

All configurations use five paired seeds and Adam \citep{kingma2014adam}. Within each benchmark, methods share the policy class, the initialization, the training initial laws and the validation schedule. On the finite benchmarks, the initial law is resampled during training and validation uses a fixed initial law, as in \citet{meunier2026model}; the continuous benchmarks use a fixed initial law throughout. Every ten updates, the frozen policy is evaluated with $M_{\mathrm{val}}=10^5$ interacting particles.

\paragraph{Evaluation of the optimality gap.}
\label{app:evaluation}
For seed $s$, let $\widehat J_s:\Theta\to\R$ denote the validation estimate of the objective. The learned parameter $\widehat\theta_s$ and the reference parameter $\theta^\star$ are evaluated from the same initial distribution, with the same validation seed and particle budget. The reported gap for that seed is
\[
    \big|\widehat J_s(\widehat\theta_s)-\widehat J_s(\theta^\star)\big|.
\]
We report the mean and standard deviation of these five gaps. Table~\ref{tab:objective-summary} reports $\widehat J_s(\widehat\theta_s)$, while the $J(\theta^\star)$ column of Table~\ref{tab:main-summary} lists the separately computed reference value. The gap calculation uses the sampled reference $\widehat J_s(\theta^\star)$.

\paragraph{Reference policies.}
Two-state control uses the closed-form policy in Appendix~\ref{app:benchmark-two-state}; advertising uses the optimal feedback of \citet{motte2021mathematical}. The linear--quadratic reference follows from the Riccati equations, and the portfolio reference combines the optimal feedback gain with a concave quadratic program for the intercepts. Bimodal allocation solves the scalar equation $J(\theta)=0$ by bisection. Distribution planning uses gradient ascent on the exact population recursion to obtain a reference value of $-0.056991$; global optimality is not certified. Cybersecurity is evaluated by its objective because no optimum is available.

For Table~\ref{tab:main-summary}, Transport's selected scales are $\lambda_\star/4$ on two-state control, linear--quadratic control and the portfolio; $\lambda_\star/2$ on distribution planning; and $\lambda_\star/8$ on bimodal allocation. Transport-Proba selects $\lambda_\star/2$ on both benchmarks.

\paragraph{Budgets.}
The implemented Transport update uses $C_\mathrm{sim}=T(M+n+B)$ simulated transitions (Proposition~\ref{prop:transport-simulation-cost}), using a shared auxiliary batch in finite state space, and every method within a benchmark receives the same simulation budget. In finite state space, the experiments estimate $\widehat D_1,\ldots,\widehat D_T$ from a single auxiliary batch of $n$ rollouts of length $T$; Appendix~\ref{app:discrete-eta} compares it with fresh batches of length $t$ for each $t$, the variant covered by Theorem~\ref{theorem:discrete-bias-mse}, which costs $nT(T+1)/2$ auxiliary transitions. REINFORCE simulates $M+n+B$ trajectories. MF-REINFORCE uses $T(M+B_\mathrm{MF})+n_\mathrm{MF}T(T+1)$ transitions and reads the population flow off the same $M$ particles as Transport.

The split between $n$ and $B$ is selected by the mean-square error of the gradient estimator against the exact gradient, over $120$ replications. On the continuous benchmarks, the auxiliary stage gives $128$ particles to each of the $2d_\theta$ shifted systems, so $n=256\,d_\theta$. The bimodal benchmark gives $4000$ particles to each of its two systems. Table~\ref{tab:allocations} collects the allocations.

All training was performed on the Runpod GPU cloud, using an RTX A4000. The code was run in parallel on the 16 available vCPUs, with at most 14 workers at a time.

\begin{table}[H]
\centering
\small
\caption{Transport allocations and optimization settings. $M$, $n$ and $B$ are the population, auxiliary and main sample sizes, and the budget is the resulting number of simulated transitions per update, $T(M+n+B)$, which every method of a benchmark receives. Transport is trained at $\lambda\in\lambda_\star\{1/8,1/4,1/2,1,2\}$, with $\lambda_\star=B^{-1/4}$; $\sigma$ is the finite-state simplex randomizer scale and $K$ the number of mixture components.}
\label{tab:allocations}
\begin{tabular}{lrrrrrrrrrrr}
\toprule
Benchmark & $T$ & $M$ & $n$ & $B$ & Budget & Updates & Step size & $\lambda_\star$ & $\eta$ & $\sigma$ & $K$ \\
\midrule
Two-state & 5 & 200 & 12 & 248 & 2\,300 & 20\,000 & $10^{-3}$ & $0.252$ & $0.98$ & $0.75$ & --- \\
Cybersecurity & 3 & 200 & 51 & 153 & 1\,212 & 20\,000 & $10^{-3}$ & $0.284$ & $0.98$ & $0.75$ & --- \\
Distribution & 5 & 500 & 280 & 280 & 5\,300 & 30\,000 & $10^{-4}$ & $0.244$ & $0.98$ & $0.5$ & --- \\
Advertising & 5 & 200 & 65 & 195 & 2\,300 & 10\,000 & $10^{-3}$ & $0.268$ & $0.95$ & $1$ & --- \\
Linear--quadratic & 20 & 150 & 10\,240 & 111 & 210\,020 & 10\,000 & $10^{-3}$ & $0.308$ & $0.95$ & --- & 1 \\
Portfolio & 10 & 100 & 5\,120 & 211 & 54\,310 & 50\,000 & $10^{-2}$ & $0.262$ & $0.95$ & --- & 1 \\
Bimodal allocation & 1 & 2\,000 & 8\,000 & 100\,000 & 110\,000 & 1\,000 & $10^{-3}$ & $0.0562$ & $0.04$ & --- & 1,2 \\
\bottomrule
\end{tabular}
\end{table}

\paragraph{Perturbation scales.}
The main scale is taken in $\lambda_\star\{1/8,1/4,1/2,1,2\}$, with $\lambda_\star=B^{-1/4}$ the order suggested by Theorem~\ref{theorem:discrete-bias-mse}.

On the finite benchmarks, $Q_t=\varphi(U_t)$ with $U_t\sim\Nc(0,\sigma^2I_{N-1})$. The auxiliary scale and the randomizer scale are selected jointly in $\eta\in\{0.85,0.95,0.98\}$ and $\sigma\in\{0.5,0.75,1\}$, before training and at $\lambda_\star$, by the same mean-square error criterion as the budget split, over three warm-started reference policies. This gives $(\eta,\sigma)=(0.98,0.75)$, $(0.98,0.75)$, $(0.98,0.5)$ and $(0.95,1)$ on two-state control, cybersecurity, distribution planning and targeted advertising. On two-state control, Transport is also trained at $\lambda_\star$ with $\eta \in \{\eta_\star/2, \eta_\star\}$, where $\eta_\star = n^{-1/4}$ suggested by Theorem~\ref{theorem:discrete-bias-mse}, and $2\eta_\star > 1$. MF-REINFORCE uses $\varepsilon=0.2,1,2$ on two-state control, cybersecurity and distribution planning, the best scales reported by \citet{meunier2026model}, and $\varepsilon=1$ on targeted advertising, the best screened scale over $\{0.5,1,2\}$ over three seeds at the matched budget.

The continuous benchmarks use $\eta=0.95$, selected against the exact sensitivities, except the bimodal benchmark, which uses $\eta=0.04$ (Appendix~\ref{app:benchmark-bimodal}). On the continuous benchmarks, the randomizer~\eqref{eq:continuous-perturbation-map} uses $A_{j,t}\sim\Nc(0,I_d)$, $B_{j,t}=e^{G_{j,t}}$ with $G_{j,t}\sim\Nc(0,1)$ in dimension one, and, for $K=2$, the same simplex randomizer for the weights. Linear--quadratic control and the portfolio use $K=1$; only the bimodal benchmark compares $K=1$ and $K=2$. Transport-Proba is trained on the same grid of $\lambda$ as Transport and has no auxiliary scale.

\subsection{Benchmarks}

\subsubsection{Two-state control}
\label{app:benchmark-two-state}

Following \citet{meunier2026model,gu2023dynamic}, states are $\{0,1\}$ and actions are stay (ST) or move (MV). A move from $x$ succeeds with probability $\lambda_x$, where $(\lambda_0,\lambda_1)=(0.5,0.8)$. Writing $m=\mu(1)$, the running and terminal rewards are $r(x,\mu)=g(x,\mu)=\mathbf1_{\{x=1\}}-m^2-10|m-0.4|$. The stationary policy uses two logits, $\pi^\theta(\mathrm{MV}\mid x)=(1+e^{-\theta_x})^{-1}$, initialized at zero. The horizon is $T=5$. Training samples $m_0\sim\Uc([0.1,0.9])$, and validation uses $m_0=0.8$. The optimal policy stays with probabilities $(0.2,0.25)$ in states $(0,1)$; it sends every initial law to $(0.6,0.4)$ in one step.

\subsubsection{Cybersecurity}
\label{app:benchmark-cybersecurity}

We use the model of \citet{meunier2026model,kolokoltsov2016mean,carmona2023model}, with states DI, DS, UI, US denoting defended/undefended and infected/susceptible computers. Action $a\in\{0,1\}$ switches protection at rate $0.8a$. Infection rates are
\begin{equation}
    \iota_D(\mu)=0.24+0.4\mu(\mathrm{DI})+0.4\mu(\mathrm{UI}),\qquad
    \iota_U(\mu)=0.18+0.3\mu(\mathrm{DI})+0.3\mu(\mathrm{UI}).
\end{equation}
Recovery rates are $0.5$ for defended and $0.4$ for undefended computers. These rates define a generator $Q^{\mu,a}$, and the transition kernel is $\exp(0.2Q^{\mu,a})$. The reward at time $t$ is $-0.2\,\gamma^t[0.3\mathbf1_{\{\mathrm{DI},\mathrm{DS}\}}(x)+0.5\mathbf1_{\{\mathrm{DI},\mathrm{UI}\}}(x)]$, with $\gamma=0.5^{0.2}$. The policy is a $\tanh$ MLP with two hidden layers of width $32$, taking $(t,\mu)$ as input. Training uses $T=3$ and $\mu_0\sim\operatorname{Dirichlet}(\mathbf1_4)$; validation uses the uniform law and $T=50$.

\subsubsection{Distribution planning}
\label{app:benchmark-distribution-planning}

Following \citet{meunier2026model,carmona2023model}, states lie on $\Z/10\Z$, actions are $a\in\{-1,0,1\}$, and $X_{t+1}=X_t+a\pmod{10}$. The rewards are $r(x,a,\mu)=-0.01|a|-\|\mu-\mu_{\mathrm{target}}\|^2$ and $g(x,\mu)=-\|\mu-\mu_{\mathrm{target}}\|^2$, with
\begin{equation}
    \mu_{\mathrm{target}}=(0.02,0.04,0.09,0.16,0.19,0.19,0.16,0.09,0.04,0.02).
\end{equation}
The policy is a $\tanh$ MLP with two hidden layers of width $64$, taking $(t,\mu)$ as input. The horizon is $T=5$. Training samples $\mu_0\sim\operatorname{Dirichlet}(\mathbf1_{10})$; validation starts from the uniform law.

\subsubsection{Targeted advertising}
\label{app:benchmark-advertising}

Adapting \citet{motte2021mathematical}, let $x\in\{0,1\}$ indicate customer status and $a\in\{0,1\}$ indicate advertising. With $p=\mu(1)$,
\begin{equation}
    P(1\mid x,a,\mu)=\min\{p+0.2a,1\},\qquad
    r_t(x,a,\mu)=0.5^t(x-0.15a),\qquad g=0.
\end{equation}
The policy is a $\tanh$ MLP with two hidden layers of width $32$, taking $(t,\mu)$ as input. The horizon is $T=5$. Training samples $p_0\sim\Uc([0.05,0.95])$; validation uses $p_0=0.5$.

\subsubsection{Linear--quadratic control}
\label{app:benchmark-lq}

For $\bar x(m)=\int x\,m(\mathrm dx)$, the policy and dynamics are
\begin{equation}
    A_t\mid X_t,m_t\sim\Nc(\theta_t^1X_t+\theta_t^2\bar x(m_t),\tau^2),\qquad
    X_{t+1}=aX_t+bA_t+c\bar x(m_t)+\sigma\varepsilon_{t+1},
\end{equation}
with independent standard Gaussian noise. The objective is $J=-\E\big[\sum_{t<T}\bigl(qX_t^2+rA_t^2+\gamma\bar x(\mu_t^\theta)^2\bigr)+q_TX_T^2+\gamma_T\bar x(\mu_T^\theta)^2\big]$, with $T=20$, $a=0.9$, $b=0.5$, $c=0.05$, $q=q_T=1$, $r=0.1$, $\gamma=\gamma_T=5$, $\tau=0.2$ and $\sigma=0.1$. Initially $X_0=1+V$, with $V\sim\operatorname{Laplace}(0,1/(2\sqrt2))$, so the population law is not Gaussian, and $\theta=0$. The optimal policy is given by the Riccati equations of the mean-field linear--quadratic problem. Transport uses $K=1$.

\subsubsection{Mean--variance portfolio}
\label{app:benchmark-portfolio}

The wealth dynamics and Gaussian policy are
\begin{equation}
    X_{t+1}=X_t+\alpha_tR_{t+1},\qquad
    \alpha_t=k_t(X_t-\bar x(m_t))+\ell_t+\tau\varepsilon_t,
\end{equation}
with $r(x,a,m)=-\gamma\bar x(m)^2$ and $g(x,m)=x-\chi(x-\bar x(m))^2$. Returns $R_{t+1}\sim\Nc(0.02,0.08^2)$ are independent, and independent of the exploration noises. We use $T=10$, $X_0\sim\Nc(1,0.04)$, $\chi=10$, $\gamma=2$ and $\tau=0.2$. The optimal feedback gain is $k_t^\star=-\bar r/(\bar r^2+\sigma_R^2)$, with $\bar r=0.02$ and $\sigma_R=0.08$, and the optimal $(\ell_t^\star)$ solves a $T$-dimensional concave quadratic program. Transport uses $K=1$.

\subsubsection{Bimodal population allocation}
\label{app:benchmark-bimodal}

This benchmark tests the effect of the number of Gaussian components. The population has two modes, and its mean and variance do not depend on the policy parameter. A single Gaussian therefore loses the information needed to optimize the objective.

\paragraph{Model and objective.}
Take $T=1$, $\Xc=\Ac=\R$, $\Theta=[0.75,0.95]$, and $X_0 \sim \frac{1}{2}\Nc(-1,0.01) + \frac{1}{2}\Nc(1,0.01)$. The policy and transition kernel are
\begin{align}
    \pi^\theta(\cdot\mid x,m)=\Nc\big(\theta x,0.99-1.01\theta^2\big), \qquad
    P(\cdot\mid x,m,a)=\Nc(a,0.01).
\end{align}
The policy variance is positive throughout $\Theta$. The terminal law is
\begin{align}
    \mu_1^\theta = \frac12\Nc(-\theta, 1-\theta^2) + \frac12\Nc(\theta,1-\theta^2).
\end{align}
Every such law is bimodal, with mean zero and variance one.

Set $r_0 \equiv 0$ and define the terminal reward $g:\R\times\Pc_2(\R) \to \R$ by
\begin{align}
    g(x,m) = -100 \bigg(\int_\R \exp\bigg[-\frac{y^2}{2(0.2)^2}\bigg] m(\d y) -0.1\bigg)^2.
\end{align}
The integral measures how concentrated the population is near zero, which cannot be determined from its mean and variance alone. Since the reward depends only on the population law, the expected REINFORCE gradient is zero.

\paragraph{Exact validation.}
Integrating the terminal reward against the Gaussian mixture gives
\begin{align}
    J(\theta) = -100\bigg(\frac{0.2}{\sqrt{1.04 - \theta^2}} \exp\bigg[-\frac{\theta^2}{2(1.04 - \theta^2)} \bigg] - 0.1\bigg)^2.
\end{align}
The unique maximizer is $\theta^\star\simeq0.868437$, with $J(\theta^\star)=0$. We initialize at $\theta=0.8$. Table~\ref{tab:decomposition} reports the exact gap $-J(\theta)$; Figure~\ref{fig:main-benchmarks} and Tables~\ref{tab:main-summary} and~\ref{tab:bimodal-components} use the Monte Carlo validation of Appendix~\ref{app:evaluation}.

\paragraph{Effect of the mixture size.}
We compare $K=1$ and $K=2$ using maximum-likelihood mixture fitting. At the population level, the fitted single Gaussian matches the mean and variance of $\mu_1^\theta$. Since these are zero and one for every $\theta$, this fit is always $\Nc(0,1)$. Evaluating the reward on this fitted distribution gives
\begin{align}
    J_1(\theta)
    = -100\left( \int_\R e^{-y^2/(2(0.2)^2)}\,\Nc(0,1)(\d y)-0.1 \right)^2
    = -100\left(\frac{0.2}{\sqrt{1.04}}-0.1\right)^2
    \simeq -0.923831.
\end{align}
Thus $J_1$ is constant: changing $\theta$ changes the actual population, but not its single-Gaussian approximation. The $K=1$ projected objective therefore provides no gradient for learning. With $K=2$, the terminal law is represented exactly, so $J_2=J$.

\subsection{Additional results}
\label{app:additional-results}

\subsubsection{Learning curves and population dynamics}
\label{app:additional-benchmarks}

The population trajectories in Figure~\ref{fig:main-benchmarks} identify the effects missed by REINFORCE. In distribution planning, the reward charges for movement while the benefit depends on the resulting population distribution. REINFORCE consequently remains close to the initial law. Transport learns a terminal law closer to the target, whereas MF-REINFORCE concentrates too much mass near its center. In the portfolio, REINFORCE increases mean wealth despite the running penalty on its square. The population correction substantially reduces this excess.

The advantage is smaller on the other finite-state benchmarks (Figure~\ref{fig:appendix-benchmarks}). Cybersecurity gives similar final objectives for Transport and REINFORCE. On advertising, both population-corrected methods improve on REINFORCE, which learns to avoid advertising. However, neither recovers the optimal policy of advertising until the customer share reaches $0.9$. MF-REINFORCE attains the higher mean objective, with substantial variation across seeds.

\begin{figure}
    \centering
    \includegraphics[width=0.7\textwidth]{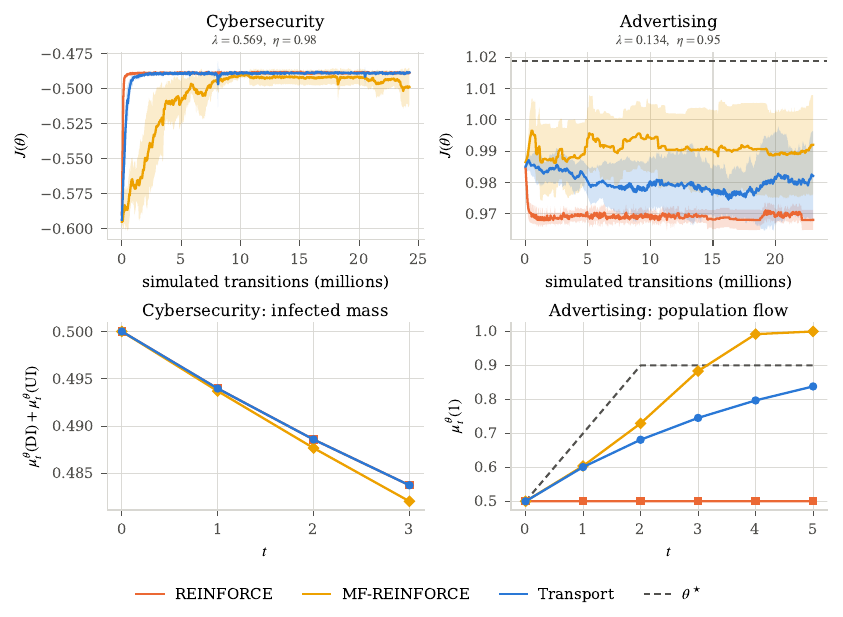}
    \caption{Cybersecurity and advertising. Top: validation objective against simulated transitions (five paired seeds). Bottom: infected mass and customer share under the final policy of the first seed. Cybersecurity has no known optimum.}
    \label{fig:appendix-benchmarks}
\end{figure}

\begin{table}[h]
\centering
\footnotesize
\setlength{\tabcolsep}{4pt}
\caption{Final validation objective on the two benchmarks absent from Table~\ref{tab:main-summary} (mean $\pm$ standard deviation over five seeds; larger is better). Transport uses the scale with the highest mean final objective. Reference values are computed from the benchmark models; gap evaluation is described in Appendix~\ref{app:evaluation}.}
\label{tab:objective-summary}
\begin{tabular}{lrrrrrl}
\toprule
Benchmark & $T$ & Reference & REINFORCE & MF-REINFORCE & Transport & $(\lambda,\eta)$ \\
\midrule
Cybersecurity & 3 & --- & $-0.4887\pm0.0007$ & $-0.4986\pm0.0142$ & $-0.4883\pm0.0010$ & $(0.568666,0.98)$ \\
Advertising & 5 & $1.0188$ & $0.9681\pm0.0036$ & $0.9921\pm0.0176$ & $0.9821\pm0.0159$ & $(0.133802,0.95)$ \\
\bottomrule
\end{tabular}
\end{table}

\subsubsection{Perturbation scale and mixture size}
\label{app:scale-results}

The scale sweeps in Figure~\ref{fig:lambda-tradeoff} and Tables~\ref{tab:bimodal-components}--\ref{tab:decomposition} show two competing effects. Larger perturbations can make the perturbed objective easier to optimize, while changing which policy it favors. Advertising gives a clear example: at the largest tested scale, the perturbed optimization error is below $10^{-5}$, but the original gap remains $0.05$. Further optimization of that same objective would retain this loss.

Bimodal allocation separates the influence of $K$ from that of $\lambda$. With $K=1$, every population law is fitted by $\Nc(0,1)$, explaining the nearly identical outcomes across scales. With $K=2$, the fitted law retains the policy dependence, and optimization of the perturbed objective is accurate to $10^{-4}$. The increasing gap with $\lambda$ then follows the displacement of the learned parameter from $\theta^\star$.

Table~\ref{tab:decomposition} reports the change in the maximum objective value, the error in optimizing the perturbed objective, and the final gap under $J$. Here $\theta_\lambda^\star$ is the perturbed reference optimizer in the benchmark's control class; the distribution-planning reference is obtained numerically. The first column includes changes in the level of the objective, so the final gap is the relevant measure of policy quality. These quantities use the population equations, whereas Table~\ref{tab:main-summary} uses the validation procedure in Appendix~\ref{app:evaluation}.

\begin{table}[h]
\centering
\small
\caption{Bimodal allocation: final gap and policy parameter, initialized at $\theta=0.8$ (exact gap $0.1772$). Gaps are Monte Carlo validation estimates. The optimum is $\theta^\star=0.868437$. For $K=1$, the projected objective is constant, $J_1=-0.923831$; for $K=2$, the representation is exact.}
\label{tab:bimodal-components}
\begin{tabular}{lrrrr}
\toprule
Estimator & $K$ & $\lambda$ & $|J(\widehat\theta)-J(\theta^\star)|$ & $\widehat\theta$ \\
\midrule
REINFORCE & --- & --- & $0.1752\pm0.0087$ & $0.8000\pm0.0000$ \\
Transport & 1 & $0.00702927$ & $0.1425\pm0.1239$ & $0.8103\pm0.0339$ \\
Transport & 1 & $0.0140585$ & $0.1425\pm0.1239$ & $0.8103\pm0.0339$ \\
Transport & 1 & $0.0281171$ & $0.1425\pm0.1240$ & $0.8103\pm0.0339$ \\
Transport & 1 & $0.0562341$ & $0.1425\pm0.1241$ & $0.8102\pm0.0339$ \\
Transport & 1 & $0.112468$ & $0.1426\pm0.1245$ & $0.8102\pm0.0340$ \\
Transport & 2 & $0.00702927$ & $0.0020\pm0.0009$ & $0.8737\pm0.0007$ \\
Transport & 2 & $0.0140585$ & $0.0068\pm0.0017$ & $0.8782\pm0.0007$ \\
Transport & 2 & $0.0281171$ & $0.0244\pm0.0030$ & $0.8865\pm0.0007$ \\
Transport & 2 & $0.0562341$ & $0.0903\pm0.0053$ & $0.9015\pm0.0007$ \\
Transport & 2 & $0.112468$ & $0.3563\pm0.0097$ & $0.9298\pm0.0006$ \\
\bottomrule
\end{tabular}
\end{table}

{\small
\setlength{\tabcolsep}{3.5pt}
\setlength{\LTcapwidth}{\textwidth}
\begin{longtable}{llrrrr}
\caption{Effect of perturbation on objective values and optimization. \emph{Value change}: $|J^\lambda(\theta_\lambda^\star)-J(\theta^\star)|$; \emph{Opt.\ error}: $|J^\lambda(\widehat\theta_\lambda)-J^\lambda(\theta_\lambda^\star)|$; \emph{Final gap}: $|J(\widehat\theta_\lambda)-J(\theta^\star)|$. Objectives are evaluated from the population equations; dispersion is over five seeds. MF-REINFORCE uses its own perturbed objective and scale $\varepsilon$.}
\label{tab:decomposition}\\
\toprule
Benchmark & Estimator & Scale & Value change & Opt.\ error & Final gap \\
\midrule
\endfirsthead
\multicolumn{6}{l}{\tablename~\thetable\ (continued)}\\
\toprule
Benchmark & Estimator & Scale & Value change & Opt.\ error & Final gap \\
\midrule
\endhead
\midrule
\multicolumn{6}{r}{Continued on the next page}\\
\endfoot
\bottomrule
\endlastfoot
Two-state & Transport & $0.0315$ & $0.1191$ & $0.3730\pm0.0617$ & $0.5452\pm0.1134$ \\
 &  & $0.063$ & $0.2394$ & $0.1032\pm0.0346$ & $0.4427\pm0.0895$ \\
 &  & $0.126$ & $0.4848$ & $0.0153\pm0.0085$ & $0.5968\pm0.0782$ \\
 &  & $0.252$ & $0.9970$ & $0.0044\pm0.0027$ & $1.4465\pm0.0889$ \\
 &  & $0.504$ & $2.1475$ & $0.0074\pm0.0028$ & $3.6452\pm0.0809$ \\
\addlinespace[2pt]
 & MF-REINFORCE & $0.2$ & $2.6476$ & $0.1892\pm0.0445$ & $0.7098\pm0.0915$ \\
\addlinespace[2pt]
 & REINFORCE & --- & --- & --- & $29.0738\pm0.0000$ \\
\midrule
Distribution & Transport & $0.0306$ & $0.0009$ & $0.0371\pm0.0010$ & $0.0369\pm0.0008$ \\
 &  & $0.0611$ & $0.0021$ & $0.0320\pm0.0009$ & $0.0315\pm0.0006$ \\
 &  & $0.122$ & $0.0054$ & $0.0252\pm0.0013$ & $0.0255\pm0.0010$ \\
 &  & $0.244$ & $0.0159$ & $0.0177\pm0.0014$ & $0.0298\pm0.0011$ \\
 &  & $0.489$ & $0.0603$ & $0.0173\pm0.0040$ & $0.0477\pm0.0040$ \\
\addlinespace[2pt]
 & MF-REINFORCE & $2$ & $1.7877$ & $0.0790\pm0.0099$ & $0.1178\pm0.0243$ \\
\addlinespace[2pt]
 & REINFORCE & --- & --- & --- & $0.2046\pm0.0000$ \\
\midrule
Advertising & Transport & $0.0335$ & $0.0044$ & $0.0421\pm0.0141$ & $0.0451\pm0.0150$ \\
 &  & $0.0669$ & $0.0088$ & $0.0338\pm0.0153$ & $0.0392\pm0.0176$ \\
 &  & $0.134$ & $0.0176$ & $0.0256\pm0.0109$ & $0.0379\pm0.0162$ \\
 &  & $0.268$ & $0.0353$ & $0.0165\pm0.0001$ & $0.0495\pm0.0007$ \\
 &  & $0.535$ & $0.0537$ & $7.6\cdot10^{-6}\pm1.3\cdot10^{-5}$ & $0.0500\pm0.0000$ \\
\addlinespace[2pt]
 & MF-REINFORCE & $1$ & $0.0331$ & $0.0196\pm0.0160$ & $0.0274\pm0.0172$ \\
\addlinespace[2pt]
 & REINFORCE & --- & --- & --- & $0.0500\pm0.0000$ \\
\midrule
Linear--quadratic & Transport & $0.0385$ & $0.2282$ & $0.0344\pm0.0024$ & $0.0370\pm0.0026$ \\
 &  & $0.077$ & $0.1299$ & $0.0160\pm0.0010$ & $0.0196\pm0.0012$ \\
 &  & $0.154$ & $1.0465$ & $0.0166\pm0.0010$ & $0.0214\pm0.0019$ \\
 &  & $0.308$ & $7.3214$ & $0.0402\pm0.0010$ & $0.0389\pm0.0026$ \\
 &  & $0.616$ & $35.5448$ & $0.1568\pm0.0077$ & $0.1389\pm0.0309$ \\
\addlinespace[2pt]
 & Transport-Proba & $0.0385$ & $0.2282$ & $0.1396\pm0.0206$ & $0.1435\pm0.0207$ \\
 &  & $0.077$ & $0.1299$ & $0.0434\pm0.0038$ & $0.0471\pm0.0042$ \\
 &  & $0.154$ & $1.0465$ & $0.0431\pm0.0029$ & $0.0469\pm0.0029$ \\
 &  & $0.308$ & $7.3214$ & $0.0909\pm0.0059$ & $0.0868\pm0.0029$ \\
 &  & $0.616$ & $35.5448$ & $0.2294\pm0.0305$ & $0.2644\pm0.0697$ \\
\addlinespace[2pt]
 & REINFORCE & --- & --- & --- & $0.2163\pm0.0009$ \\
\midrule
Portfolio & Transport & $0.0328$ & $0.4970$ & $0.6392\pm0.1421$ & $0.7551\pm0.2237$ \\
 &  & $0.0656$ & $0.9687$ & $0.4231\pm0.0796$ & $0.6003\pm0.1769$ \\
 &  & $0.131$ & $1.7597$ & $0.2543\pm0.0556$ & $0.7145\pm0.1265$ \\
 &  & $0.262$ & $2.4207$ & $0.1009\pm0.0329$ & $1.2568\pm0.0613$ \\
 &  & $0.525$ & $0.5683$ & $0.1134\pm0.0654$ & $2.4766\pm0.1431$ \\
\addlinespace[2pt]
 & Transport-Proba & $0.0328$ & $0.4970$ & $1.6853\pm0.4038$ & $1.8498\pm0.3817$ \\
 &  & $0.0656$ & $0.9687$ & $1.3225\pm0.4507$ & $1.3289\pm0.4980$ \\
 &  & $0.131$ & $1.7597$ & $0.7035\pm0.1708$ & $0.7190\pm0.1738$ \\
 &  & $0.262$ & $2.4207$ & $0.5543\pm0.0648$ & $1.2154\pm0.1074$ \\
 &  & $0.525$ & $0.5683$ & $0.1902\pm0.0333$ & $2.8870\pm0.4371$ \\
\addlinespace[2pt]
 & REINFORCE & --- & --- & --- & $6.8874\pm0.0036$ \\
\midrule
Bimodal allocation & Transport ($K=2$) & $0.00703$ & $0.0010$ & $2.8\cdot10^{-5}\pm2.6\cdot10^{-5}$ & $0.0019\pm0.0005$ \\
 &  & $0.0141$ & $0.0031$ & $2.6\cdot10^{-5}\pm2.1\cdot10^{-5}$ & $0.0066\pm0.0010$ \\
 &  & $0.0281$ & $0.0098$ & $2.6\cdot10^{-5}\pm1.9\cdot10^{-5}$ & $0.0239\pm0.0020$ \\
 &  & $0.0562$ & $0.0303$ & $3.3\cdot10^{-5}\pm2.8\cdot10^{-5}$ & $0.0888\pm0.0041$ \\
 &  & $0.112$ & $0.0988$ & $6.2\cdot10^{-5}\pm6.9\cdot10^{-5}$ & $0.3517\pm0.0078$ \\
\addlinespace[2pt]
 & REINFORCE & --- & --- & --- & $0.1772\pm0.0000$ \\
\end{longtable}
}

\subsubsection{Sensitivity estimator and runtime}
\label{experiments:transport-proba-comparison}

Transport-Proba replaces centered differences of fitted coordinates with a likelihood-ratio estimate of moment sensitivities. On linear--quadratic control, centered differences give a smaller final gap at every tested scale (Table~\ref{tab:decomposition}) and a shorter runtime. The portfolio comparison is closer relative to the variation across seeds. Both estimators use $K=1$ here: the Gaussian fit retains the relevant population moments even when the true state law is non-Gaussian.

Table~\ref{tab:budget-runtime} quantifies the computational cost beyond simulation. Transport is slower than REINFORCE because each update also estimates sensitivities and evaluates population scores. The overhead is particularly large on distribution planning and bimodal allocation. Thus the improvement at matched simulator budgets comes with an additional arithmetic cost; it does not imply the same gain per second.

\begin{table}
\centering
\small
\caption{Simulator transitions per update and wall-clock time per complete training run. Each benchmark uses the same simulator budget across methods.}
\label{tab:budget-runtime}
\begin{tabular}{llrrr}
\toprule
Benchmark & Estimator & Transitions/update & Seconds/run & Ratio \\
\midrule
Two-state & REINFORCE & $2\,300$ & $70$ & $1.0\times$ \\
 & MF-REINFORCE & $2\,300$ & $387$ & $5.5\times$ \\
 & Transport & $2\,300$ & $220$ & $3.1\times$ \\
\midrule
Cybersecurity & REINFORCE & $1\,212$ & $167$ & $1.0\times$ \\
 & MF-REINFORCE & $1\,212$ & $595$ & $3.6\times$ \\
 & Transport & $1\,212$ & $393$ & $2.4\times$ \\
\midrule
Distribution & REINFORCE & $5\,300$ & $169$ & $1.0\times$ \\
 & MF-REINFORCE & $5\,300$ & $2\,124$ & $12.5\times$ \\
 & Transport & $5\,300$ & $2\,706$ & $16.0\times$ \\
\midrule
Advertising & REINFORCE & $2\,300$ & $47$ & $1.0\times$ \\
 & MF-REINFORCE & $2\,300$ & $298$ & $6.3\times$ \\
 & Transport & $2\,300$ & $168$ & $3.6\times$ \\
\midrule
Linear--quadratic & REINFORCE & $210\,020$ & $349$ & $1.0\times$ \\
 & Transport & $210\,020$ & $719$ & $2.1\times$ \\
 & Transport-Proba & $210\,020$ & $2\,577$ & $7.4\times$ \\
\midrule
Portfolio & REINFORCE & $54\,310$ & $799$ & $1.0\times$ \\
 & Transport & $54\,310$ & $1\,796$ & $2.2\times$ \\
 & Transport-Proba & $54\,310$ & $1\,888$ & $2.4\times$ \\
\midrule
Bimodal allocation & REINFORCE & $110\,000$ & $19$ & $1.0\times$ \\
 & Transport & $110\,000$ & $192$ & $10.0\times$ \\
\bottomrule
\end{tabular}
\end{table}

\subsubsection{Auxiliary scale and reuse of trajectories}
\label{app:discrete-eta}

The auxiliary-scale sweep uses exact population flows, so its sensitivity error comes entirely from the auxiliary estimator. At small $\eta$, the factor $(1-\eta)/\eta$ amplifies noise in the earlier sensitivities used by the recursion. The resulting error can grow substantially across time steps, explaining the large finite-budget errors near $\eta=n^{-1/4}$ in Table~\ref{tab:discrete-eta}. On two-state control, the dynamics and policy are independent of the population argument, and larger radii mainly reduce this variance amplification.

Advertising shows the cost of stronger perturbations: sensitivity error is minimized at an intermediate scale, while gradient error is flatter over the larger scales. The gradient weights the sensitivities by returns and population scores, and also contains online sampling noise. Its MSE is therefore the relevant criterion for selecting $\eta$, as in the protocol.

Shared and fresh auxiliary batches give comparable errors in this sweep. Reusing a batch reduces the auxiliary simulation cost from $nT(T+1)/2$ to $nT$, but introduces dependence between successive sensitivity estimates. The theorem analyzes fresh batches; Table~\ref{tab:discrete-eta} assesses the shared-batch implementation empirically. In training, two-state control at $\lambda_\star$ reaches final gaps of $1.93\pm0.15$ at $\eta=\eta_\star/2$, and between $1.34\pm0.13$ and $1.36\pm0.13$ at $\eta\in\{\eta_\star,0.85,0.95,0.98\}$ (five seeds): only the smallest scale degrades the final performance.

\begin{table}
\centering
\small
\caption{Relative sensitivity MSE $\sum_t\E\big[\|\widehat D_t-D_t\|_F^2\big]/\sum_t\|D_t\|_F^2$ (200 replications) and gradient MSE $\E\big[\|\widehat G-\nabla_\theta J\|^2\big]/\|\nabla_\theta J\|^2$ (100 replications). The population flow is exact; the policy is fixed at initialization and $\lambda=B^{-1/4}$. The first scales are $\eta_\star/2$, $\eta_\star$ and, when below one, $2\eta_\star$, with $\eta_\star=n^{-1/4}$.}
\label{tab:discrete-eta}
\begin{tabular}{lrrrrr}
\toprule
Benchmark & $n$ & $\eta$ & \multicolumn{2}{c}{Sensitivity error} & Gradient error \\
 & & & shared batch & fresh batches & shared batch \\
\midrule
Two-state & 12 & $0.269$ & $4.33\cdot10^{4}$ & $3.93\cdot10^{4}$ & $8.46\cdot10^{3}$ \\
 &  & $0.537$ & $31.6$ & $32.2$ & $15.4$ \\
 &  & $0.85$ & $1.26$ & $1.31$ & $2.91$ \\
 &  & $0.95$ & $1.05$ & $1.07$ & $2.75$ \\
 &  & $0.98$ & $1.04$ & $1.06$ & $2.75$ \\
\midrule
Cybersecurity & 51 & $0.187$ & $1.13\cdot10^{4}$ & $1.76\cdot10^{4}$ & $1.39\cdot10^{6}$ \\
 &  & $0.374$ & $312$ & $448$ & $4.17\cdot10^{4}$ \\
 &  & $0.748$ & $8.6$ & $9.14$ & $1.44\cdot10^{3}$ \\
 &  & $0.85$ & $6.29$ & $6.67$ & $1.00\cdot10^{3}$ \\
 &  & $0.95$ & $5.61$ & $5.97$ & $836$ \\
 &  & $0.98$ & $5.58$ & $5.92$ & $812$ \\
\midrule
Distribution & 280 & $0.122$ & $2.56\cdot10^{10}$ & $3.09\cdot10^{10}$ & $2.85\cdot10^{10}$ \\
 &  & $0.244$ & $3.04\cdot10^{7}$ & $3.75\cdot10^{7}$ & $3.60\cdot10^{7}$ \\
 &  & $0.489$ & $6.45\cdot10^{3}$ & $8.04\cdot10^{3}$ & $9.71\cdot10^{3}$ \\
 &  & $0.85$ & $0.597$ & $0.603$ & $1.37$ \\
 &  & $0.95$ & $0.178$ & $0.177$ & $0.487$ \\
 &  & $0.98$ & $0.155$ & $0.154$ & $0.443$ \\
\midrule
Advertising & 65 & $0.176$ & $20.2$ & $12.2$ & $1.43\cdot10^{4}$ \\
 &  & $0.352$ & $1.57$ & $1.15$ & $1.51\cdot10^{3}$ \\
 &  & $0.704$ & $0.804$ & $0.72$ & $496$ \\
 &  & $0.85$ & $0.881$ & $0.793$ & $449$ \\
 &  & $0.95$ & $0.913$ & $0.866$ & $465$ \\
 &  & $0.98$ & $0.922$ & $0.892$ & $468$ \\
\bottomrule
\end{tabular}
\end{table}

\subsubsection{Perturbation of the law and gradient}
\label{app:theory-verification}

At a fixed policy, the finite-state total-variation distance is proportional to $\lambda$ by construction. For a single Gaussian, the $\Wc_2$ distance also scales approximately linearly, through changes in the mean and standard deviation. The two-component case additionally changes mixture weights and redistributes mass between modes, producing different behavior over the tested scales. The gradient perturbation is approximately linear on the benchmarks where it is evaluated (Figure~\ref{fig:theory-verification}).

The policy is fixed after $200$ updates. Distances use $2000$ randomizer draws in finite states and $200$ per time step in continuous states; the latter use mixture quantiles to compute root-mean-square $\Wc_2$.

\begin{figure}
    \centering
    \includegraphics[width=0.7\textwidth]{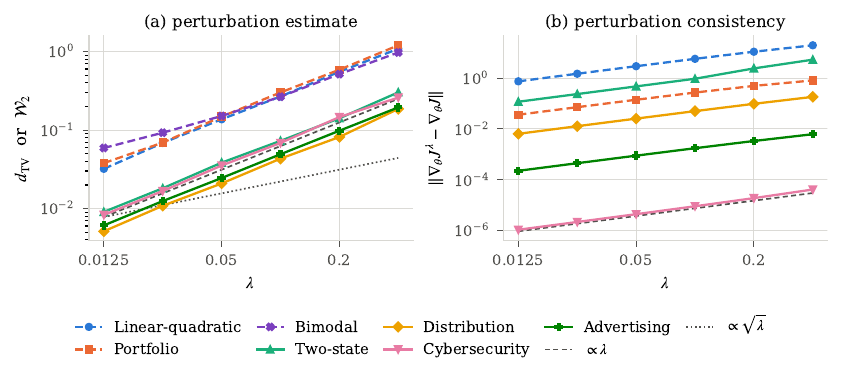}
    \caption{Perturbation at a fixed policy on all benchmarks. Left: total variation (finite states, solid) and $\Wc_2$ (continuous states, dashed). Right: $\|\nabla_\theta J^\lambda-\nabla_\theta J\|$. Grey reference lines have slopes $1$ and $1/2$.}
    \label{fig:theory-verification}
\end{figure}

\subsubsection{Bias and mean-square error}
\label{experiments:bias-mse}

We vary one quantity at a time to distinguish sampling error from perturbation error. The inverse-sample-size slopes in $M$ and $B$ are consistent with Monte Carlo averaging. Sensitivity error shows the same behavior in $n$ until systematic error becomes dominant. On advertising, the error only decreases from $0.47$ to $0.38$ between $n=130$ and $n=260$, giving a fitted slope of $-0.51$. This is consistent with the error remaining in the perturbed sensitivity at fixed $\eta$.

Figure~\ref{fig:bounds-radius} isolates the continuous-state sensitivity calculation. With exact flows, centered differences have squared error of order $\eta^4$. With fitted particle flows, dividing the endpoint fitting errors by $2\eta$ increases their contribution as $\eta$ decreases. The observed curves reflect the sum of these truncation and fitting errors, as in~\eqref{eq:continuous-fitting-sensitivity-error}. At the benchmark budgets, the fitting term dominates at every tested $\eta$: the total error decreases as $\eta^{-2}$ up to $\eta=0.95$ and remains well above the truncation error.

The continuous-state bias sweep uses exact flows and sensitivities, $B=4096$ and $1024$ replications. The bias is the norm of an average over these replications, so it is only known up to its Monte Carlo standard error. The slope in $\lambda$ is therefore fitted in two steps. First, we discard every scale at which the estimated bias is smaller than twice its standard error, since such a value cannot be distinguished from noise; this removes $\lambda=0.025$ on the portfolio. Second, among the remaining scales, we keep the smaller half, because at large $\lambda$ the bias can become comparable to the gradient itself and stop growing linearly. The fit thus uses $\lambda\in\{0.025,0.05,0.1\}$ on linear--quadratic control and $\lambda\in\{0.05,0.1,0.2\}$ on the portfolio. Finite-state bias is computed exactly, so only the second step applies, and the fit uses the three smallest scales, $\lambda\in\{0.025,0.05,0.1\}$. Saturation is visible on two-state control, where the bias is $0.32$, $0.87$ and $1.92$ at these scales and then stays below the gradient norm $2.06$. The continuous-state policies are $\theta=\theta^\star/2$, and the finite-state policies are fixed at initialization. Full fitted slopes are reported in Table~\ref{tab:bounds-exponents}.

\begin{figure}
    \centering
    \includegraphics[width=0.7\textwidth]{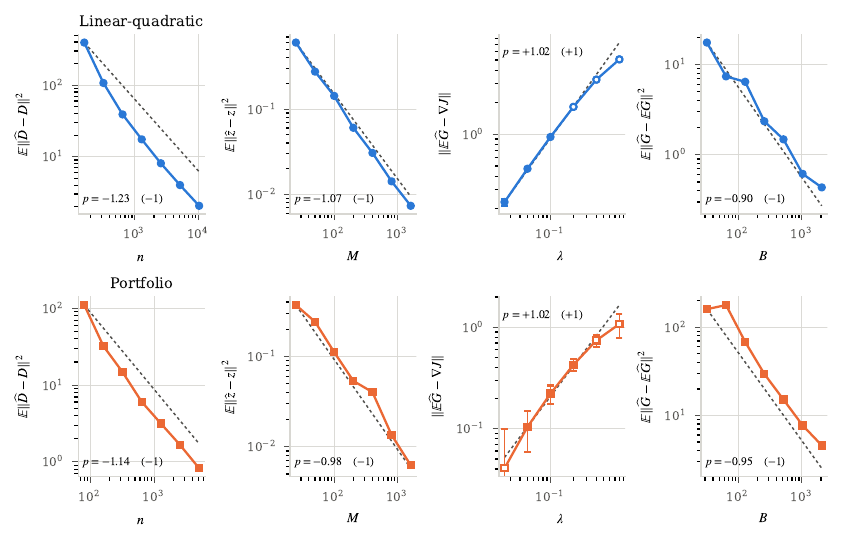}
    \caption{Continuous-state estimator errors as one quantity is varied. Dashed lines show predicted powers. Bias error bars show the Monte Carlo standard error; hollow markers are the scales excluded from the slope fit.}
    \label{fig:bounds-rates}
\end{figure}

\begin{figure}
    \centering
    \includegraphics[width=0.7\textwidth]{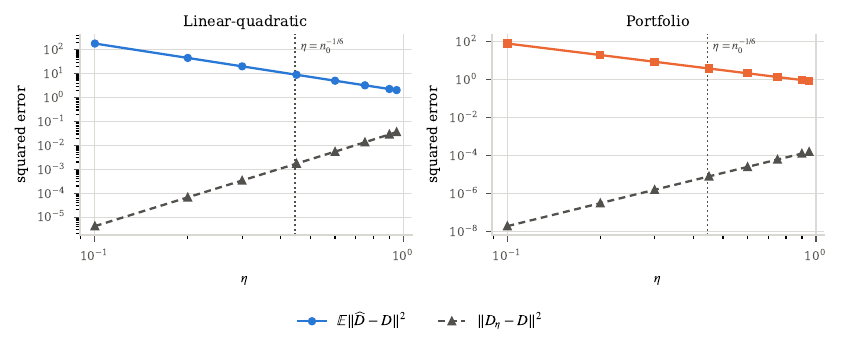}
    \caption{Continuous-state sensitivity error against $\eta$, using fitted particle flows and exact population flows. Exact flows isolate truncation error. The vertical line marks $\eta=n_0^{-1/6}$, with $n_0=n/(2d_\theta)=128$.}
    \label{fig:bounds-radius}
\end{figure}

\begin{table}
\centering
\footnotesize
\setlength{\tabcolsep}{3pt}
\caption{Log-log slopes for all benchmarks. Each row varies one quantity, holding the remaining quantities at the benchmark allocation, except that the continuous-state $\lambda$ sweep uses $B=4096$. The $\lambda$ slopes are fitted on the three smallest scales at which the bias is resolved, as described above.}
\label{tab:bounds-exponents}
\begin{tabular}{llrrrrrrr}
\toprule
Quantity & Swept & LQ & Portfolio & Two-state & Cybersecurity & Distribution & Advertising & $p_0$ \\
\midrule
$\mathbb{E}\|\widehat D-D\|^2$ & $\eta$ (exact flows) & $+4.02$ & $+3.99$ & --- & --- & --- & --- & $+4$ \\
$\mathbb{E}\|\widehat D-D\|^2$ & $n$ & $-1.23$ & $-1.14$ & $-1.06$ & $-1.00$ & $-1.02$ & $-0.51$ & $-1$ \\
$\|\mathbb{E}\widehat G-\nabla J\|$ & $\lambda$ & $+1.02$ & $+1.02$ & $+1.28$ & $+1.00$ & $+1.00$ & $+1.00$ & $+1$ \\
$\E\|\widehat G-\E\widehat G\|^2$ & $B$ & $-0.90$ & $-0.95$ & $-0.99$ & $-1.01$ & $-1.03$ & $-0.92$ & $-1$ \\
$\mathbb{E}\|\widehat z-z\|^2$ & $M$ & $-1.07$ & $-0.98$ & $-1.02$ & $-1.02$ & $-1.00$ & $-1.06$ & $-1$ \\
\bottomrule
\end{tabular}
\end{table}

\end{document}